\documentclass[12pt]{amsart}
\usepackage[margin=1in]{geometry}                % See geometry.pdf to learn the layout options. There are lots.
\usepackage{graphicx}
\usepackage{amssymb}
\usepackage{amsthm}
\usepackage{mathrsfs}
\usepackage{caption}
\usepackage{xcolor}

\newtheorem{theorem}{Theorem}[section]
\newtheorem{proposition}[theorem]{Proposition}
\newtheorem{lemma}[theorem]{Lemma}
\newtheorem{definition}[theorem]{Definition}
\newtheorem{corollary}[theorem]{Corollary}

\newtheorem{remark}{Remark}[section]

\numberwithin{equation}{section}

\newcommand{\ZZ}{\mathbb{Z}}

\newcommand{\NN}{\mathbb{N}}

\newcommand{\bbA}{\mathbb{A}}
\newcommand{\bbB}{\mathbb{B}}

\newcommand{\bbI}{\mathbb{I}}

\newcommand{\bbJ}{\mathbb{J}}
\newcommand{\bbK}{\mathbb{K}}

\newcommand{\cali}{\mathcal{I}}
\newcommand{\calj}{\mathcal{J}}
\newcommand{\calk}{\mathcal{K}}

\newcommand{\calm}{\mathcal{M}}

\newcommand{\Div}{{\mathrm{Div}}}
\newcommand{\Span}{{\mathrm{Span}}}
\newcommand{\Bispan}{{\mathrm{Bispan}}}

\subjclass[2020]{Primary 05B45, 11B75, 20K01. Secondary 11C08, 43A47, 51D20, 52C22.}

\title[Even Coven-Meyerowitz]{The Coven-Meyerowitz tiling conditions for 3 prime factors: the even case} 
\author{Izabella {\L}aba and Itay Londner}

\date{\today}

\begin{document}

\begin{abstract} We consider finite sets $A\subset\ZZ$ that the integers by translations. By periodicity, any such tiling is equivalent to a factorization $A\oplus B=\ZZ_M$ of a finite cyclic group. Building on the work in \cite{LaLo1, LaLo2, LaLo3}, we prove that a tentative characterization of finite tiles proposed by Coven and Meyerowitz \cite{CM}
holds for all integer tilings of period $M=(p_ip_jp_k)^2$, where $p_i,p_j,p_k$ are distinct primes.  This extends the main result of \cite{LaLo2}, where we assumed that $M$ is odd. We also improve parts of the argument from \cite{LaLo2}.

\end{abstract}

\maketitle

\setcounter{tocdepth}{1}
\tableofcontents

%%%%%%%%%%%%%%%%%%%%%%%%%%%%%%%%%%%%%%

\section{Introduction}

%%%%%%%%%%%%%%%%%%%%%%%%%%%%%%%%%%%%%%

A finite set $A\subset\ZZ$ {\em tiles the integers by translations} 
if there is a (necessarily infinite) set $T\subset\ZZ$ such that every $n\in\ZZ$ can be represented uniquely
as $n=a+t$ with $a\in A$ and $t\in T$. We will call such sets {\em finite integer tiles}. Newman
\cite{New} proved
that any tiling of $\ZZ$ by a finite set $A$ must be periodic: there exists a $M\in\NN$
such that $T=B\oplus M\ZZ$ for some finite set $B\subset \ZZ$. We then have $|A|\,|B|=M$, and
$A\oplus B$ modulo $M$ is a factorization of the cyclic group $\ZZ_M$. We will write this as $A\oplus B=\ZZ_M$.

The main open problem concerning integer tilings is to determine whether the Coven-Meyerowitz tiling conditions hold for all finite tiles. 
To state these conditions, we need some notation. By translational invariance, we may assume that $A,B\subset\{0,1,\dots\}$ and that $0\in A\cap B$. The {\it  mask polynomials} of $A$ and $B$ are
$$
A(X)=\sum_{a\in A}X^a,\ B(x)=\sum_{b\in B}X^b .
$$
The tiling condition $A\oplus B=\ZZ_M$ is then equivalent to
\begin{equation}\label{poly-e1}
 A(X)B(X)=1+X+\dots+X^{M-1}\ \mod (X^M-1).
\end{equation}
Let $\Phi_s(X)$ be the $s$-th cyclotomic polynomial, i.e., the unique monic, irreducible polynomial whose roots are the primitive $s$-th roots of unity. %, $\Phi_s$ allow the inductive definition
%\begin{equation}\label{poly-e0}
%X^n-1=\prod_{s|n}\Phi_s(X).
%\end{equation}
Factorizing the polynomial on the right hand side of (\ref{poly-e1}) to cyclotomic factors, tiling may be equivalently restated as
\begin{equation}\label{poly-e2}
  |A||B|=M\hbox{ and }\Phi_s(X)\ |\ A(X)B(X)\hbox{ for all }s|M,\ s\neq 1.
\end{equation}
Since $\Phi_s$ are irreducible, each $\Phi_s(X)$ with $s|M$ must divide at least one of $A(X)$ and $B(X)$.

Let $S_A$ be the set of prime powers
$p^\alpha$ such that $\Phi_{p^\alpha}(X)$ divides $A(X)$.  Then the Coven-Meyerowitz tiling conditions 
are as follows:

\smallskip
{\it (T1) $A(1)=\prod_{s\in S_A}\Phi_s(1)$,}

\smallskip
{\it (T2) if $s_1,\dots,s_k\in S_A$ are powers of different
primes, then $\Phi_{s_1\dots s_k}(X)$ divides $A(X)$.}
\medskip

It is proved in \cite{CM} that:

\begin{itemize}

\item if $A$ satisfies (T1), (T2), then $A$ tiles $\ZZ$;

\item  if $A$ tiles $\ZZ$ then (T1) holds;

\item if $A$ tiles $\ZZ$ and $|A|$ has at most two distinct prime factors,
then (T2) holds.
\end{itemize}

The conjecture that (T1) and (T2) hold for all finite integer tiles was popularized in the literature, in recent years, as the Coven-Meyerowitz conjecture.  

The statement (T1) is a counting condition, ensuring that the equation $|A||B|=M$ is compatible with the second part of (\ref{poly-e2}).
The second condition (T2) is a much deeper structural property. For finite sets $A$ satisfying (T1) and (T2), Coven and Meyerowitz constructed an explicit tiling $A\oplus B^\flat=\ZZ_M$, where $M=\hbox{lcm}(S_A)$ and $B^\flat$ is an explicit ``standard" tiling complement (described here in Section \ref{standards}). We proved in \cite{LaLo1} (although this argument was already implicit in \cite{CM}) that having a tiling complement of this type is in fact equivalent to (T2).  This places (T2) in close relation to questions on factor replacement in factorizations of abelian groups \cite{Szabo-book}.

The Coven-Meyerowitz proof can be extended to a limited range of tilings where $M$ has more than two prime factors. In \cite[Corollary 6.2]{LaLo1}, we use the methods of \cite{CM} to prove that if $A\oplus B=\ZZ_M$, and if $|A|$ and $|B|$ have at most two \emph{shared} distinct prime factors, then both $A$ and $B$ satisfy (T2). Similar results have also appeared elsewhere in the literature, see e.g., \cite{Tao-blog}, \cite[Proposition 4.1]{shi}, \cite[Theorem 1.5]{M}.

Unfortunately, the methods of \cite{CM} do not extend to the case when $|A|$ and $|B|$ share three or more distinct prime factors. This was already known to Coven and Meyerowitz, who cited the examples due to Szab\'o \cite{Sz} (see also \cite{LS}). While (T2) still holds for these particular examples, a key element of  the Coven-Meyerowitz proof is Sands's factor replacement
theorem \cite{Sands}, which in general does not hold for tilings with three or more prime factors..
The only cases that had been resolved prior to our work in \cite{LaLo1}, \cite{LaLo2}  are either covered by the methods of \cite{CM} (see \cite{Tao-blog}, \cite{shi}, \cite{M}) or else concern tilings with explicitly given structure (\cite{KL}, \cite{dutkay-kraus}).

In \cite{LaLo2}, we proved the following theorem.

\begin{theorem}\label{T2-3primes-thm} \cite{LaLo2}
Let $M=p_i^2p_j^2p_k^2$, where $p_i,p_j,p_k$ are distinct odd primes. 
Assume that $A\oplus B=\ZZ_M$, with $|A|=|B|=p_ip_jp_k$. Then both $A$ and $B$ satisfy (T2).
\end{theorem}

The goal of this article is to extend Theorem \ref{T2-3primes-thm} to the even case, as follows.

\begin{theorem}\label{T2-3primes-even}
Let $M=p_i^2p_j^2p_k^2$, where $p_i,p_j,p_k$ are distinct primes and $2\in\{p_i,p_j,p_k\}$. 
Assume that $A\oplus B=\ZZ_M$, with $|A|=|B|=p_ip_jp_k$. Then both $A$ and $B$ satisfy (T2).
\end{theorem}

Theorems \ref{T2-3primes-thm} and \ref{T2-3primes-even} cover all tilings of period $M=p_i^2p_j^2p_k^2$, where $p_i,p_j,p_k$ are distinct primes. 
Additionally, the results of this article together with those of \cite{LaLo2} provide a classification of all tilings $A\oplus B=\ZZ_M$, where $M=p_i^2p_j^2p_k^2$. The detailed statements are given in Theorems \ref{unfibered-mainthm} and \ref{fibered-thm}.

We will rely on the methods and partial results of \cite{LaLo1}, \cite{LaLo2}.
While the broad outline of our proof is similar to that in \cite{LaLo2} (see Sections \ref{classified} and \ref{fibered-sec} for more details), the even case is different enough to make the overlap between the actual arguments rather limited. We take advantage of {\it splitting}, a method introduced more recently in \cite{LaLo3}.  Splitting saves the day in several situations where the saturating set arguments of \cite{LaLo2} do not work well enough in the even case. We also reorganize and simplify the proof in the ``fibered grids'' case, the most technical part of \cite{LaLo2}. The argument is still divided into cases, but we define these cases differently and use the splitting and structure results from \cite{LaLo3} to resolve one of them entirely. Additionally, the splitting formulation of the slab reduction \cite{LaLo3} leads to significant simplifications in some of the technical arguments. We include both odd and even cases in this part of the article, since the improvements apply in both cases and the additional length needed to cover the odd case is minimal.

The method of splitting turns out to be powerful enough to provide an almost immediate proof of the Coven-Meyerowitz conjecture for an additional class of tilings \cite{LaLo3}. This includes the following cases:

\begin{itemize}
	\item  $M=p_1^{n_1}p_2^{n_2}p_3^{n_3}$ with $p_1>p_2^{n_2-1}p_3^{n_3-1}$ 
	\item $M=p_1^{n_1}p_2^2p_3^2p_4^2$ with $p_1>p_2p_3p_4$
\end{itemize}
where $p_\nu,\nu\in\{1,2,3,4\},$ are distinct primes and $n_\nu\in \NN, \nu\in\{1,2,3\}$. The latter case is built on the main result of this paper.

While progress on the Coven-Meyerowitz conjecture beyond \cite{CM} and our papers \cite{LaLo1, LaLo2, LaLo3} has been limited, there is a considerable interest in tiling questions, for instance the periodic tiling conjecture \cite{Bh}, \cite{GK}, \cite{GT}, \cite{GT2}. A closely related conjecture of Fuglede \cite{Fug} has also drawn significant interest. We refer the reader to \cite{LaLo3} for further discussion of this connection and the implications of our work in this context.

%%%%%%%%%%%%%%%%%%%%%%%%%%%%%%%%%%%%%

\section{Notation and preliminaries}\label{sec-prelim}

%%%%%%%%%%%%%%%%%%%%%%%%%%%%%%%%%%%%%

This section summarizes the relevant definitions and results of \cite{LaLo1}, specialized to the 3-prime case. All material due to other authors is indicated explicitly as such.
\subsection{Multisets and mask polynomials}

We will assume that $M=p_i^{n_i}p_j^{n_j}p_k^{n_k}$, where $p_i,p_j,p_k$ are distinct primes and $n_i,n_j,n_k\in\NN$. While the indices $\{i,j,k\}$ are a permutation of $\{1,2,3\}$, we will use $i,j,k$ for this purpose, freeing up numerical subscripts for other uses. Though the results of Sections \ref{product-dilations}, \ref{splitting-section}, and \ref{tiling-reductions} are not restricted to the 3-prime setting, we will only state the results in the setup it is used.

The full proof of Theorem \ref{T2-3primes-thm} requires that $n_i=n_j=n_k=2$ and $2\in\{p_i,p_j,p_k\}$. However, many of our intermediate results are valid under weaker assumptions as indicated.

We will always work in either $\ZZ_M$ or in $\ZZ_N$ for some $N|M$. 
We use $A(X)$, $B(X)$, etc. to denote polynomials modulo $X^M-1$ with integer coefficients. 
Each such polynomial $A(X)=\sum_{a\in\ZZ_M} w_A(a) X^a$ is associated with a weighted multiset in $\ZZ_M$, which we will also denote by $A$, with weights $w_A(x)$ assigned to each $x\in\ZZ_M$. (If the coefficient of $X^x$ in $A(X)$ is 0, we set $w_A(x)=0$.) In particular, if $A$ has $\{0,1\}$ coefficients, then
$w_A$ is the characteristic function of a set $A\subset \ZZ_M$. We will use $\calm(\ZZ_M)$ to denote the 
family of all weighted multisets in $\ZZ_M$, and reserve the notation $A\subset \ZZ_M$ for sets.

If $N|M$, then any $A\in \calm(\ZZ_M)$ induces a weighted multiset $A$ mod $N$ in $\ZZ_N$, with the corresponding mask polynomial $A(X)$ mod $(X^N-1)$.
% and induced weights 
%\begin{equation}\label{induced-weights}
%$$
%w_A^N(x)= \sum_{x'\in\ZZ_M: x'\equiv x\,{\rm mod}\, N} w_A(x'),\ \ x\in\ZZ_N.
%$$
%\end{equation}
We will continue to write $A$ and $A(X)$ for $A$ mod $N$ and $A(X)$ mod $X^N-1$, respectively, while working in $\ZZ_N$.

 If $A,B\in\calm(\ZZ_M)$, 
 %we will use $A+B$ to indicate the weighted multiset corresponding to the mask polynomial
 %$(A+B)(X)=A(X)+B(X)$, with the weight function $w_{A+B}(x)=w_A(x)+w_B(x)$.
we use the convolution notation $A*B$ to denote the weighted sumset of $A$ and $B$, so that $(A*B)(X)=A(X)B(X)$.
% and
% $$w_{A*B}(x)=(w_A*w_B)(x)=\sum_{y\in\ZZ_M} w_A(x-y)w_B(y).$$
 If one of the sets is a singleton, say $A=\{x\}$, we will simplify the notation and write $x*B=\{x\}*B$. 
 The direct sum notation $A\oplus B$ is reserved for tilings, i.e., $A\oplus B=\ZZ_M$ means that $A,B\subset\ZZ_M$ are both sets and $A(X)B(X)=\frac{X^M-1}{X-1}$ mod $X^M-1$. 
Notation such as $A'$, $A''$, etc., will be used to denote auxiliary multisets and polynomials rather than derivatives.

%%%%%%%%%%%%%%%%%%%%%%%%%%%%

%%%%%%%%%%%%%%%%%%%%%%%%%%%%%%%%%%%%%

\subsection{Array coordinates and geometric representation}
\label{sec-array}

%%%%%%%%%%%%%%%%%%%%%%%%%%%%%%%%%%%%%

By the Chinese Remainder Theorem, the cyclic group $\ZZ_M$ may be identified with $\ZZ_{p_i^{n_i}}\oplus \ZZ_{p_j^{n_j}}\oplus\ZZ_{p_k^{n_k}}$.
We set up an explicit isomorphism as follows. For $\nu\in\{i,j,k\}$, define 
$M_\nu: = M/p_\nu^{n_\nu}= \prod_{\kappa\neq \nu} p_\kappa^{n_\kappa}$. 
Then each $x\in \ZZ_M$ can be written uniquely as
$$
x=\sum_{\nu\in\{i,j,k\}} \pi_\nu(x) M_\nu,\ \ \pi_\nu(x)\in \ZZ_{p_\nu^{n_\nu}}.
$$
Geometrically, this maps each
$x\in\ZZ_M$ to an element of a $3$-dimensional lattice with coordinates $(\pi_i(x),\pi_j(x),\pi_k(x))$. The tiling $A\oplus B=\ZZ_M$ corresponds to a tiling of that lattice. 

Let $D|M$ and $\nu\in\{i,j,k\}$. A {\em $D$-grid} in $\ZZ_M$ is a set of the form
$$
\Lambda(x,D):= x*D\ZZ_M=\{x'\in\ZZ_M:\ D|(x-x')\}
$$
for some $x\in\ZZ_M$. We note the following important special cases.

\begin{itemize}
\item A {\em line} through $x\in\ZZ_M$ in the $p_\nu$ direction is the set
$\ell_\nu(x):= \Lambda(x,M_\nu)$.
\item A {\em plane} through $x\in\ZZ_M$ perpendicular to the $p_\nu$ direction, on the scale $M_\nu p_\nu^{\alpha_\nu}$, is the set
$\Pi(x,p_\nu^{\alpha_\nu}):=\Lambda(x,p_\nu^{\alpha_\nu}).$

\item An {\em $M$-fiber in the $p_\nu$ direction} is a set of the form $x*F_\nu$, where $x\in\ZZ_M$ and
%\begin{equation}\label{def-Fi}
$$
F_\nu=\{0,M/p_\nu,2M/p_\nu,\dots,(p_\nu-1)M/p_\nu\}.
$$
%\end{equation}
Thus $x*F_\nu=\Lambda(x,M/p_\nu)$.  

\end{itemize}

We also need more general fibers, defined as follows. Let $N|M$, $c\in\NN$, and $\nu\in\{i,j,k\}$ such that $p_\nu|N$.
An {\em $N$-fiber in the $p_\nu$ direction} with multiplicity $c$ is a set $F\subset \ZZ_M$ such that
$F$ mod $N$ has the mask polynomial
$$
F(X)\equiv cX^a(1+X^{N/p_\nu}+ X^{2N/p_\nu}+\dots +X^{(p_\nu-1)N/p_\nu})\mod (X^{N}-1)
$$
for some $a\in \ZZ_M$. We will say sometimes that $F$ {\em passes through} $a$. %, or {\em is rooted at} $a$. 

A set $A\subset \ZZ_M$ is {\em $N$-fibered in the $p_\nu$ direction} if it can be written as a union of disjoint 
$N$-fibers in the $p_\nu$ direction, all with the same multiplicity. In particular, $A$ is {\em $M$-fibered in the $p_\nu$ direction} if there is a subset $A'\subset A$ such that $A=A'*F_\nu$.

If $N=p_i^{\alpha_i}p_j^{\alpha_j} p_k^{\alpha_k}$ is a divisor of $M$, with  $0\leq \alpha_\nu\leq n_\nu$, we let 
$$
D(N):= p_i^{\gamma_i} p_j^{\gamma_j}p_k^{\gamma_k},\ \hbox{ where }
\gamma_\nu=\max(0,\alpha_\nu-1)\hbox{ for }\nu\in\{i,j,k\}.
$$ 
We will also write $N_\nu=M/p_\nu$ for $\nu\in\{i,j,k\}$.

%%%%%%%%%%%%%%%%%%%%%%%%%%%%%%%%%%%%

\subsection{Divisor sets and dilations}

%%%%%%%%%%%%%%%%%%%%%%%%%%%%%%%%%%%%

For $N|M$ and $A\subset\ZZ_M$, we define
%\begin{equation}\label{divisors}
$$
\Div_N(A):=\{(a-a',N):\ a,a'\in A\}
$$
%\end{equation}
When $N=M$, we will omit the subscript and write $\Div(A)=\Div_M(A)$. Informally, we will refer to the elements of $\Div(A)$ as the {\em divisors of $A$} or {\em differences in $A$}. 
In cases when we need to indicate where a particular divisor of $A$ must occur, we will 
use the following notation for localized divisor sets. If $A_1,A_2\subset \ZZ_M$, we will write
%\begin{equation}\label{divisor-local}
$$
\Div_N(A_1,A_2):=\{(a_1-a_2,N):\ a_1\in A_1,a_2\in A_2\}.
$$
%\end{equation}
If one of the sets is a singleton, say $A_1=\{a_1\}$, we will simplify the notation and write $\Div_N(a_1,A_2)=\Div_N(\{a_1\},A_2)$. We note a basic theorem due to Sands.

\begin{theorem}\label{thm-sands} {\bf (Divisor exclusion; Sands \cite{Sands})}
Let $A,B\subset \ZZ_M$. Then $A\oplus B=\ZZ_M$
if and only if $|A|\,|B|=M$ and 
$$
%\begin{equation}\label{div-exclusion}
\Div(A) \cap \Div(B)=\{M\}.
%\end{equation}
$$
\end{theorem}

We note an important consequence, due also to Sands \cite{Sands} and generalized by Tijdeman \cite{Tij}. Let
$R=\{r\in\ZZ_M:(r,M)=1\}$. Then for any $r\in R$,
%\begin{equation}\label{dilation-invariance}
$$
A\oplus B=\ZZ_M \ \ \Leftrightarrow \ \ rA\oplus B=\ZZ_M,
$$
%\end{equation}
where $rA=\{ra: \ a\in A\}$. This follows from Theorem \ref{thm-sands} since $\Div(A)=\Div(rA)$ for any $r\in R$.

%%%%%%%%%%%%%%%%%%%%%%%%%%%%%%%%%%

\subsection{Standard tiling complements}\label{standards}

Suppose that $A\oplus B=\ZZ_M$. The {\em standard tiling complement} $A^\flat\subset\ZZ_M$, defined in \cite[Definition 3.1]{LaLo1} and based on a construction in \cite{CM}, is an explicit set that has the same prime power cyclotomic divisors as $A(X)$ and satisfies both (T1) and (T2). 
For the general definition of $A^\flat$, see \cite[Definition 3.1]{LaLo1}. If $M=p_i^2p_j^2p_k^2$ and $|A|=p_ip_jp_k$, the
only special case we need here is
$$
\hbox{ if }\Phi_{p_i^2}\Phi_{p_j^2}\Phi_{p_k^2}\mid A,
\hbox{ then }A^\flat = \Lambda(0,D(M)).
$$

\begin{proposition}\label{replacement}\cite[Proposition 3.4]{LaLo1}, \cite{CM}
Let $A\oplus B=\ZZ_M$. Then $A^\flat \oplus B= \ZZ_M$
if and only if $B$ satisfies (T2).
\end{proposition}

We say that the tilings $A\oplus B=\ZZ_M$ and $A'\oplus B=\ZZ_M$ are {\em T2-equivalent} if 
%\begin{equation}\label{t2-equi-a}
$$
A\hbox{ satisfies (T2) }\Leftrightarrow A'\hbox{ satisfies (T2).}
%\end{equation}
$$

Since $A$ and $A'$ tile the same group $\ZZ_M$ with the same tiling complement $B$, they must have the same cardinality and the same prime power cyclotomic divisors. We will sometimes say simply that $A$ is T2-equivalent to $A'$ if both $M$ and $B$ are clear from context.
Usually, $A'$ will be derived from $A$ using certain permitted manipulations such as fiber shifts (Lemma \ref{fibershift}).
In particular, if we can prove that either $A$ or $B$ in a given tiling is T2-equivalent to a standard tiling complement, this resolves the problem completely in that case.

\begin{corollary}[\cite{LaLo1} Corollary 3.6]\label{get-standard}
Suppose that the tiling $A\oplus B=\ZZ_M$ is T2-equivalent to the tiling $A^\flat \oplus B=\ZZ_M$. Then 
$A$ and $B$ satisfy (T2).
\end{corollary}

%%%%%%%%%%%%%%%%%%%%%%%%%%%%%

\subsection{Box notation and cuboids}\label{cuboid-section}

%%%%%%%%%%%%%%%%%%%%%%%%%%%%%

We use the $N$-box notation of \cite{LaLo1}, \cite{LaLo2}.
For $x\in\ZZ_M$, define
\begin{align*}
\bbA^N_m[x] & = \# \{a\in A:\ (x-a,N)=m \}.
\end{align*}
If $N=M$, we will usually omit the superscript and write
$\bbA^M_m[x]=\bbA_m[x]$. For $X\subset \ZZ_M$ and $x\in\ZZ_M$, we define
$\bbA_m^N [X] :=\sum_{x'\in X}\bbA^N_m[x']$ and
$$
\bbA^N_m[x|X] = \# \{a\in A\cap X: \ (x-a,N)=m\}.
$$

Cuboids are an important tool in the literature on cyclotomic divisibility and Fuglede's conjecture, see e.g., 
\cite{KMSV}, \cite{KMSV2}, \cite{LaLo1}, \cite{LaLo2}, \cite{Steinberger}. We follow the presentation in \cite[Section 5]{LaLo1}. 

\begin{definition}
Let $M=p_i^{n_i}p_j^{n_j}p_k^{n_k}$and $N|M$. An $N$-cuboid is a weighted multiset corresponding to a mask polynomial of the form
$$%\begin{equation}\label{def-N-cuboids}
	\Delta(X)= X^c\prod_{p_\nu|N} (1-X^{d_\nu}),
$$ %\end{equation}
with $(d_\nu,N)=N/p_\nu$ for all $\nu $ such that $p_\nu|N$.
\end{definition}

Cuboids provide useful criteria to determine cyclotomic divisibility properties of mask polynomials. 
For $A\in\calm(\ZZ_N)$, we have
$\Phi_N(X)|A(X)$ if and only if $\bbA^N_N[\Delta]=0$ for every $N$-cuboid $\Delta$. This 
has been known and used previously in the literature,
see e.g.  \cite[Section 3]{Steinberger}, or \cite[Section 3]{KMSV}.
In particular, for any $N|M$, $\Phi_N$ divides $A$ if and only if it divides the mask polynomial of $A\cap\Lambda(x,D(N))$ for every $x\in\ZZ_M$.

%%%%%%%%%%%%%%%%%%%%%%%%%%%%%%%%%%%%%%%%

\section{Saturating sets and fiber shifting}\label{product-dilations}

%%%%%%%%%%%%%%%%%%%%%%%%%%%%%%%%%%%%%

%%%%%%%%%%%%%%%%%%%%%%%%%%%%%%%

\subsection{Box product and saturating sets}\label{box product revisited}

Following \cite{LaLo1}, we define the {\em $M$-box product} as follows.
If $A,B\subset\ZZ_M$, let
$$
%\begin{equation}\label{inner-product}
\langle \bbA[x], \bbB[y] \rangle = \sum_{m|M} \frac{1}{\phi(M/m)} \bbA_m[x] \bbB_m[y].
%\end{equation}
$$
Here $\phi$ is the Euler totient function: if 
$n=\prod_{\iota=1}^L q_\iota^{r_\iota}$, where $q_1,\dots,q_L$ are distinct primes and $r_\iota\in\NN$, then
$
\phi(n)= \prod_{\iota=1}^L (q_\iota-1)q_\iota^{r_\iota-1}.
$

\begin{theorem}\label{GLW-thm}
(\cite{LaLo1}; following \cite[Theorem 1]{GLW}) If
$A\oplus B=\ZZ_M$ is a tiling, then 
\begin{equation}\label{e-ortho2}
\langle \bbA[x], \bbB[y] \rangle =1\ \ \forall x,y\in\ZZ_M.
\end{equation}
\end{theorem}

Let $A\oplus B= \ZZ_M$, and $x,y\in\ZZ_M$. We define the sets $A_{x,y}$ and $B_{y,x}$ to be the sets that saturate the box product on the left side of (\ref{e-ortho2}):
$$
A_{x,y}:=\{a\in A:\ (x-a,M)=(y-b,M) \hbox{ for some }b\in B\},
$$
and similarly for $B_{y,x}$ with $A$ and $B$ interchanged.
The {\em saturating set} for $x$ is 
$$
A_{x}:=\{a\in A: (x-a,M)\in\Div(B)\} =\bigcup_{b\in B} A_{x,b},
$$
with $B_{y}$ defined similarly. We encourage the reader to consult \cite[Section 3]{LaLo3} for a combinatorial proof of Theorem \ref{GLW-thm} and an alternative description of saturating sets in terms of dilations by the elements of $R$.

By Theorem \ref{thm-sands}, $A_a=\{a\}$ for all $a\in A$. For $x\in\ZZ_M\setminus A$, $A_x$ must be nonempty by (\ref{e-ortho2}), and obeys the following geometric constraints.
For $x,x'\in\ZZ_M$ such that $(x-x',M)=p_i^{\alpha_i} p_j^{\alpha_j} p_k^{\alpha_k}$, where $0\leq \alpha_\nu\leq n_\nu$, define
%\begin{equation}\label{e-span}
%\begin{split}
$$
\Span(x,x')=\bigcup_{\nu: \alpha_\nu<n_\nu} \Pi(x,p_\nu^{\alpha_\nu+1}),
$$
$$
\Bispan(x,x')= \Span(x,x')\cup \Span(x',x).
$$
%\end{split}
%\end{equation}
Then for any $x,x',y\in\ZZ_M$, we have
\begin{equation}\label{setplusspan}
 A_{x',y}\subset A_{x,y}\cup\Bispan(x,x'),
\end{equation}
and in particular,
\begin{equation}\label{bispan}
A_x \subset \bigcap_{a\in A} \Bispan (x,a).
\end{equation}
In an important special case, if $x\in \ZZ_M\setminus A$ satisfies $(x-a,M)=M/p_i$ for some $a\in A$, then
$$
A_x\subset \Bispan(x,a)=\Pi(x,p_i^{n_i}) \cup\Pi(a,p_i^{n_i}).
$$
Our evaluations of saturating sets will always begin with (\ref{bispan}).

%%%%%%%%%%%%%%%%%%%%%%%%%%%%%%%%%%%%%%%%%%%%%

\subsection{Cofibered structures}

%%%%%%%%%%%%%%%%%%%%%%%%%%%%%%%%%%%%%%%%%%%%%%%

The following is a simplified version of the definitions and results of  \cite[Section 8]{LaLo1}, restricted to $M=p_i^2p_j^2p_k^2$. 
%Most of this article can be read with just these definitions, if the reader is willing to 
%substitute $n_i=n_j=n_k=2$ in all arguments. On those few occasions when the more general versions are necessary even in that case, we have to refer the reader to \cite[Section 8]{LaLo1}.
If $F\subset \ZZ_M$ is an $M$-fiber in the $p_\nu$ direction, we say that an element $x\in \ZZ_M$ is at {\em distance} $m$ from $F$ if $m|M$ is the maximal divisor
such that $(z-x,M)=m$ for some $z\in F$. It is easy to see that such $m$ exists.

Let $A\oplus B=\ZZ_M$ be a tiling. We will often be interested 
in finding ``complementary" fibers and fibered structures in $A$ and $B$, in the following sense.

\begin{definition}[\bf Cofibers and cofibered structures]\label{cofibers} 
Let $A, B\subset \ZZ_M$ and $\nu\in\{i,j,k\}$.

\smallskip
 
(i) We say that $F\subset A,G\subset B$ are {\em $(1,2)$-cofibers} in the $p_\nu$ direction
if $F$ is an $M$-fiber and $G$ is an $M/p_\nu$-fiber, both in the $p_\nu$ direction.

\smallskip
(ii) We say that the pair $(A,B)$ has a {\em (1,2)-cofibered structure} in the $p_\nu$ direction if

\begin{itemize}
\item $B$ is  $M/p_\nu$-fibered in the $p_\nu$ direction,

\item $A$ contains at least one ``complementary" $M$-fiber $F\subset A$
 in the $p_\nu$ direction, which we will call a  {\em cofiber} for this structure. 
 \end{itemize}

\end{definition}

The advantage of cofibered structure is that it permits fiber shifts as described below. In many cases, we will be able to use this to reduce the given tiling to a simpler one.

\begin{lemma}[\bf Fiber-Shifting Lemma]\label{fibershift} 
Let $A\oplus B=\ZZ_M$. Assume that the pair $(A,B)$ has a $(1,2)$-cofibered structure,
with a cofiber $F\subset A$. Let $A'$ be the set obtained from $A$ by shifting $F$ to a point $x\in\ZZ_M$ at 
a distance $M/p_i^{2}$ from it. 
Then $A'\oplus B=\ZZ_M$, and $A$ is T2-equivalent to $A'$.
\end{lemma}

In order to find cofibered structures in $(A,B)$, we will typically use 
saturating sets, via the following lemma.

\begin{lemma}\label{1dim_sat-cor}
Assume that $A\oplus B =\ZZ_M$ is a tiling, with $M=p_i^2p_j^2p_k^2$. 
Suppose that $x\in\ZZ_M\setminus A$, $b\in B$, $M/p_\nu\in\Div(A)$, and $A_{x,b}\subset \ell_\nu(x)$ for some $\nu\in\{i,j,k\}$ and $b\in B$. Then
%\begin{equation}\label{fib-e300}
$$
\mathbb{A}^M_{M/p_\nu^2}[x]\mathbb{B}^M_{M/p_\nu^2}[b]=\phi(p_\nu^2).
$$
%\end{equation}
with the product saturated by a $(1,2)$-cofiber pair $(F,G)$ such that $F\subset A$ is at distance $M/p_i^2$ from $x$ and $G\subset B$ is rooted at $b$.
In particular, if $A_x\subset \ell_\nu(x)$,
then the pair $(A,B)$ has a $(1,2)$-cofibered structure.
\end{lemma}

%%%%%%%%%%%%%%%%%%%%%%%%%%%%%%%%%%%%%%

\section{Splitting and tiling reductions}\label{splitting-section}

%%%%%%%%%%%%%%%%%%%%%%%%%%%%%%%%%%%%%%%%%%%%

\subsection{Splitting}
The definitions and notation below are from \cite[Section 4]{LaLo3}.

\begin{definition}
Let $M=p_1^{n_1}\dots p_K^{n_K}$, and assume that $A\oplus B=\ZZ_M$ is a tiling.
For a set $Z\subset \ZZ_M$, define
\begin{align*}
\Sigma_{A}(Z)&=\{a\in A:\ z=a+b\hbox{ for some }z\in Z,\ b\in B\},
\\
\Sigma_{B}(Z)&=\{b\in B:\ z=a+b\hbox{ for some }z\in Z,\ a\in A\}.
\end{align*}
\end{definition}

Note that $\Sigma_{A}(Z)$ depends on both $A$ and $B$. When more than one tiling complement of $A$ is being considered, we will identify the relevant tiling explicitly.

\begin{definition}\label{def-splitting}
	Let $Z=x*F_i\subset \ZZ_M$ be an $M$-fiber in the $p_i$ direction. We will say that $Z$ {\em splits with parity $(A,B)$} if:
	\begin{itemize}
		\item[(i)] $p_i^{n_i}|a-a'$ for any $a,a'\in\Sigma_A(Z)$,
		\item[(ii)] $p_i^{n_i-1}\parallel b-b'$ for any two distinct $b,b'\in \Sigma_B(Z)$. 
	\end{itemize}
\end{definition}

\begin{lemma}\label{no-upgrades}{\bf (Splitting for fibers)} \cite[Lemma 4.3]{LaLo3}
	Every $M$-fiber $Z$ splits with parity either $(A,B)$ or $(B,A)$. In particular, if $Z$ is an $M$-fiber in the $p_i$ direction, then for any $a\in \Sigma_A(Z)$ and $b\in \Sigma_B(Z)$, we have
	$\Sigma_A(Z)\subset\Pi(a,p_i^{n_i-1})$ and $\Sigma_B(Z)\subset\Pi(b,p_i^{n_i-1})$.
\end{lemma}

We do not know, in general, whether all $M$-fibers in a given direction split with the same parity. However, if a tiling does have uniform splitting in the sense of Definition \ref{uniform-splitting} below, we can use this to our advantage as described in the next section.

\begin{definition}\label{uniform-splitting}
Let $A\oplus B=\ZZ_M$. We say that the tiling $A\oplus B=\ZZ_M$ has {\em uniform $(A,B)$ splitting parity in the $p_i$ direction} if all $M$-fibers in the $p_i$ direction split with parity $(A,B)$. Uniform $(B,A)$ splitting parity is defined analogously. 
\end{definition}

%\begin{definition}\label{uniform-splitting2}
%Let $A\oplus B=\ZZ_M$. %, where $M=p_1^{n_1}\dots p_K^{n_K}$. 
%Given $a\in A$ and $b\in B$, 
%we say that the product $\langle\bbA[a*F_i],\bbB[b]\rangle$ {\em splits with parity $(A,B)$} if $A_{x,b}\subset \Pi(a,p_i^{n_i})$ for all $x\in a*F_i$, and {\em with parity $(B,A)$} if $A_{x,b}\subset \Pi(x,p_i^{n_i})$ for all $x\in a*F_i$.
%\end{definition}

%%%%%%%%%%%%%%%%%%%%%%%%%%%%%%%%%

\subsection{Tiling reductions}\label{tiling-reductions}

%%%%%%%%%%%%%%%%%%%%%%%%%%%%%%%%%%%%%

%\subsection{Tiling reductions for 3 primes}
The tiling reductions below allow us, under certain assumptions, to decompose a tiling $A\oplus B=\ZZ_M$ into a family of tilings of $\ZZ_{M/p_\nu}$ for some $\nu$, with the additional property that if (T2) holds for both sets in each of the smaller tilings, then it also holds for $A$ and $B$. Both reductions are valid for tilings of $\ZZ_M$ with no assumptions on the prime factorization of $M$. In order to deduce (T2) for $A$ and $B$, we must know that (T2) holds for both sets in the smaller tilings. If $M=p_i^2p_j^2p_k^2$, this is provided by \cite[Corollary 6.2]{LaLo1} (based on the methods of \cite{CM}). Theorem \ref{subgroup-reduction} and Corollary \ref{slab-reduction} combine the two steps in a form ready to apply here.

The subgroup reduction is due to Coven and Meyerowitz \cite{CM}; the formulation we use is from \cite[Theorem 6.1]{LaLo1}.

\begin{theorem}\label{subgroup-reduction} {\bf (Subgroup reduction)} \cite[Lemma 2.5]{CM}
Let $M=p_i^{n_i}p_j^{n_j}p_k^{n_k}$.
Assume that $ A\oplus B=\ZZ_M $, and that $A\subset p_\nu \ZZ_M$ for some 
$\nu\in\{i,j,k\}$ such that $p_\nu\parallel |B|$.
Then $A$ and $B$ satisfy (T2). 
\end{theorem}

The slab reduction was introduced in \cite{LaLo2}. The statement below follows from \cite[Theorem 6.5]{LaLo2}. and \cite[Lemma 5.4]{LaLo3}.

\begin{theorem}\label{subtile} \cite[Theorem 6.5]{LaLo1} and \cite[Lemma 5.4]{LaLo3}
Let $M=p_i^{n_i}p_j^{n_j} p_k^{n_k}$.
Assume that $A\oplus B=\ZZ_M$, and let $\nu\in\{i,j,k\}$.
Define
$$
%\begin{equation}\label{Asubtile}
A_{p_\nu}=\{a\in A:\ 0\leq\pi_\nu(a)\leq p_\nu^{n_\nu-1}-1\}.
%\end{equation}
$$
Then the following are equivalent:

\begin{enumerate}
\item [(i)]  For any translate $A'$ of $A$, we have $A'_{p_\nu}\oplus B=\ZZ_{M/p_\nu}$.

\item[(ii)] The tiling $A\oplus rB=\ZZ_M$
has uniform $(rB,A)$ splitting parity in the $p_\nu$ direction, for all $r\in R$.
\item [(iii)] For every $a\in A$ and $b\in B$, we have $A_{x,b}\subset \Pi(x,p_i^{n_i})$ for all $x\in a*F_i$.

\end{enumerate}
\end{theorem}

We have not included parts (ii) and (iii) of \cite[Theorem 6.5]{LaLo1} since we will not use them here; the conditons stated above, taken from \cite[Lemma 5.4]{LaLo3}, are much easier to use. We have also made a minor modification in (i), as follows: in \cite{LaLo1}, we use $\Phi_{p_\nu^{n_\nu}}|A$ as an assumption of the theorem. However, if (i) holds, then $B$ tiles $\ZZ_{M/p_\nu}$. Hence it satisfies (T1), and in particular $\Phi_{p_\nu^{n_\nu}}\nmid B$. Therefore (i) together with the tiling assumption $A\oplus B=\ZZ_M$ implies that $\Phi_{p_\nu^{n_\nu}}|A$.

\begin{corollary}\label{slab-reduction} {\bf (Slab reduction)} \cite[Corollary 6.7]{LaLo2}
Let $M=p_i^{n_i}p_j^{n_j}p_k^{n_k}$.
Assume that $A\oplus B=\ZZ_M$, and that 
there exists a $\nu\in\{i,j,k\}$ such that $p_\nu\parallel |A|$, and $A,B$ obey any of the conditions of Theorem \ref{subtile} for that $\nu$. 
Then $A$ and $B$ satisfy (T2). 
\end{corollary}

\begin{corollary}\label{slabcor} \cite[Corollary 5.5]{LaLo3}
Let $A\oplus B=\ZZ_M$. 
Assume that at least one of the following holds for some $\nu\in\{i,j,k\}$:
\begin{itemize}
\item[(i)] $\bbA_{M/p_\nu}[a]>0$ for every $a\in A$ (in particular, this holds if $A$ is $M$-fibered in the $p_\nu$ direction),
\item[(ii)] for every $b\in B$, we have 
\begin{equation}\label{maxplanebound}
|B\cap\Pi(b,p_\nu^{n_\nu})|=|B|/(|B|,p_\nu^{n_\nu}).
\end{equation}
\end{itemize}
Then $A$ satisfies the conditions of Theorem \ref{subtile}.
\end{corollary}

%%%%%%%%%%%%%%%%%%%%%%%%%%%%%%%%%%%%%%

\section{Classification results}\label{classified}

%%%%%%%%%%%%%%%%%%%%%%%%%%%%%%%%%%%%%%

\subsection{Classification results}\label{classified1}

We now state our results on the classification of tilings with three prime factors and the (T2) property for such tilings. We restrict our attention to tilings 
$A\oplus B=\ZZ_M$, where $M=p_i^{n_i}p_j^{n_j}p_k^{n_k}$ has three distinct prime factors. Our main results require the additional assumption that 
\begin{equation}\label{allnare2}
n_i=n_j=n_k=2 \hbox{ and }  |A|=|B|=p_ip_jp_k,
\end{equation}
but some of our intermediate results are also valid without (\ref{allnare2}).

As in \cite{LaLo2}, we start with the assumption that $\Phi_M|A$. (Since $\Phi_M$ divides at least one of $A(X)$ and $B(X)$ by (\ref{poly-e2}), we may always assume this after interchanging $A$ and $B$ if necessary.) This implies structure results for $A$ on grids $\Lambda(x,D(M))$ for every $x\in\ZZ_M$. 

Let $\Lambda:=\Lambda(a,D(M))$ for some $a\in A$, so that $A\cap\Lambda$ is nonempty.  
By the classic results on vanishing sums of roots of unity \cite{deB}, \cite{Re1}, \cite{Re2}, \cite{schoen}, \cite{Mann}, \cite{LL}, $\Phi_M$ divides $A\cap\Lambda$ if and only if $A\cap\Lambda$ is a linear combination of $M$-fibers with integer coefficients. In other words, 
$$
(A\cap\Lambda)(X)=\sum_{\nu\in\{i,j,k\}} Q_\nu(X) F_\nu(X),
$$
where $Q_i,Q_j,Q_k$ are polynomials with integer coefficients depending on both $A$ and $\Lambda$.

In \cite{LaLo2}, we used this to develop a classification of sets $A\cap\Lambda$, where $A$ is a finite tile, $\Phi_M|A$, and $\Lambda$ is a $D(M)$-grid. One possibility is that $A$ is $M$-fibered on each such grid $\Lambda$, so that $(A\cap\Lambda)(X)=
Q_{\Lambda}(X)F_{\nu(\Lambda)}(X)$ for some $\nu(\Lambda)\in \{i,j,k\}$, possibly depending on $\Lambda$. If $A$ is fibered on all $D(M)$-grids in the same direction, so that $\nu(\Lambda)$ can be chosen independent of $\Lambda$, the conditions of Theorem \ref{subtile} are satisfied and we may use Corollary \ref{slab-reduction} to conclude that (T2) holds for both $A$ and $B$.
However, it is also possible for $A\cap\Lambda$ to be fibered in different directions on different grids $\Lambda$. Additionally, there may exist grids $\Lambda$ such that $A\cap\Lambda$ is not fibered. This can happen if $A\cap\Lambda$ contains nonintersecting $M$-fibers in two or three different directions, or if some of the polynomials $Q_i,Q_j,Q_k$ have negative coefficients, resulting in cancellations between fibers in different directions.

Our classification and (T2) results are summarized in 
Theorems \ref{unfibered-mainthm} and \ref{fibered-thm} below.
Both theorems were proved in \cite{LaLo2} with the additional assumption that $M$ is odd. In this paper, we prove that the same conclusions hold when $M$ is even. 

\begin{theorem}\label{unfibered-mainthm}
Let $A\oplus B=\ZZ_M$, where $M=p_i^{2}p_j^{2}p_k^{2}$. Assume that
$|A|=|B|=p_ip_jp_k$, $\Phi_M|A$, and that there exists a $D(M)$-grid $\Lambda$ such that $A\cap\Lambda$ is nonempty and is not $M$-fibered in any direction. Assume further, without loss of generality, that $0\in\Lambda$.
Then $A^\flat=\Lambda$, and the tiling $A\oplus B=\ZZ_M$ is T2-equivalent to $\Lambda \oplus B=\ZZ_M$ via fiber shifts. By Corollary \ref{get-standard}, both $A$ and $B$ satisfy (T2).
\end{theorem}

\begin{theorem}\label{fibered-thm}
Let $A\oplus B=\ZZ_M$, where $M=p_i^{2}p_j^{2}p_k^{2}$. Assume that
$|A|=|B|=p_ip_jp_k$, $\Phi_M|A$, and that for every $a\in A$, the set $A\cap\Lambda(a,D(M))$ is $M$-fibered in at least one direction (possibly depending on $a$). 

\medskip\noindent
{\rm (I)} Suppose that there exists an element $a_0\in A$ such that 
\begin{equation}\label{intro-inter}
a_0*F_\nu\subset A\ \ \forall \nu\in\{i,j,k\}.
\end{equation}
Then the tiling $A\oplus B=\ZZ_M$ is T2-equivalent to $\Lambda\oplus B=\ZZ_M$ via fiber shifts, where $\Lambda:=\Lambda(a_0,D(M))$.
By Corollary \ref{get-standard}, both $A$ and $B$ satisfy (T2).

\medskip\noindent
{\rm (II)} Assume that no $a_0\in A$ satisfies (\ref{intro-inter}). Then at least one of the following holds.

\begin{itemize}
\item We have $A\subset \Pi(a,p_\nu)$ for some $a\in A$ and $\nu\in\{i,j,k\}$. By Theorem \ref{subgroup-reduction}, both $A$ and $B$ satisfy (T2).

\item There exists a $\nu\in\{i,j,k\}$ such that (possibly after interchanging $A$ and $B$) the conditions of Theorem \ref{subtile} are satisfied in the $p_\nu$ direction. By Corollary \ref{slab-reduction}, both $A$ and $B$ satisfy (T2).

\end{itemize}

\end{theorem}

We will provide a more detailed breakdown of the case (II) of  Theorem \ref{fibered-thm}  in
Theorem \ref{fibered-mainthm}, after the appropriate terminology has been introduced.

%%%%%%%%%%%%%%%%%%%%%%%%%%%%%%%

\subsection{Outline of the proof}\label{classified2}

The general scheme of the proof of Theorems \ref{unfibered-mainthm} and \ref{fibered-thm} is similar to that in \cite{LaLo2} for odd $M$. We will take advantage of the results already proved in \cite{LaLo2} where possible, and use the methods and technical tools developed there. In the outline below, we describe the new contributions of this paper and explain how they fit into the framework of \cite{LaLo2}.

We assume that $A\oplus B=\ZZ_M$, where $M=p_i^{2}p_j^{2}p_k^{2}$, 
$|A|=|B|=p_ip_jp_k$, and $\Phi_M|A$. 
Our proof splits into two parts according to the fibering properties of $A$.

Assume first that there exists a $D(M)$-grid $\Lambda$ such that $A\cap\Lambda$ is not $M$-fibered in any direction. In Propositions 5.2 and 5.5 in \cite{LaLo2}, we proved that $A\cap\Lambda$ must then contain at least one of two special structures, either {\em diagonal boxes} or an {\em extended corner}. The latter case was resolved in \cite[Theorem 8.1]{LaLo2}, for both odd and even $M$.

We are left with the case when $M$ is even and $A\cap\Lambda$ contains diagonal boxes for some $D(M)$-grid $\Lambda$. As a preliminary reduction, we adapt the proof of Propositions 5.2 and 6.1 in \cite{LaLo2} to prove that $A\cap\Lambda$ must in fact be the union of the diagonal boxes and possibly additional $M$-fibers in one or more directions (Proposition \ref{db-prop}; see also the remarks after the proposition).  We will also need the classification of unfibered structures with
$\{m:\ D(M)|m|M\}\not\subset\Div(A)$, developed in \cite[Section 6]{LaLo2}. The relevant results are 
summarized in Section \ref{special-structures}. They will be 
needed both in the unfibered case currently under consideration, and in the fibered case where they will be applied on lower scales.

We resolve the diagonal boxes case in Section \ref{res-boxes}. Our main result in this regard is Theorem \ref{db-theorem}, stating that if $A\cap\Lambda$ contains diagonal boxes, then $A$ is T2-equivalent via fiber shifts to $\Lambda$. In particular, (T2) holds for both $A$ and $B$. 

The fiber shifting method was already used in \cite{LaLo2}, and some of our techniques are similar. However, the diagonal boxes structures in the even case are significantly more difficult to resolve than their odd case counterparts. For one thing, they can contain very few points. (We draw the reader's attention to the case labelled here as (DB2), when $A\cap\Lambda$ consists of just two incomplete fibers and $\Div(A\cap\Lambda)$ is a three-element set.) In such cases, saturating set arguments are less efficient than they were in \cite{LaLo2}, basically because there are fewer geometric restrictions coming from (\ref{bispan}). Additionally, unlike in the odd case, diagonal boxes with $p_i=2$ do not exclude {\it a priori} the possibility that $B$ might be $M$-fibered in the $p_i$ direction; this does turn out to be impossible, but only after a longer argument.

 In the particularly difficult case (DB2), we deal with this by combining saturating set techniques with splitting arguments. Intuitively, splitting arguments require less information than saturating sets. They also provide less information; however, in the even case this can be sufficient. We refer the reader to Lemma \ref{edb2-nodiag} and Proposition \ref{nodiagdiv} for examples of this.

In Section \ref{fibered-sec}, we consider the case when $A\cap\Lambda$ is fibered in some direction (possibly depending on $\Lambda$) for each $D(M)$-grid $\Lambda$. 
Part (I) of Theorem \ref{fibered-thm} was proved in \cite{LaLo2} for both odd and even $M$. 
Our proof of (T2) in the remaining case (II) is a significant departure from that in \cite{LaLo2}, even though some of the intermediate results and technical ingredients are the same. We still prove that only two fibering directions are allowed in this case, and the detailed breakdown of the conclusions is the same as in \cite{LaLo2}. However, the proof is organized differently, and the use of the splitting method from \cite{LaLo3} leads to many simplifications. We defer a longer discussion to Section \ref{fibered-sec}, after the relevant concepts and notation have been introduced.

As in \cite{LaLo2}, our final result is restricted to the case when $M=p_i^{2}p_j^{2}p_k^{2}$, but some of our methods and intermediate results are valid under weaker assumptions. Ultimately, however, we do need to assume that $n_i=n_j=n_k=2$ in order to complete the proof of (T2).
In that case, a single application of either the subgroup reduction or the slab reduction brings us to a case when (T2) can be proved using simpler methods. It is likely that new multiscale methods will be needed to go beyond this constraint.

%%%%%%%%%%%%%%%%%%%%%%%%%%%%%%%%%%%%%%

\section{Toolbox}\label{sec-toolbox}

%%%%%%%%%%%%%%%%%%%%%%%%%%%%%%%%%%%%%%

\subsection{Divisors} The first part of Lemma \ref{triangles} below is Lemma 8.9 of \cite{LaLo1}. The second part is specific to the case when $p_i=2$.

\begin{lemma}[\bf Enhanced divisor exclusion] \label{triangles} 
Let $A\oplus B=\ZZ_M$, with $M=p_i^{n_i}p_j^{n_j}p_k^{n_k}$. Let $m=p_i^{\alpha_i}p_j^{\alpha_j}p_k^{\alpha_k}$ and $m'=p_i^{\alpha'_i}p_j^{\alpha'_j}p_k^{\alpha'_k}$, with $0\leq \alpha_\iota,\alpha'_\iota\leq n_\iota,\,\iota\in\{i,j,k\}$. 
	
\smallskip
(i) Assume that at least one of $m,m'$ is different from $M$, and that for every $\iota\in\{i,j,k\}$ we have 
\begin{equation}\label{triangles-e1}
\hbox{ either }\alpha_\iota\neq \alpha'_\iota \hbox{ or }
\alpha_\iota =\alpha'_\iota=n_\iota.
\end{equation}
Then for all $x,y\in\ZZ_M$ we have
$$
\bbA_m[x] \,\bbA_{m'}[x] \, \bbB_m[y] \, \bbB_{m'}[y] =0.
$$

\smallskip
(ii) If $p_\nu=2$ for some $\nu\in\{i,j,k\}$, then for that $\nu$, the assumption (\ref{triangles-e1}) may be replaced by
\begin{equation}\label{triangles-e2}
\hbox{ either }\alpha_\nu\neq \alpha'_\nu \hbox{ or }
\alpha_\nu =\alpha'_\nu\in\{n_\nu,n_\nu-1\},
\end{equation}
and the same conclusion holds.
\end{lemma}

\begin{proof}
Assume towards contradiction that there exist $a,a'\in A$, $b,b'\in B$, $x,y\in\ZZ_M$ such that
$$
(a-x,M)=(b-y,M)=m,\ \ (a'-x,M)=(b'-y,M)=m'.
$$
This together with (\ref{triangles-e1}) for all $\iota$ implies that
$$
(a-a',M)=(b-b',M)=\prod_{\iota\in\{i,j,k\}} p_\iota^{\min(\alpha_\iota,\alpha'_\iota)},
$$
with the right side different from $M$. But that is prohibited.
	
If $p_\nu=2$ and we assume (\ref{triangles-e2}) instead of (\ref{triangles-e1}) for $\iota=\nu$, then the same conclusion holds, since in this case $\alpha_\nu =\alpha'_\nu=n_\nu-1$ still implies that $p_\nu^{n_\nu}|a-a'$ and $p_\nu^{n_\nu}|b-b'$.
\end{proof}

%%%%%%%%%%%%%%%%%%%%%%%%%%%%%%%%%%%%%%

\subsection{Cyclotomic divisibility and fibering}

The results here are borrowed from \cite[Section 4.2]{LaLo2}.

\begin{lemma}[\bf Cyclotomic divisibility on grids]\label{cyclo-grid}
Let $A\in\calm(\ZZ_M), M=p_i^{n_i}p_j^{n_j}p_k^{n_k}$, and let $m,s|M$ with $s\neq 1$. Suppose that for every $a\in A$, $\Phi_s$ divides $A\cap\Lambda(a,m)$. Then $\Phi_s|A$.
\end{lemma}

\begin{lemma}[\bf Plane bound]\label{planebound}
Let $A\oplus B=\ZZ_M$, where $M=p_i^{n_i}p_j^{n_j}p_k^{n_k}$ and $|A|=p_i^{\beta_i}p_j^{\beta_j}p_k^{\beta_k}$.
Then for every $x\in\ZZ_M$ and $0\leq\alpha_i\leq n_i$ we have 
$$
|A\cap\Pi(x,p_i^{n_i-\alpha_i})|\leq p_i^{\alpha_i}p_j^{\beta_j} p_k^{\beta_k}.
$$
\end{lemma}

\begin{corollary}\label{planegrid}
Let $A\oplus B=\ZZ_M$, where $M=p_i^{n_i}p_j^{n_j}p_k^{n_k}$ and $|A|=p_i^{\beta_i}p_j^{\beta_j}p_k^{\beta_k}$ 
with $\beta_i>0$.
Suppose that for some $x\in \ZZ_M$ and $1\leq\alpha_0\leq n_i$
$$
%\begin{equation}\label{Aplanebound}
|A\cap\Pi(x,p_i^{n_i-\alpha_0})|>p_i^{\beta_i-1}p_j^{\beta_j}p_k^{\beta_k},
%\end{equation}
$$
then $\Phi_{p_i^{n_i-\alpha}}|A$ for at least one $\alpha\in \{0,\ldots,\alpha_0-1\}$. 
\end{corollary}

Lemma \ref{2d-cyclo} below is a simple version of the de Bruijn-R\'edei-Schoenberg theorem for cyclic groups $\ZZ_N$ when $N$ has at most two distinct prime factors. This has been known in the literature, see \cite{deB}, \cite[Theorem 3.3]{LL}. The version here is from \cite[Lemma 4.7]{LaLo2}. 

\begin{lemma}\label{2d-cyclo} {\bf (Cyclotomic divisibility for 2 prime factors)}
Let $M=p_i^{n_i}p_j^{n_j}p_k^{n_k}$, and let $ A\in \mathcal{M}(\ZZ_N)$,  for some $N\mid M$ such that $N=p_j^{\alpha_j}p_k^{\alpha_k}$ has only two distinct prime factors. Then:

\smallskip
(i) $\Phi_N|A$ if and only if $A$ is a linear combination of $N$-fibers in the $p_j$ and $p_k$ direction with non-negative integer coefficients.

\smallskip
(ii) Let $\Lambda$ be a $D(N)$-grid. Assume that $\Phi_N|A$, and that there
exists $c_0\in\NN$ such that $\bbA^N_N[x]\in\{0,c_0\}$ for all $x\in\Lambda$. 
Then $A\cap\Lambda$ is $N$-fibered in either the $p_j$ or the $p_k$ direction.
\end{lemma}

The following special case will be used several times.

\begin{corollary}\label{doublediv}
Let $A\subset\ZZ_M$, where $M=p_i^{2}p_j^{2}p_k^{2}$. Assume that $\Phi_{N_i}\Phi_{M_i}|A$, and that there
exists $c_0\in\NN$ such that 
\begin{equation}\label{binary-cond-2}
\bbA^{N_i}_{N_i}[x]\in\{0,c_0\}\hbox{ for all }x\in\ZZ_M.
\end{equation}
(Note that (\ref{binary-cond-2}) is satisfied with $c_0=1$ if $M/p_i\not\in\Div(A)$, and with $c_0=p_i$ if $A$ is $M$-fibered in the $p_i$ direction.)
Then $A$ is a union of pairwise disjoint $N_i$-fibers in the $p_j$ and $p_k$ directions, each of multiplicity $c_0$.
\end{corollary}

\begin{proof}
Consider $A$ modulo $N_i$. Without loss of generality, we may assume that $c_0=1$. 
The assumption that $\Phi_{N_i}\Phi_{M_i}|A$ implies that %$A$ mod $N_i$ is $\calt$-null with respect to the cuboid type $\calt=(N_i,(0,1,1),1)$. In other words, 
for any $N_i/p_jp_k$-grid $\Lambda$ in $\ZZ_{N_i}$, $\Phi_{M_i}$ divides $(A\cap\Lambda)(X)$. By  Lemma \ref{2d-cyclo}, $A\cap\Lambda$ mod $N_i$ is $N_i$-fibered in one of the $p_j$ and $p_k$ directions. Writing $\ZZ_{N_i}$ as a union of pairwise disjoint $N_i/p_jp_k$-grids, we get the conclusion of the corollary.
\end{proof}

%%%%%%%%%%%%%%%%%%%%%%%%%%%%%%%%%%%%%%

\section{Structure on unfibered grids}\label{unfibered-structure-section}

%%%%%%%%%%%%%%%%%%%%%%%%%%%%%%%%%%%%%%

Throughout this section, we will use the following notation. Let $M=p_i^{n_i}p_j^{n_j}p_k^{n_k}$, and let
$\Lambda$ be a fixed $D(M)$-grid such that $A\cap\Lambda\neq\emptyset$. 
We identify $\Lambda$ with $\ZZ_{p_i}\oplus \ZZ_{p_j} \oplus \ZZ_{p_k}$, and represent
each point $x\in \Lambda$ as $(\lambda_ix,\lambda_j x, \lambda_k x)$ in the implied coordinate system. Assume that $A\cap\Lambda$ is not $M$-fibered in any direction (we will call such grids {\em unfibered}). 
In Sections 5 and 6 of \cite{LaLo2}, we classified the types of structure that $A$ may have on such grids. Below, we outline those results and identify the cases that were already resolved in \cite{LaLo2} for all $M$ as above, including the even case. The remaining cases with $M$ even will be resolved here in Section \ref{res-boxes}.

%%%%%%%%%%%%%%%%%%%%%%%%%%%%%%%%

\subsection{Basic structure results}\label{basic structure}
We start with the case when $A\cap\Lambda$ is a union of pairwise disjoint fibers. Since we are assuming that $A\cap\Lambda$ is unfibered, it must contain fibers in at least two different directions. 
In that case, 
$A\cap\Lambda$ contains the following structure.

\begin{definition}\label{corner}\cite[Definition 5.4]{LaLo2}
Suppose that $ A\subset \ZZ_M$, and let $\Lambda$ be a $D(M)$-grid.
	
\smallskip
	
(i) We say that $A\cap\Lambda$ contains a {\em $p_i$ corner} if there exist $a,a_i\in A\cap\Lambda$
with $(a-a_i,M)=M/p_i$ satisfying  
$$
%\begin{equation}\label{corner-strict}
A\cap(a*F_j*F_k)= a*F_j,\ \ A\cap(a_i*F_j*F_k)= a_i*F_k.
%\end{equation}
$$

\smallskip
(ii) We say that $A\cap\Lambda$ contains a
{\em $p_i$ extended corner} if there exist $a,a_i\in A\cap\Lambda$
such that $(a-a_i,M)=M/p_i$ and 
\begin{itemize}
\item $A\cap(a*F_j*F_k)$ is $M$-fibered in the $p_j$ direction but not in the $p_k$ direction,
\item $A\cap(a_i*F_j*F_k)$ is $M$-fibered in the $p_k$ direction but not in the $p_j$ direction.
\end{itemize}
	
\end{definition}

\begin{proposition}\label{prop-ecorner}\cite[Proposition 5.5]{LaLo2}
	Let $D=D(M)$, and let $\Lambda$ be a $D$-grid. Assume that $A\cap\Lambda$ is a union of disjoint $M$-fibers, but is not fibered in any direction. Then $A\cap\Lambda$ contains a $p_\nu$ extended corner for some $\nu\in\{i,j,k\}$.
\end{proposition}

This case was resolved entirely in \cite{LaLo2}, including the case when $M$ is even.

\begin{theorem}\label{cornerthm}\cite[Theorem 8.1]{LaLo2}
Assume that $A\oplus B=\ZZ_M$, where $M=p_i^{n_i}p_j^{2}p_k^{2}$, $|A|=p_ip_jp_k$, 
and $\Phi_M|A$. Moreover, assume that $A$ contains a $p_i$ extended corner in the sense of Definition \ref{corner} (ii) on a $D(M)$-grid $\Lambda$.
Then the tiling $A\oplus B=\ZZ_M$ is T2-equivalent to $\Lambda\oplus B=\ZZ_M$ via fiber shifts.
By Corollary \ref{get-standard}, both $A$ and $B$ satisfy (T2).
\end{theorem}

%This case was resolved entirely in \cite{LaLo2} for both odd and even $M$.
%We prove in \cite[Proposition 5.5]{LaLo2} that, in this case, $A\cap\Lambda$ must contain a {\em $p_i$ extended corner} structure, defined in \cite[Definition 5.4]{LaLo2}, for some $\nu\in\{i,j,k\}$. If we further assume that $n_i=n_j=n_k=2$, \cite[Theorem 8.1]{LaLo2} implies that the conclusion of Theorem \ref{unfibered-mainthm} holds in this case. We summarize our results in this regard as follows.
%
%
%
%\begin{theorem}\label{cornerthm}\cite[Proposition 5.5 and Theorem 8.1]{LaLo2}
%Assume that $A\oplus B=\ZZ_M$, where $M=p_i^{2}p_j^{2}p_k^{2}$, $|A|=p_ip_jp_k$, 
%and $\Phi_M|A$. Assume further that there exists a $D(M)$-grid $\Lambda$ such that
%$A\cap\Lambda$ is a union of disjoint $M$-fibers, but is not fibered in any direction. 
%Then the tiling $A\oplus B=\ZZ_M$ is T2-equivalent to $\Lambda\oplus B=\ZZ_M$ via fiber shifts.
%By Corollary \ref{get-standard}, both $A$ and $B$ satisfy (T2).
%\end{theorem}

It remains to consider the case when $A\cap\Lambda$ is not a union of disjoint $M$-fibers. This is the case that was not resolved in \cite{LaLo2} for even $M$, hence we need to do additional work here. However, a version of the preliminary structure results in \cite[Proposition 5.2 and 6.1]{LaLo2} still applies.

\begin{definition}\label{def-db}\cite[Definition 5.1]{LaLo2}
Let $A\subset\ZZ_M$. We say that $A\cap\Lambda$ {\em contains diagonal boxes} if
there are nonempty sets 
$I\subset \ZZ_{p_i}$, $J\subset \ZZ_{p_j}$, $K\subset \ZZ_{p_k}$, such that  
$$
I^c:=\ZZ_{p_i}\setminus I,\  J^c:=\ZZ_{p_j}\setminus J,\  K^c:=\ZZ_{p_k}\setminus K
$$
are also nonempty, and 
$$
%\begin{equation}\label{db-e3}
(I\times J\times K) \cup (I^c\times J^c\times K^c)\subset A\cap\Lambda.
%\end{equation}
$$
\end{definition}

\begin{proposition}\label{db-prop}
Let $A\subset\ZZ_M$ (we do not require $A$ to be a tile). Assume that $\Phi_M\mid A$, and 
that $p_\nu=2$ for some $\nu\in\{i,j,k\}$. Let $\Lambda$ be a $D(M)$-grid such that
$A\cap\Lambda$ is not a union of disjoint $M$-fibers. Then $A\cap\Lambda$ is a union of 
one set of diagonal boxes 
$$
%\begin{equation}\label{db-e19}
A_1=(I_1\times J_1\times K_1) \cup (I_1^c\times J_1^c\times K_1^c),
%\end{equation}
$$
where $I_1\subset \ZZ_{p_i}$, $J_1\subset \ZZ_{p_j}$, $K_1\subset \ZZ_{p_k}$ are non-empty sets such that  
$
I_1^c:=\ZZ_{p_i}\setminus I_1,\  J_1^c:=\ZZ_{p_j}\setminus J_1,\  K_1^c:=\ZZ_{p_k}\setminus K_1
$
are also non-empty, and possibly additional $M$-fibers in one or more directions, disjoint from $A_1$ and from each other. 
\end{proposition}

Note the difference between the statement of Proposition \ref{db-prop} here and Propositions 5.2 and 6.1 in \cite{LaLo2}. In the even case considered here, the complement of $A_1$ in $A\cap\Lambda$ (if nonempty) must be a disjoint union of $M$-fibers. No such claim is made in the odd case in \cite[Proposition 5.2]{LaLo2}. The conclusions of Proposition \ref{db-prop} are identical to those of \cite[Proposition 6.1]{LaLo2} for odd $M$. However, Proposition 6.1 in \cite{LaLo2} requires stronger assumptions: we must assume, in addition, that 
$\{D(M)|m|M\}\not\subset \Div(A\cap\Lambda)$.

\begin{proof}
Let $\Lambda$ be the $D(M)$-grid satisfying the assumptions of the lemma. We first claim that $A\cap\Lambda$ must contain diagonal boxes.
To see this, we follow the proof of \cite[Proposition 5.2]{LaLo2} up to Claim 3, then note that the condition (5.2) of \cite{LaLo2} cannot be satisfied when $M$ is even, hence the proof is complete at that point. 

Let $A_1$ be the set of diagonal boxes thus obtained. By the same argument as in the proof of Proposition 6.1 in \cite{LaLo2}, we see that $(A\cap\Lambda)\setminus A_1$ is a union of disjoint $M$-fibers in one or more directions.

Since our assumptions here differ from those of \cite[Propositions 5.2 and 6.1]{LaLo2}, a few remarks on this are in order. Claims 1-3 in the proof of \cite[Proposition 5.2]{LaLo2} do not require $A$ to be a tile. In \cite[Proposition 6.1]{LaLo2}, we assume in (6.1) that $\{m:D(M)|m|M\}\not\subset\Div(A)$. The purpose of that assumption is to ensure that, in the case under consideration, the proof of 
\cite[Proposition 5.2]{LaLo2} can be halted after Claim 3 (see \cite[(5.3)]{LaLo2}). In our case, the same purpose is served by the
assumption that $M$ is even, so that the divisor assumption is not needed. Other than that, the proof of  \cite[Proposition 6.1]{LaLo2} can be repeated verbatim here.
\end{proof}

\begin{remark}
Proposition \ref{db-prop} is only stated for sets $A\subset \ZZ_M$. However, if we assume instead that $A\in\calm(\ZZ_N)$ for some $N|M$, and that 
\begin{equation}\label{bin_cond}
\bbA^N_N[x]\in\{0,c_0\}\hbox{ for some }c_0\in\NN 
\end{equation}
for all $x\in \Lambda$
(i.e., $A\cap\Lambda$ is a multiset of
constant multiplicity $c_0$ mod $N$), the same argument applies except that the diagonal boxes and fibers in the conclusion also have multiplicity $c_0$.
\end{remark}

%%%%%%%%%%%%%%%%%%%%%%%%%%%%%%%%%%%%%%

\subsection{Special unfibered structures}\label{special-structures}

We will need the structure results of \cite{LaLo2}  for unfibered grids with missing top differences.
We will apply then both to the set $A$ in a tiling $A\oplus B=\ZZ_M$, on the scale $M$, and to subsets of $A$ (which do not need to be tiling complements) on various scales. Therefore, in this section we do not require $A$ to be a tile. We only state Lemma \ref{even_struct} in the even case; for the odd case, see \cite[Section 6.2]{LaLo2}.

\begin{lemma}\label{even_struct}\cite[Lemmas 6.4 and 6.5]{LaLo2}
Let $M=p_i^{n_i}p_j^{n_j}p_k^{n_k}$, with $2|M$. Assume that $N|M$ satisfies $N=M/p_i^{\alpha_i}p_j^{\alpha_j}p_k^{\alpha_k}$ with 
$\alpha_\iota<n_\iota$ for all $\iota\in\{i,j,k\}$, so that any $D(N)$-grid is 3-dimensional. Let $\Lambda$ be a $D(N)$-grid such that $A\cap\Lambda\neq\emptyset$.
Assume that:
\begin{itemize}
\item  $ A\in \mathcal{M}(\ZZ_N)$ satisfies $\Phi_N|A$, 
\item there exists $c_0\in\NN$ such that (\ref{bin_cond}) holds for all $ x\in\Lambda$,
\item $A\cap\Lambda$ is not fibered in any direction,
\end{itemize}
and that
$$
%\begin{equation}\label{div-not-full}
\{m: D(N)|m|N\}\not\subset \Div_N(A\cap\Lambda).
%\end{equation}
$$
Then, possibly after a permutation of $\{i,j,k\}$,
one of the following must hold.

\medskip
(i) 
We have $\{D(N)|m|N\}\setminus \Div_N(A\cap\Lambda)=\{N/p_ip_j\},$
and $A\cap\Lambda$ has a $p_k$ corner structure in the sense of Definition \ref{corner} (i).

\medskip
(ii) We have $p_k=2$ and $N/p_k\not\in \Div_N(A\cap\Lambda)$. Moreover, 
there is a pair of diagonal boxes 
$$A_0=
(I\times J\times K) \cup (I^c\times J^c\times K^c)\subset \Lambda,$$
as in Definition \ref{def-db}, such that for all $z\in A\cap\Lambda$ we have 
$$
\bbA^N_N[z]=c_0 {\bf 1}_{A_0}(z).
$$
Thus, $A\cap\Lambda$ is a set of diagonal boxes with multiplicity $c_0$, with no additional fibers in the same grid.
\end{lemma}

%%%%%%%%%%%%%%%%%%%%%%%%%%%%%%%%%%%%%%%%

\section{Resolving diagonal boxes}\label{res-boxes}

%%%%%%%%%%%%%%%%%%%%%%%%%%%%%%%%%%%%%%%%

Our main theorem for the diagonal boxes case is as follows.

\begin{theorem}\label{db-theorem}
Let $A\oplus B=\ZZ_M$ be a tiling, where $M=p_i^2p_j^2p_k^2$ is even and $|A|=|B|=p_ip_jp_k$, and assume that $\Phi_M\mid A$. 
Let $D=D(M)$, and let $\Lambda$ be a $D$-grid such that $A\cap\Lambda\neq\emptyset$.
Assume further that $A\cap\Lambda$ is not a disjoint union of $M$-fibers.
Then the tiling $A\oplus B=\ZZ_M$ is T2-equivalent to $\Lambda \oplus B=\ZZ_M$ via fiber shifts. Thus $\Lambda$ is a translate of $A^\flat$, and by Corollary \ref{get-standard}, both $A$ and $B$ satisfy (T2).
\end{theorem}

We begin the proof with a few reductions. Let $A$ and $\Lambda$ be as in Theorem \ref{db-theorem}, so that in particular
$A\cap\Lambda$ is not a disjoint union of $M$-fibers. By Proposition \ref{db-prop}, $A\cap\Lambda$ is a union of one pair of diagonal boxes and possibly additional $M$-fibers, disjoint from the boxes and from each other. We fix that pair of diagonal boxes, and denote it 
$$
%\begin{equation}\label{db-defineboxes}
A_1=(I\times J\times K)\cup (I^c\times J^c\times K^c).
%\end{equation}
$$
The proof splits further into cases, depending on the dimensions of boxes and on whether we have
\begin{equation}\label{db-e1000}
\{m:\ D|m|M\}\subset \Div(A\cap\Lambda).
\end{equation}
If (\ref{db-e1000}) fails, by Lemma \ref{even_struct} we must have
\begin{equation}\label{A=A1}
A\cap\Lambda =A_1.
\end{equation}
We note that the corner structure in Lemma \ref{even_struct} (i) is a union of disjoint $M$-fibers, therefore does not fall under the purview of Theorem \ref{db-theorem}. This case was already resolved in \cite[Theorem 8.1]{LaLo2}, stated here as Theorem \ref{cornerthm}.

We claim that it suffices to consider the following sets of assumptions.

\medskip\noindent
{\bf Case (DB1):} The tiling $A\oplus B$ satisfies the assumptions of Theorem \ref{db-theorem}. Additionally, we have $p_i=2$, (\ref{A=A1}) holds, and $\min( |J^c|, |K^c|)\geq 2$.

\medskip\noindent
{\bf Case (DB2):} The tiling $A\oplus B$ satisfies the assumptions of Theorem \ref{db-theorem}.
Additionally, we have $p_i=2$,  (\ref{A=A1}) holds, and $ |J^c|=|K|=1$.

\medskip\noindent
{\bf Case (DB3):} The tiling $A\oplus B$ satisfies the assumptions of Theorem \ref{db-theorem}. Moreover,
$p_i=2$, (\ref{db-e1000}) holds, and $A\cap\Lambda$ is a union of a pair of diagonal boxes $A_1$ and one or more $M$-fibers, disjoint from the boxes and from each other.

\medskip

Indeed, assume that $M$ is even, with $p_i=2$. Without loss of generality, we have $I=\{0\}, I^c=\{1\}$. 
\begin{itemize}
\item If (\ref{db-e1000}) holds, we cannot have $A\cap\Lambda =A_1$, since $M/p_i\not\in\Div(A_1)$. Hence $(A\cap\Lambda) \setminus A_1$ is a nonempty union of $M$-fibers, and we are in the case (DB3).
\item Suppose now that (\ref{db-e1000}) fails. Then (\ref{A=A1}) holds by Lemma \ref{even_struct} (ii). Moreover, 
at least one of the sets from each pair 
$J,J^c$ and $K,K^c$ must have cardinality greater than 1. If $\min( |J|, |K|)\geq 2$ or $\min( |J^c|, |K^c|)\geq 2$, we are in the case (DB1), and if either $ |J^c|=|K|=1$ or $ |J|=|K^c|=1$, we are in the case (DB2),
possibly after relabelling the sets.
\end{itemize}

The case (DB3) was resolved in \cite[Corollary 7.3]{LaLo2}.

\begin{proposition}\label{db3-cor} (Special case of \cite[Corollary 7.3]{LaLo2})
Assume that (DB3) holds. Then the tiling $A\oplus B=\ZZ_M$ is T2-equivalent to $\Lambda\oplus B=\ZZ_M$. Consequently, $A$ and $B$ both satisfy T2.
\end{proposition}

We will proceed to resolve the cases (DB1) and (DB2) in Sections \ref{subsec-db1} and \ref{subsec-db2}, respectively. We will be able to reduce these situations to either (DB3) or an extended corner structure. By Theorem \ref{cornerthm} and Proposition \ref{db3-cor}, this implies the conclusion of the theorem.

%%%%%%%%%%%%%%%%%%%%%%%%%%%%%%%%%%%%%%%%%

\subsection{Case (DB1)}\label{subsec-db1}

Assume that $A\oplus B=\ZZ_M$ is a tiling satisfying the assumptions of Theorem \ref{db-theorem}. Let %$D=D(M)$, and let 
$\Lambda$ be the $D$-grid provided by the assumption of the theorem.
Additionally, we assume that $p_i=2$, (\ref{A=A1}) holds, and
\begin{equation}\label{edb1-e0}
\min( |J^c|, |K^c|)\geq 2.
\end{equation}
In particular, (\ref{A=A1}) implies that
every point $x\in\ZZ_M$ such that
\begin{equation}\label{db1-e1}
x\in I^c\times J\times K
\end{equation} 
satisfies $x\not\in A$ and 
$\bbA_{M/p_j}[x] = \bbA_{M/p_k}[x]=0.$

The proof of Theorem \ref{db-theorem} is based on saturating sets considerations, leading to the T2-equivalence via fiber shifting as stated in the theorem. Compared to the corresponding case with $M$ odd, an additional difficulty is that the assumption (DB1) does not exclude {\it a priori} the possibility that $M/p_i\in\Div(B)$. It turns out, however, that $M/p_i\in\Div(B)$ implies that $B$ must in fact be $M$-fibered in the $p_i$ direction. By Lemma \ref{nopifibers-alt}, this is incompatible with the structure of $A$ provided by (DB1). The remaining cases are covered by Lemma \ref{done-if-jsatura}.

\begin{lemma}\label{edb1-satline}
Assume (DB1), and let $x$ satisfy (\ref{db1-e1}). 
Then for every $b\in B$ we have exactly one of the following:
\begin{equation}\label{db1-jsatura}
A_{x,b}\subset \ell_j(x), \hbox{ with }\bbA_{M/p_j^2}[x]\bbB_{M/p_j^2}[b]=\phi(p_j^2),
\end{equation} 
\begin{equation}\label{db1-ksatura}
A_{x,b}\subset \ell_k(x), \hbox{ with }\bbA_{M/p_k^2}[x]\bbB_{M/p_k^2}[b]=\phi(p_k^2),
\end{equation}
\begin{equation}\label{db1-isatura}
A_{x,b}\subset \ell_i(x) \hbox{ with } \bbA_{M/p_i}[x]\bbB_{M/p_i}[b]=\phi(p_i)
\end{equation}	

Furthermore:
\begin{itemize}
\item If (\ref{db1-jsatura}) holds for some $b\in B$, then the product $\langle \bbA[x],\bbB[b]\rangle$ is saturated by a $(1,2)$-cofiber pair in the $p_j$
direction, with the $A$-cofiber at distance $M/p_j^2$ from $x$ and the $B$-fiber rooted at $b$. The same is true for (\ref{db1-ksatura}), with $k$ and $j$ interchanged.
\item If (\ref{db1-isatura}) holds for some $b\in B$, then the product $\langle \bbA[x],\bbB[b]\rangle$ is saturated by a $(0,1)$-cofiber pair in the $p_i$
direction, i.e. $\bbA_{M/p_i}[x]= \bbB_{M/p_i}[b]=1$.
\end{itemize}
\end{lemma}

\begin{proof}
By (DB1) and (\ref{db1-e1}), we have $A_x\subset \Pi(x,p_i^2)\cup \Pi(a,p_i^2)$, where $a\in A$ satisfies $(a-x,M)=M/p_i$. Moreover, by (\ref{bispan}) we have
$$
A_x\subset \bigcap_{a':(x-a',M)=M/p_jp_k} \Bispan(x,a')= \Pi(x,p_j^2)\cup\Pi(x,p_k^2),
$$
thus 
\begin{equation}\label{fp-2}
A_x\subset (\Pi(x,p_i^2)\cup \Pi(a,p_i^2))\cap(\Pi(x,p_j^2)\cup\Pi(x,p_k^2))=\ell_j(x)\cup\ell_k(x)\cup\ell_j(a)\cup\ell_k(a).
\end{equation}
By (\ref{edb1-e0}), we also have 
\begin{equation}\label{fp-1}
\{M/p_jp_k \mid m \mid M\}\subset\Div (A).
\end{equation}
If $\bbB_{M/p_i}[b]>0$, then $\bbA_{M/p_i}[x]\bbB_{M/p_i}[b]=\phi(p_i)=1$ and we are in the case (\ref{db1-isatura}). Assume therefore, for the rest of the proof, that
\begin{equation}\label{fp-0}
\bbB_{M/p_i}[b]=0.
\end{equation}

We start by claiming that 
\begin{equation}\label{fp-5}
\hbox{ at most one of }A_{x,b}\cap\ell_j(x) \hbox{ and } A_{x,b}\cap\ell_j(a)\hbox{ is non-empty}.
\end{equation}
Indeed, suppose that $A_{x,b}\cap\ell_j(x)\neq\emptyset$. By (\ref{fp-1}), this is only possible if
\begin{equation}\label{fp-3}
\bbA_{M/p_j^2}[x] \bbB_{M/p_j^2}[b]>0.
\end{equation} 
Suppose that we also have $A_{x,b}\cap\ell_j(a)\neq \emptyset$, then
$$
\bbA_{M/p_i}[x] \bbB_{M/p_i}[b] + \bbA_{M/p_ip_j}[x] \bbB_{M/p_ip_j}[b]
+ \bbA_{M/p_ip_j^2}[x] \bbB_{M/p_ip_j^2}[b] >0.
$$
The first term is 0 by (\ref{fp-0}). The second term cannot be non-zero concurrently with (\ref{fp-3}), by Lemma \ref{triangles}. For the third term to be non-zero, we would need $\bbA_{M/p_ip_j^2}[x] = \bbA_{M/p_j^2}[a]>0$,
but that again cannot hold concurrently with (\ref{fp-3}). This proves (\ref{fp-5}).

By Lemma \ref{triangles}, if either of the sets in (\ref{fp-5}) is non-empty,  we have $A_{x,b}\cap(\ell_k(x)\cup\ell_k(a))=\emptyset$. 
Repeating the same argument with $j$ and $k$ interchanged, we get that $A_{x,b}$ is in fact contained in just one of the lines in (\ref{fp-2}). 

It remains to prove that we cannot have $A_{x,b}\subset\ell_j(a)$ or $A_{x,b}\subset\ell_k(a)$.
Suppose that  $A_{x,b}\subset\ell_j(a)$. Then, using (\ref{fp-0}) and that $p_i=2$, we get
\begin{align*}
	1&=\frac{1}{\phi(p_ip_j)}\bbA_{M/p_ip_j}[x|\ell_j(a)]\bbB_{M/p_ip_j}[b]+\frac{1}{\phi(p_ip_j^2)}\bbA_{M/p_ip_j^2}[x|\ell_j(a)]\bbB_{M/p_ip_j^2}[b]\\
	&=\frac{1}{\phi(p_j)}\bbA_{M/p_j}[a]\bbB_{M/p_j}[y]+\frac{1}{\phi(p_j^2)}\bbA_{M/p_j^2}[a]\bbB_{M/p_j^2}[y],
\end{align*}
where $y\in \ZZ_M\setminus B$ is the unique element such that $(b-y,M)=M/p_i$. 
If both expressions in the last sum were nonzero, we would have $\bbA_{M/p_j^2}[a]\bbB_{M/p_j}[y]\bbB_{M/p_j^2}[y]>0$, hence $M/p_j^2\in \Div(A)\cap \Div(B)$, which is a contradiction. This means that either
\begin{equation}\label{isatur1}
\phi(p_j)=\bbA_{M/p_j}[a]\bbB_{M/p_j}[y]
\end{equation}
or
\begin{equation}\label{isatur2}
\phi(p_j^2)=\bbA_{M/p_j^2}[a]\bbB_{M/p_j^2}[y].
\end{equation}

Suppose that (\ref{isatur2}) holds. Since $M/p_j$, $M/p_j^2 \in \Div(A)$, we must have $\bbB_{M/p_j^2}[y]=1$, so that $\bbA_{M/p_j^2}[a]=\phi(p_j^2)$. But this implies $|A\cap\Pi(a,p_k^2)|\geq\bbA_{M/p_j^2}[a]+\bbA_M[a] \geq \phi(p_j^2)+1>p_ip_j$, which contradicts Lemma \ref{planebound}.

Finally, we prove (\ref{isatur1}) cannot hold. Assume the contrary. Then, since $M/p_j\notin \Div(B)$, 
we must have $\bbB_{M/p_j}[y]=1$ and therefore
$$
\bbA_{M/p_j}[a]=\phi(p_j).
$$
However, this contradicts the assumption (\ref{A=A1}).

The proof that we cannot have $A_{x,b}\subset\ell_k(a)$ is identical. The lemma follows.
\end{proof}

\begin{corollary}\label{edb1-unif1}
Assume (DB1), and fix $b\in B$. Suppose that $A_{x,b}\subset\ell_{\nu}(x)$, where $\nu\in\{i,j,k\}$ and $x$ satisfies (\ref{db1-e1}). Then $A_{x',b}\subset\ell_{\nu}(x')$ for all $x'$ satisfying (\ref{db1-e1}).

\end{corollary}	

\begin{proof}
Assume first that $\bbB_{M/p_i}[b]=1$. Since $\bbA_{M/p_i}[x]=1$ for all $x$ satisfying (\ref{db1-e1}), it follows that (\ref{db1-isatura}) holds for all such $x$, as required. 

We now prove the corollary with $\nu=j$. Assume that $A_{x,b}\subset\ell_{j}(x)$, so that (\ref{db1-isatura}) holds. Clearly, 
(\ref{db1-isatura}) also holds for all $x'$ such that $(x'-x,M)=M/p_j$, so that 
\begin{equation}\label{edb1-align}
A_{x_j,b}\subset\ell_{j}(x_j)\ \ \forall x_j\in x*F_j.
\end{equation}
We also note that
\begin{equation}\label{edb1-Adiv}
M/p_ip_j^2\in \Div(\ell_j(x),a)\subset \Div(A)
\end{equation}
where $a$ satisfies $(x-a,M)=M/p_i$.

If $|K|=1$, the above argument already proves the lemma. 
Assume now that $|K|>1$, and let $x''$ satisfy (\ref{db1-e1}) and $(x''-x_j,M)=M/p_k$ for some $x_j\in x*F_j$.
By (\ref{edb1-align}), we have $\bbB_{M/p_j^2}[b]>0$, and by (\ref{edb1-Adiv}) we have $M/p_ip_j^2\notin \Div(B)$. It follows that  $\bbB_{M/p_i}[b]=0$, so that $A_{x'',b}\cap\ell_i(x'')=\emptyset$. Suppose now that $A_{x'',b}\subset \ell_k(x'')$. By the same argument as above with $j$ and $k$ interchanged, this would imply that $A_{x_j,b}\subset\ell_{k}(x_j)$, which would contradict (\ref{edb1-align}). Hence $A_{x'',b}\subset \ell_j(x'')$, and as above, the same is true for all $x''_j$ with $(x''_j-x'',M)=M/p_j$. This proves the lemma for $\nu=j$. 

The proof for $\nu=k$ is identical, interchanging $k$ and $j$.
\end{proof}

\begin{lemma}\label{fp-unif2}
Assume that (DB1) holds.
Let $x$ satisfy (\ref{db1-e1}). If there exists $b\in B$ for which (\ref{db1-jsatura}) holds, then (\ref{db1-ksatura}) does not hold for any $b'\in B$. The same holds with (\ref{db1-jsatura}) and (\ref{db1-ksatura}) interchanged. 
\end{lemma}

\begin{proof}
Assume, by contradiction, that there exist $b_j,b_k\in B$ such that $b_j$ satisfies (\ref{db1-jsatura}) and $b_k$ satisfies (\ref{db1-ksatura}). It follows by Corollary \ref{edb1-unif1} that (\ref{db1-jsatura}) with $b=b_j$ and (\ref{db1-ksatura}) with $b=b_k$ hold for all 
$x'\in I^c\times J\times K$. Thus $A\cap\Pi(x,p_i^2)$ contains $ I^c\times J^c\times K^c$, as well as $|K|$ $M$-fibers in the $p_j$ direction and $|J|$ $M$-fibers in the $p_k$ direction, all disjoint from each other. We get
\begin{align*}
|A\cap \Pi(x,p_i^2)|&\geq (p_j-|J|)(p_k-|K|)+p_k|J|+p_j|K|\\
&=p_jp_k + |J| |K| >p_jp_k.
\end{align*}	
This contradicts Lemma \ref{planebound}, and completes the proof of the lemma.
\end{proof}

\begin{lemma}\label{db1-saturating}
Assume (DB1) holds, and that (\ref{db1-jsatura}) holds for some $x$ satisfying (\ref{db1-e1}) and $b\in B$. Then, for every $x'\in I^c\times J\times K^c$,
\begin{equation}\label{db1-isatura2}
A_{x',b}\subset\ell_i(x'), \hbox{ with } \bbA_{M/p_i^2}[x']=p_i, \ \bbB_{M/p_i^2}[b]=\phi(p_i).
\end{equation}
Moreover, (\ref{db1-isatura2}) holds for all $b\in B$, and $B$ is $N_i$-fibered in the $p_i$ direction.

If (\ref{db1-ksatura}) holds for some $x$ satisfying (\ref{db1-e1}) and $b\in B$, then the same conclusion is true with $j$ and $k$ interchanged.
\end{lemma}

\begin{proof}
Suppose that (\ref{db1-jsatura}) holds for some $x$ satisfying (\ref{db1-e1}), and for some $b\in B$. By Corollary \ref{edb1-unif1} and Lemma \ref{fp-unif2}, (\ref{db1-jsatura}) holds for the same $b\in B$ with $x$ replaced by any other element $x''$ satisfying (\ref{db1-e1}).

Let $x'\in I^c\times J\times K^c$. By the above argument, we may assume that
$$
%\begin{equation}\label{x'loc}
(x-x',M)=M/p_k.
%\end{equation}
$$
Denote by $a$ the unique element in $I\times J\times K$ satisfying 
$$
%\begin{equation}\label{aloc}
(x-a,M)=M/p_i,(x'-a,M)=M/p_ip_k.
%\end{equation}
$$

\noindent
{\bf Claim 1.} Let $b\in B$. If $b$ satisfies either $A_{x,b}\subset \ell_i(x)$ or (\ref{db1-jsatura}), then
\begin{equation}\label{db1-satsetloc}
A_{x',b}\subset \ell_i(x')\cup \ell_i(x).
\end{equation}

\begin{proof}
Consider the saturating set $A_{x',b}$. By (\ref{edb1-e0}) we have $\bbA_{M/p_j}[x']\geq 2$. By (\ref{bispan}) 
$$
A_{x',b}\subset \bigcap_{a'\in A, (x'-a',M)=M/p_j} (\Pi(x',p_j^2)\cup \Pi(a',p_j^2))=\Pi(x',p_j^2).
$$
Additionally, applying (\ref{setplusspan}) to $x$ and $x'$, and using that $A_{x,b}\subset\ell_\mu(x)$ for some $\mu\in \{i,j\}$, we get
$$
A_{x',b}\subset \Pi(x',p_k^2)\cup \Pi(x,p_k^2).
$$
Hence (\ref{db1-satsetloc}) follows.
\end{proof}

\noindent
{\bf Claim 2.} Let $b\in B$ satisfy (\ref{db1-jsatura}). Then 
\begin{equation}\label{db1-isatset}
A_{x',b}\subset \ell_i(x').
\end{equation}

\begin{proof}
Clearly, if $|K|>1$ then $A_{x',b}\cap \ell_i(x)\neq \emptyset$ implies $A_{x'',b}\cap \ell_i(x)\neq \emptyset$ for all $x''$ satisfying (\ref{db1-e1}) with $(x-x'',M)=M/p_k$. But this would contradict the fact that $x''$ satisfies (\ref{db1-jsatura}), as follows from Corollary \ref{edb1-unif1}. We are left with the case $|K|=1$. Assume, for contradiction, that 
\begin{equation}\label{Ax'}
	A_{x',b}\cap \ell_i(x)\neq \emptyset.
\end{equation}
We first claim that (\ref{Ax'}) implies 
\begin{equation}\label{Ax'only}
	A_{x',b}\subset \ell_i(x).
\end{equation}
Indeed, (\ref{Ax'}) implies that either $\bbA_{M/p_i^2p_k}[x'|\ell_i(x)]=\bbA_{M/p_i^2}[a]>0$, or else $a\in A_{x',b}$ and $\bbB_{M/p_ip_k}[b]>0$. If the former holds, then $M/p_i^2\in \Div(A)$; since $\bbA_{M/p_i}[x']=0$ by (\ref{A=A1}), this implies that $A_{x',b}\cap \ell_i(x')=\emptyset$ as claimed. Assume now that $a\in A_{x',b}$. Then the failure of (\ref{Ax'only}) would imply that $M/p_i^2p_k\in \Div(a,\ell_i(x))\subset \Div(A)$; on the other hand, we have $M/p_i^2p_k\in \Div(\ell_i(b), b'')\subset \Div(B)$, where $(b-b'',M)=M/p_ip_k$. The claim follows.

With (\ref{Ax'only}) in place, we now have
\begin{equation}\label{ell_i(x)sat}
	1=\frac{1}{\phi(p_k)}\bbA_{M/p_ip_k}[x'|\ell_i(x)]\bbB_{M/p_ip_k}[b]+\frac{1}{\phi(p_i^2p_k)}\bbA_{M/p_i^2p_k}[x'|\ell_i(x)]\bbB_{M/p_i^2p_k}[b].
\end{equation}
By (\ref{A=A1}), we have $\bbA_{M/p_ip_k}[x'|\ell_i(x)]=1$, and since
\begin{equation}\label{AM/p_kdiv}
	M/p_k\in \Div(A),
\end{equation}
we have $\bbB_{M/p_ip_k}[b]\leq 1$. It follows that 
\begin{equation}\label{ell_i(x)sat2}
	\bbA_{M/p_i^2p_k}[x'|\ell_i(x)]\bbB_{M/p_i^2p_k}[b]>0,
\end{equation}
with 
\begin{equation}\label{ell_i(x)satbounds}
	\bbA_{M/p_i^2p_k}[x'|\ell_i(x)]\leq \phi(p_i^2) \text{ and } \bbB_{M/p_i^2p_k}[b]\leq 2,
\end{equation}
where the latter follows from the fact that $p_i=2$ and (\ref{AM/p_kdiv}). Plugging in these restrictions to (\ref{ell_i(x)sat}), we get $\phi(p_k)\leq 3$. 
%We also have
%\begin{equation}\label{M/p_i^2divA}
%	M/p_i^2\in Div(A)
%\end{equation}

It is, therefore, left to consider the case when
$$
p_k=3\hbox{ and }|K|=1.
$$
 We then have
\begin{align*}
	|A\cap \Pi(x,p_k)| & \geq |A_1|+p_j\cdot |K|\\
	&= p_jp_k-p_k|J|+2|J||K|\\
	&=3p_j-|J|\\
	&>p_ip_j
\end{align*}
and by Corollary \ref{planegrid} 
\begin{equation}\label{cycA}
	\Phi_{p_k^2}|A.
\end{equation}

If $\bbB_{M/p_ip_k}[b]=0$, it follows from (\ref{ell_i(x)satbounds}) with $p_k=3$ that $\bbA_{M/p_i^2p_k}[x'|\ell_i(x)]=\phi(p_i^2)$ as well. In particular, $M/p_i\in \Div(A)$, hence (\ref{db1-isatura}) does not hold for any $b\in B$. Therefore (\ref{db1-jsatura}) must hold for both $b_1,b_2\in B$ with $(b-b_\mu,M)=M/p_i^2p_k$, for $\mu=1,2$. Since (\ref{db1-jsatura}) also holds for $b$, and $b_1,b_2\in\Pi(b,p_k)$,  we get
$$
|B\cap \Pi(b,p_k)|\geq 3p_j>p_ip_j.
$$
Applying Corollary \ref{planegrid} to $B$, we get $\Phi_{p_k^2}|B$, contradicting (\ref{cycA}).

When $\bbB_{M/p_ip_k}[b]=1$, we denote $b'\in B$ with $(b-b',M)=M/p_ip_k$. From (\ref{AM/p_kdiv}), it is evident that $\bbB_{M/p_i}[b']=0$ and so $b'$ cannot satisfy (\ref{db1-isatura}). It follows that $b'$ satisfies (\ref{db1-jsatura}). In addition, by (\ref{ell_i(x)sat2}), there must also be $b''\in B$ with $(b-b'',M)=M/p_i^2p_k$. Combining all of the above, we have
$$
|B\cap \Pi(b,p_k)|\geq 2p_j+1>p_ip_j.
$$
By Corollary \ref{planegrid}, we get $\Phi_{p_k^2}|B$ again, contradicting (\ref{cycA}).
\end{proof}

\noindent
{\bf Claim 3.} Let $b\in B$ satisfy (\ref{db1-jsatura}). Then (\ref{db1-isatura2}) holds.

\begin{proof}
By Claim 2, (\ref{db1-isatset}) holds. We need to prove that
\begin{equation}\label{satsetsol}
	\bbA_{M/p_i^2}[x']=p_i, \ \bbB_{M/p_i^2}[b]=\phi(p_i).
\end{equation}
Assume, by contradiction, that (\ref{satsetsol}) does not hold. Then
\begin{equation}\label{db1-falseassu}
\bbA_{M/p_i^2}[x']=\phi(p_i), \ \bbB_{M/p_i^2}[b]=p_i
\end{equation}
and 
\begin{equation}\label{satsetBdiv}
M/p_i,M/p_i^2\in \Div(B).
\end{equation} 
Let $b'\in B$ with $(b-b',M)=M/p_i^2$ so that 
\begin{equation}\label{db1-satsetb'}
\bbB_{M/p_i}[b']=\bbB_{M/p_i^2}[b']=1.
\end{equation}
We show that the product $\langle\bbA[x'],\bbB[b']\rangle$ cannot be saturated.

Consider the saturating set $A_{x',b'}$. Since $\bbA_{M/p_i}[x]\bbB_{M/p_i}[b']=1=\phi(p_i)$, we have $A_{x,b'}\subset \ell_i(x)$, so that the assumptions of Claim 1 are satisfied with $b$ replaced by $b'$. It follows that 
$$
A_{x',b'}\subset \ell_i(x')\cup \ell_i(x).
$$
By (\ref{satsetBdiv}), if $A_{x',b'}\cap\ell_i(x)$ is nonempty, then $a\in A_{x',b'}$, hence  $\bbB_{M/p_ip_k}[b']>0$; but this is in contradiction with (\ref{AM/p_kdiv}) and (\ref{db1-satsetb'}). It follows that $A_{x',b'}\subset\ell_i(x')$. By (\ref{A=A1}), (\ref{db1-falseassu}) and (\ref{db1-satsetb'}) we get
$$
1=\frac{1}{\phi(p_i^2)}\bbA_{M/p_i^2}[x']\bbB_{M/p_i^2}[b']=1/2.
$$
This contradiction proves (\ref{satsetsol}). 
\end{proof}

\noindent
{\bf Claim 4.} Suppose there exists $b\in B$ satisfying (\ref{db1-jsatura}). Then (\ref{satsetsol}) holds for all $b\in B$. Consequently $B$ is $N_i$-fibered in the $p_i$ direction. 

\begin{proof}
By (\ref{satsetsol}), we have $M/p_i\in \Div(A)$ and $M/p_i^2\in \Div(B)$. We conclude that no element $b'\in B$ satisfies (\ref{db1-isatura}), so that (\ref{satsetsol}) holds for all $b\in B$.
\end{proof}

This ends the proof of Lemma \ref{db1-saturating}.
\end{proof}

\begin{lemma}\label{done-if-jsatura}
Assume that (DB1) holds, and that (\ref{db1-jsatura}) holds for some $x$ satisfying (\ref{db1-e1}) and some $b\in B$.  Then the tiling $A\oplus B=\ZZ_M$ is T2-equivalent to $\Lambda\oplus B=\ZZ_M$ via fiber shifts. Therefore, the conclusions of Theorem \ref{db-theorem} are satisfied.
\end{lemma}

\begin{proof}
Assume that there exists $x\in I^c\times J\times K$ satisfying (\ref{db1-e1}) and (\ref{db1-jsatura}) for some $b\in B$. 
By Lemma \ref{db1-saturating}, for every element $x'\in I^c\times J\times K^c$ the line $\ell_i(x)$ contains an $M$-fiber in the $p_i$ direction, at distance $M/p_i^2$ from $x'$, and $B$ is $N_i$-fibered in the $p_i$ direction. Let $A'$ be the set obtained from $A$ by shifting all such $M$-fibers, so that after these shifts we have $\{0,1\}\times J\times K^c\subset A'$. By Lemma \ref{fibershift}, $A'$ is T2-equivalent to $A$. 
Moreover, $A'\cap \Lambda$ contains both the original set of diagonal boxes $A_1$ and at least two (since $|K^c|\geq 2$) additional $M$-fibers in the $p_i$ direction, disjoint from $A_1$. Therefore the tiling $A'\oplus B=\ZZ_M$ satisfies the assumption (DB3). The conclusion follows from Proposition \ref{db3-cor}.
\end{proof}

Since the argument above is symmetric with respect to interchanging $j$ and $k$, the same conclusion holds if (\ref{db1-ksatura}) holds for some $x$ satisfying (\ref{db1-e1}) and some $b\in B$.
It remains to consider the case when (\ref{db1-isatura}) holds for all $x$ satisfying (\ref{db1-e1}) and all $b\in B$. This, however, means that $B$ is $M$-fibered in the $p_i$ direction. We could apply the slab reduction to $B$ and conclude that (T2) holds in this case as well. However, in order to prove the full T2-equivalence in Theorem \ref{unfibered-mainthm}, we need to prove that the hypothetical tiling obtained from the slab reduction cannot actually exist, so that this case can be eliminated altogether.

\begin{lemma}\label{nopifibers-alt}
Assume (DB1). Then $B$ cannot be $M$-fibered in the $p_i$ direction.
\end{lemma}

\begin{proof}
Assume, for contradiction, that $B$ is $M$-fibered in the $p_i$ direction. We apply Theorem 
\ref{subtile} and Corollary \ref{slab-reduction} to get a tiling $A'\oplus B'=\ZZ_{M'}$, where $M'=M/p_i$, $A'\equiv A$ mod $M'$, and $B'$ is obtained from $B$ by selecting one element from each $M$-fiber in the $p_i$ direction and then reducing mod $M'$. In particular, $|B'|=p_jp_k$. We also use $\bbA'$ and $\bbB'$ to denote the boxes associated with $A'$ and $B'$ in $\ZZ_{M'}$.

Let $\Lambda'=\{x'\in\ZZ_{M'}:\ p_i|a-x'\}$ for any $a\in A'_1$, where $A'_1$ is the reduction of $A_1$ mod $M'$. (Note that the images of both of the diagonal boxes in $A_1$ lie in one plane $\Pi(a,p_i)$ mod $M'$ after this reduction.)
Let $x_j' \in \Lambda'\setminus A'_1$ so that $x'_j\equiv x_j$ mod $M'$ for some $x_j\in I\times J^c\times K$. By
(\ref{A=A1}), we have $x'_j\not\in A'$, and
\begin{equation}\label{modM'}
\bbA'_{M'/p_j}[x'_j]\geq 1,\ \ \bbA'_{M'/p_k}[x'_j]\geq 2.
\end{equation}
In this proof, we will use $A'_{x_j}$ and $A'_{x_j,b}$ to denote saturating sets with respect to the tiling $A'\oplus B'=\ZZ_{M'}$. We claim that
\begin{equation}\label{satmodM'}
A'_{x_j}\subset \ell_i(x_j), \hbox{ with }\bbB'_{M'/p_i}[b]=1\hbox{ for all }b\in B'.
\end{equation}
Hence $B'$ must be $M'$-fibered in the $p_i$ direction, contradicting the fact that $|B'|=p_jp_k$.

We now prove (\ref{satmodM'}). If $\bbA'_{M'/p_j}[x'_j]\geq 2$, it follows immediately from (\ref{modM'}) and (\ref{bispan}) that 
$$
A'_{x_j}\subset \bigcap_{a'\in A', (x_j-a',M')=M'/p_\nu, \nu\in\{j,k\}} (\Pi(x_j,p_\nu^2)\cup \Pi(a',p_\nu^2)),
$$
which proves (\ref{satmodM'}).

If $\bbA'_{M'/p_j}[x'_j]=1$, we have
$$
A'_x\subset \ell_i(x_j)\cup\ell_i(a_j),
$$
where $a_j$ is the unique element of $A'$ with $(a_j-x_j,M')=M'/p_j$. 
Let $b\in B'$. Since $M'/p_j\in\Div_{M'}(A')$, we have $\bbB'_{M'/p_j}[b]=0$ and $\bbB'_{M'/p_ip_j}[b]\leq 1$, so that
$$
\langle \bbA'[x'_j|\ell_i(a_j)],\bbB'[b]\rangle \leq \frac{1}{\phi(p_ip_j)}<1.
$$
Hence $A'_{x_j,b}\cap\ell_i(x_j)\neq\emptyset$. But then there must be a $b'\in B'$ such that $(b-b',M')=M'/p_i$. Since $b\in B'$ was arbitrary, the claim follows.
\end{proof}

%%%%%%%%%%%%%%%%%%%%%%%%%%%%%%%%%%%%%%%%%% 
 
 \subsection{Case (DB2)}\label{subsec-db2}

We continue to assume that the tiling $A\oplus B=\ZZ_M$ satisfies the conditions of Theorem \ref{db-theorem}. Let %$D=D(M)$, and let 
$\Lambda$ be the $D$-grid as in the statement of the theorem.	
Additionally, we assume that $p_i=2$ and that (\ref{A=A1}) holds with
\begin{equation}\label{edb2-ev0}
|J^c|=|K|=1.
\end{equation}
Let $x_j,x_k\in \ZZ_M$ such that $I\times J^c\times K=\{x_j\}$ and  $I^c\times J^c\times K=\{x_k\}$. By (\ref{A=A1}) and
(\ref{edb2-ev0}), we have $x_j,x_k\not\in A$ and 
\begin{equation}\label{edb2-str}
(x_j-x_k,M)=M/p_i \hbox{ and } \frac{1}{\phi(p_j)}\bbA_{M/p_j}[x_j]=\frac{1}{\phi(p_k)}\bbA_{M/p_k}[x_k]=1.
\end{equation}

From this point, the proof is organized as follows. We prove in Lemma \ref{edb2-lines} that for each $b\in B$, the saturating set $A_{x_j,b}$ is contained in one of the lines $\ell_i(x_j), \ell_k(x_j), \ell_k(x_k)$, and that the same is true with $j$ and $k$ interchanged. We then ask which combinations of these lines can work for both of the points $x_j,x_k$ simultaneously. 
Define
%\begin{equation}\label{edb2-Bbreakdown}
\begin{align*}
&B_0:=\{b\in B:\,A_{x_j,b}, A_{x_k,b}\subset \ell_i(x_j)\},
\\
&B_1:=\{b\in B:\,A_{x_j,b}\subset \ell_k(x_j) \hbox{ and } A_{x_k,b}\subset\ell_j(x_k)\}.
\end{align*}
%\end{equation}
(Note that $\ell_i(x_j)=\ell_i(x_k).$) 
We will prove that
\begin{equation}\label{Bdecomp}
B=B_0 \hbox{ or } B= B_1.
\end{equation}
In both cases, we will be able to use fiber shifts to prove Theorem \ref{db-theorem}.

\begin{lemma}\label{edb2-lines}
Assume (DB2). Then, for each $b\in B$, $A_{x_j,b}$ is contained in exactly one of the lines $\ell_i(x_j), \ell_k(x_j), \ell_k(x_k)$. The same statement holds with $j$ and $k$ interchanged. 
In particular, the sets $B_0$ and $B_1$ are disjoint.
Moreover, for any $b\in B$,
\begin{equation}\label{equiv-iline}
A_{x_j,b}\subset \ell_i(x_j) \ \Leftrightarrow \ A_{x_k,b}\subset \ell_i(x_j).
\end{equation}
\end{lemma}

\begin{proof}
Fix $b\in B$. By (\ref{bispan}) we have
$$
A_{x_j,b}\subset\ell_i(x_j)\cup \ell_k(x_j)\cup \ell_k(x_k).
$$
We prove first that $A_{x_j,b}$ cannot intersect both $\ell_k(x_j)$ and $\ell_k(x_k)$ simultaneously. Suppose that 
%\begin{equation}\label{edb2-satur}
$$
A_{x_j,b}\cap \ell_k(x_j)\neq \emptyset.
$$
%\end{equation}
Since $M/p_k\in \Div(A)$, we must have 
\begin{equation}\label{edb2-div2}
\bbA_{M/p_k^2}[x_j]\bbB_{M/p_k^2}[b]>0.
\end{equation}
Assume furthermore that $A_{x_j,b}\cap\ell_k(x_k)\neq \emptyset$. If 
$$\bbA_{M/p_ip_k^2}[x_j|\ell_k(x_k)]\bbB_{M/p_ip_k^2}[b]>0,$$
then this together with (\ref{edb2-str}) would imply that $M/p_k^2\in \Div(A\cap\ell_k(x_k))$, contradicting (\ref{edb2-div2}). Hence $\bbA_{M/p_ip_k}[x_j|\ell_k(x_k)]\bbB_{M/p_ip_k}[b]>0.$
But by Lemma \ref{triangles}, this cannot hold concurrently with (\ref{edb2-div2}). 
	
Next, suppose that $A_{x_j,b}\cap\ell_i(x_j)\neq\emptyset$.
Since $x_k\not\in A$, 
we must have $\bbA_{M/p_i^2}[x_j]\bbB_{M/p_i^2}[b]>0$. By Lemma \ref{triangles}, this implies that
$$A_{x_j,b}\cap(\ell_k(x_j)\cup\ell_k(x_k))=\emptyset.$$
Hence $A_{x_j,b}\subset\ell_i(x_j)$, and the first conclusion of the lemma follows.

Finally, we prove (\ref{equiv-iline}). 
Suppose that $A_{x_k,b_k}\subset \ell_i(x_k)$. Then 
$$
\phi(p_i^2)=\bbA_{M/p_i^2}[x_k]\bbB_{M/p_i^2}[b_k]=\bbA_{M/p_i^2}[x_j]\bbB_{M/p_i^2}[b_k],
$$
hence $A_{x_j,b_k}\subset\ell_i(x_j)$ as claimed. The same argument works in the other direction.
\end{proof}

\begin{lemma}\label{edb2-ell_jx_j}
Assume that (DB2) holds, and let $b\in B$. Then:
\begin{itemize}
\item[(i)] if $A_{x_j,b}\subset \ell_i(x_j)$, then
\begin{equation}\label{edb2-p_isat}
\bbA_{M/p_i^2}[x_j]\bbB_{M/p_i^2}[b]=\phi(p_i^2),
\end{equation}
\item[(ii)] if $A_{x_j,b}\subset \ell_k(x_j)$, then 
\begin{equation}\label{edb2-p_ksat}
\bbA_{M/p_k^2}[x_j]=p_k,\ \bbB_{M/p_k^2}[b_k]=\phi(p_k),
\end{equation}
\item[(iii)] if $A_{x_j,b}\subset \ell_k(x_k)$, then

\begin{equation}\label{edb2-Bell_jx_j} 
\bbB_{M/p_ip_k}[b]=1.
\end{equation}
\end{itemize}
The same statements hold with $j$ and $k$ interchanged.
\end{lemma}
\begin{proof}
The first statement is true since $x_k\not\in A$. For (ii), 
since $M/p_k\in\Div(A)$, we must have
$\bbA_{M/p_k^2}[x_j]\bbB_{M/p_k^2}[b]=\phi(p_k^2)$, implying (\ref{edb2-p_ksat}). 

We now prove (iii).
The assumption that $A_{x_j,b}\subset \ell_k(x_k)$ implies 
$$
1=\frac{1}{\phi(p_ip_k)}\bbA_{M/p_ip_k}[x_j|\ell_k(x_k)]\bbB_{M/p_ip_k}[b]
+\frac{1}{\phi(p_ip_k^2)}\bbA_{M/p_ip_k^2}[x_j|\ell_k(x_k)]\bbB_{M/p_ip_k^2}[b].
$$
On the other hand, since $p_i=2$, we have by (\ref{edb2-str}) 
$$
\bbA_{M/p_ip_k}[x_j|\ell_k(x_k)]=\bbA_{M/p_k}[x_k]=\phi(p_k)=\phi(p_ip_k).
$$
If $\bbB_{M/p_ip_k}[b]>0$, this implies (\ref{edb2-Bell_jx_j}) and we are done. Suppose now that
$\bbB_{M/p_ip_k}[b]=0$. Then 
$$
\bbA_{M/p_ip_k^2}[x_j|\ell_k(x_k)]\bbB_{M/p_ip_k^2}[b]=\phi(p_ip_k^2)=\phi(p_k^2),
$$
and by (\ref{edb2-str}) we have $M/p_k,M/p_k^2\in \Div(A)$. Hence $\bbB_{M/p_ip_k^2}[b]=1$, so that
$$
\bbA_{M/p_ip_k^2}[x_j|\ell_k(x_k)]=\bbA_{M/p_k^2}[x_k]=\phi(p_k^2).
$$
But now 
$$
|A\cap\Pi(x_k,p_j^2)|\geq \bbA_{M/p_k^2}[x_k]+\bbA_{M/p_k}[x_k]=\phi(p_k^2)+\phi(p_k)=p_k^2-1>p_ip_k,
$$
contradicting Lemma \ref{planebound}. 
\end{proof}

\begin{lemma}\label{edb2-unif1}
Assume that (DB2) holds, and that there exist $b_i,b_k\in B$ such that $A_{x_j,b_i}\subset \ell_i(x_j)$ and $A_{x_j,b_k}\subset \ell_k(x_j)$. Then we have the following:
\begin{itemize}
\item[(i)] $\bbA_{M/p_i^2}[x_j]=1$ and $\bbB_{M/p_i^2}[b]=p_i$ (in particular, $M/p_i\in\Div(B)$),
\item[(ii)] $\Phi_{p_j^2}|A$, $A\subset \Pi(x_j,p_j)$, and $A_{x_k}\cap \ell_j(x_k)=\emptyset$,
\item[(iii)] $A_{x_k,b_k}\subset \ell_j(x_j)$, with $\bbB_{M/p_ip_j}[b_k]=1$.
\end{itemize}
The same statement holds with $j$ and $k$ interchanged.
\end{lemma}

\begin{proof}
Let $b_i,b_k$ be as in the assumptions of the lemma. Then $b_i$ satisfies (\ref{edb2-p_isat}), and 
$b_k$ satisfies (\ref{edb2-p_ksat}).
Hence
\begin{align*}
|A\cap\Pi(x_j,p_j^2)|&\geq \bbA_{M/p_k}[x_k]+\bbA_{M/p_k^2}[x_j]+\bbA_{M/p_i^2}[x_j]\\
&=\phi(p_k)+p_k+\bbA_{M/p_i^2}[x_j]\\
&=p_ip_k -1 + \bbA_{M/p_i^2}[x_j].
\end{align*}
By Lemma \ref{planebound}, the last line must be less than or equal to $p_ip_k$. Hence $\bbA_{M/p_i^2}[x_j]=1$, and 
by (\ref{edb2-p_isat}) we must have
$\bbB_{M/p_i^2}[b_i]=\phi(p_i^2)=2$, proving (i). Furthermore, it follows that 
$$
%\begin{equation}\label{edb2-p_j^2pi}
|A\cap\Pi(x_j,p_j^2)|=p_ip_k.
%\end{equation}
$$
Hence
$$
|A\cap\Pi(x_j,p_j)|\geq |A\cap\Pi(x_j,p_j^2)|+ \bbA_{M/p_j}[x_j]>p_ip_k. 
$$
By Corollary \ref{planegrid}, we have $\Phi_{p_j^2}|A$ and $A\subset\Pi(x_j,p_j)$. In particular, $M/p_j^2\not\in\Div(x_\nu, A)$ for $\nu\in\{j,k\}$. Since we also have $(x_k*F_j)\cap A=\emptyset$, it follows that $A_{x_k}\cap \ell_j(x_k)=\emptyset$. This proves (ii).

We now prove (iii). By Lemma \ref{edb2-lines} with $j$ and $k$ interchanged, $A_{x_k,b_k}$ must be contained in one of the lines $\ell_i(x_k), \ell_j(x_j), \ell_j(x_k)$. 
\begin{itemize}
\item Suppose that $A_{x_k,b_k}\subset \ell_i(x_k)$. By (\ref{equiv-iline}), this would imply 
$A_{x_j,b_k}\subset\ell_i(x_j)$, contradicting the choice of $b_k$.
\item Suppose now that $A_{x_k,b_k}\subset \ell_j(x_k)$. Then (\ref{edb2-p_ksat}) holds with $j$ and $k$ interchanged, and in particular $\bbA_{M/p_j^2}[x_k]>0$, contradicting part (ii) of the lemma.
\end{itemize}
Hence $A_{x_k,b_k}\subset \ell_j(x_k)$, and the last part follows from (\ref{edb2-Bell_jx_j}) with $j$ and $k$ interchanged.
\end{proof}

\begin{lemma}\label{edb2-nodiag}
Assume (DB2). For every $b\in B$, we have the following:

\begin{itemize}
\item[(i)] $\bbB_{M/p_ip_j}[b]\bbB_{M/p_ip_k}[b]=0$,
\item[(ii)] if $A_{x_j,b}\subset \ell_k(x_k)$, then $A_{x_k,b}\subset \ell_j(x_k)$.
\end{itemize}
The same holds with $j$ and $k$ interchanged. 
\end{lemma}

\begin{proof}
Assume, by contradiction, that $\bbB_{M/p_ip_j}[b]\bbB_{M/p_ip_k}[b]>0$ for some $b\in B$. 
Let $y\in \ZZ_{M}$ be the unique point with $(b-y,M)=M/p_i$. Then there exist $b_j,b_k\in B$ such that 
$(y-b_\nu,M)=M/p_\nu$ for $\nu=j,k$. 

Let $a\in A$, and consider the saturating set $B_{y,a}$. Then, by (\ref{bispan}) 
$$
B_{y,a}\subset \bigcap_{b'\in B, (y-b',M)=M/p_\nu,\nu\in\{i,j,k\}} (\Pi(y,p_j^2)\cup \Pi(b',p_\nu^2))
$$ 
and thus contained in the vertices of the $M$-cuboid with vertices at $b,b_j,b_k$ and $y$. By (DB2), we have 
\begin{equation}\label{divB-diag}
\{M/p_j,M/p_k,M/p_ip_jp_k\}\subset \Div(A),
\end{equation}
hence no other elements of $B$ are permitted at the vertices of that cuboid except possibly at the point $u$ with $(u-b,M)=M/p_jp_k$. We consider two cases.

\medskip\noindent
{\bf Case 1.} Suppose that $u\not\in B$.
By Lemma \ref{triangles}, $B_{y,a}$ must consist of just one of the points $b,b_j,b_k$. The latter, in turn, implies that $a$ must be contained in an $M$-fiber in some direction, which is clearly false for any $a\in A\cap\Lambda$.

\medskip\noindent
{\bf Case 2.}
Assume now that $u\in B$. In this case, we will use a splitting argument. By translational invariance, we may assume that $x_j=b=0$, so that $\Lambda=\Lambda(0,D)$, $x_k=M/p_i$, and $A_1=(F_j\setminus\{0\}) \cup ((x_k*F_k)\setminus\{x_k\})$. Replacing $B$ by $rB$ for some $r\in R$ if necessary (see \cite[Lemma 3.6]{LaLo3}), we may assume that
$$
b_j=M/p_i+M/p_j,\ b_k=M/p_i+M/p_k, \ u=M/p_j+M/p_k.
$$
By (\ref{divB-diag}), no other points of $\Lambda$ are permitted in $B$, so that \begin{equation}\label{nootherb}
B\cap\Lambda=\{0,b_j,b_k,u\}.
\end{equation}

Let $z=\mu M/p_j + \nu M/p_k$ for some $\mu\in\{2,3,\dots,p_j-1\}$ and $\nu\in\{2,3,\dots,p_k-1\}$, and let 
$a'\in A,b'\in B$ be the elements such that $a'+b'=z$. We first consider splitting for the fiber $z*F_j$. We have
\begin{align*}
&\nu M/p_k = b_k+ (M/p_i+(\nu-1)M/p_k), \hbox{ with }
 b_k\in B,\ M/p_i+(\nu-1)M/p_k\in  A, 
\\
&\nu M/p_k+ M/p_j = b_j + (M/p_i+ \nu M/p_k), \hbox{ with }
 b_j \in B,\ M/p_i+\nu M/p_k\in  A.
\end{align*}
Hence $z*F_j$ splits with parity $(A,B)$, so that 
\begin{equation}\label{pinnedplanej}
a'\in\Pi(0,p_j^2),\ b'\in\Pi(z,p_j^2).
\end{equation}
Next, consider the fiber $z*F_k$. We have
\begin{align*}
&\mu M/p_j = 0 + \mu M/p_j,\ \hbox{ with }
 0\in B,\ \mu M/p_j \in  A, 
\\
&\mu M/p_j + M/p_k = u + (\mu-1) M/p_j, \ \hbox{ with }
 u\in B,\ (\mu-1) M/p_j \in  A.
\end{align*}
Hence $z*F_k$ also splits with parity $(A,B)$, so that 
\begin{equation}\label{pinnedplanek}
a'\in\Pi(0,p_k^2),\ b'\in\Pi(z,p_k^2).
\end{equation}
By (\ref{pinnedplanej}) and (\ref{pinnedplanek}), we have
$$
%\begin{equation}\label{a'offgrid}
a'\in \ell_i(0),\ b'\in\ell_i(z).
%\end{equation}
$$
This together with (\ref{nootherb}) implies that
\begin{equation}\label{toomanyb}
\bbB_{M/p_i^2}[\mu M/p_j + \nu M/p_k]\geq 1\ \ \forall 
\mu\in\{2,3,\dots,p_j-1\},  \nu\in\{2,3,\dots,p_k-1\}.
\end{equation}
However, since $p_i=2$, by (\ref{divB-diag}) we also have $\bbB_{M/p_j}[b']=\bbB_{M/p_k}[b']=0$ and $\bbB_{M/p_i p_j}[b']$, $\bbB_{M/p_i p_k}[b']\leq 1$. At least one of $p_j$ and $p_k$ is greater than or equal to 5, say $p_k\geq 5$. Applying (\ref{toomanyb}) with $\mu=2$ and $\nu=2,3,4$, we get a contradiction.
This completes the proof of (i).

We now prove (ii). Assume that  $A_{x_j,b}\subset \ell_k(x_k)$ for some $b\in B$. By Lemma 
\ref{edb2-ell_jx_j}, we have $\bbB_{M/p_ip_k}[b]=1$. Consider now $A_{x_k,b}$.
\begin{itemize}
\item If $A_{x_k,b}\subset \ell_j(x_j)$, then we also have $\bbB_{M/p_ip_j}[b]=1$, contradicting (i).
\item We cannot have $A_{x_k,b}\subset \ell_i(x_k)$, since that would contradict (\ref{equiv-iline}).
\end{itemize}
It follows by Lemma \ref{edb2-lines} that $A_{x_k,b}\subset \ell_j(x_k)$, as claimed.
\end{proof}

%%%%%%%%%%%%%%%%%%

\begin{proposition}\label{nodiagdiv}
Assume (DB2). Then $\bbB_{M/p_ip_j}[b]=\bbB_{M/p_ip_k}[b]=0$
for all $b\in B$.
\end{proposition}

\begin{proof}
Assume for contradiction that $(b-b_k,M)=M/p_ip_k$. We again use a splitting argument. By translational invariance, we may assume that $x_j=b=0$, so that $\Lambda=\Lambda(0,D)$, $x_k=M/p_i$, and $A_1=(F_j\setminus\{0\}) \cup ((x_k*F_k)\setminus\{x_k\})$. Replacing $B$ by $rB$ for some $r\in R$ if necessary (see \cite[Lemma 3.6]{LaLo3}), we may assume that $b'=M/p_i+M/p_k$.

Let $z=\mu M/p_j + \nu M/p_k$ for some $\mu\in\{1,2,\dots,p_j-1\}$ and $\nu\in\{2,3,\dots,p_k-1\}$, and let 
$a_z\in A,b_z\in B$ be the elements such that $a_z+b_z=z$. We first consider splitting for the fiber $z*F_j$. Since
$$
\nu M/p_k = b'+ (M/p_i+(\nu-1)M/p_k), \hbox{ with }
 b'\in B,\ M/p_i+(\nu-1)M/p_k\in  A, 
$$
we have
\begin{equation}\label{twoplanesj}
\{a_z,b_z\}\subset\Pi(0,p_j^2)\cup \Pi(z,p_j^2).
\end{equation}
Next, consider the fiber $z*F_k$. Since
$$
\mu M/p_j = 0 + \mu M/p_j,\ \hbox{ with }
 0\in B,\ \mu M/p_j \in  A, 
$$
it follows that
\begin{equation}\label{twoplanesk}
\{a_z,b_z\}\subset \Pi(0,p_k^2)\cup \Pi(z,p_k^2).
\end{equation}
Combining (\ref{twoplanesj}) and (\ref{twoplanesk}), we have
$$
\{a_z,b_z\}\subset\ell_i(0)\cup\ell_i(z)\cup\ell_i(\mu M/p_j)\cup\ell_i(\nu M/p_k).
$$
Taking also into account that $a_z+b_z=z$, we only need to consider the following two cases:
\begin{equation}\label{diag-case1}
\{a_z,b_z\}\subset\ell_i(0)\cup\ell_i(z),
\end{equation}
\begin{equation}\label{diag-case2}
\{a_z,b_z\}\subset\ell_i(\mu M/p_j)\cup\ell_i(\nu M/p_k).
\end{equation}

\medskip\noindent
{\bf Case 1.} Assume that (\ref{diag-case1}) holds. We cannot have $a_z,b_z\in\Lambda$, since $A\cap\Lambda$ has no points on these two lines. Hence $(a_z,p_i)=(b_z,p_i)=1$.

Suppose that $a_z\in\ell_i(z)$ and $b_z\in\ell_i(0)$. Then 
$$
(a_z-\mu M/p_j,M)=M/p_i^2p_k = (b_z-b',M),
$$
contradicting divisor exclusion. Hence $a_z\in\ell_i(0)$ and $b_z\in\ell_i(z)$, with
\begin{equation}\label{diag1-e1}
(a_z,M)=(b_z-z,M)=M/p_i^2.
\end{equation}
Note that this implies that
\begin{equation}\label{diag1-e2}
(a_z-(M/p_i+M/p_k),M)=M/p_i^2p_k\in\Div(A).
\end{equation}

\medskip\noindent
{\bf Case 2.} Assume that (\ref{diag-case2}) holds. Suppose first that $a_z,b_z\in\Lambda$, then $a_z\in\{\mu M/p_j,M/p_i+\nu M/p_k\}$. If $a_z=\mu M/p_j$, then $b_z-b=b_z=\nu M/p_k$, contradicting divisor exclusion since $M/p_k\in\Div(A)$. If on the other hand $a_z=M/p_i+\nu M/p_k$, then $b_z=M/p_i+\mu M/p_j$, contradicting Lemma \ref{edb2-nodiag} (i).

Hence $a_z,b_z\not\in\Lambda$. If $a_z\in\ell_i(\nu M/p_k)$ and $b_z\in\ell_i(\mu M/p_j)$, it follows that
$$
(a_z-\mu M/p_j,M)=M/p_i^2p_jp_k=(b_z-b',M),
$$
contradicting divisor exclusion. Hence  $a_z\in\ell_i(\mu M/p_j)$ and $b_z\in\ell_i(\nu M/p_k)$, with 
$$
%\begin{equation}\label{diag2-e1}
(a_z-\mu M/p_j,M)=(b_z-\nu M/p_k,M)=M/p_i^2.
%\end{equation}
$$
In this case, we have
\begin{equation}\label{diag2-e2}
(b'-b_z,M)=M/p_i^2p_k\in\Div(B).
\end{equation}

\medskip

We now allow $\mu$ and $\nu$ to vary. Clearly, (\ref{diag1-e2}) and (\ref{diag2-e2}) are mutually exclusive, so that we are either always in Case 1 or always in Case 2. More precisely, either  (\ref{diag1-e1}) holds for all $\mu,\nu$, so that
\begin{equation}\label{toomanyb2}
\bbB_{M/p_i^2}[\mu M/p_j + \nu M/p_k]\geq 1\ \ \forall 
\mu\in\{1,2,\dots,p_j-1\},  \nu\in\{2,3,\dots,p_k-1\},
\end{equation}
or else (\ref{diag1-e2}) holds for all $\mu,\nu$, so that 
\begin{equation}\label{toomanya}
\begin{split}
&\bbA_{M/p_i^2}[\mu M/p_j ]\geq 1\ \ \forall 
\mu\in\{1,2,\dots,p_j-1\},\\
&\bbB_{M/p_i^2}[\nu M/p_k]\geq 1\ \ \forall 
\nu\in\{2,3,\dots,p_k-1\}.
\end{split}
\end{equation}

Since $p_i=2$, by (\ref{divB-diag}) we have $\bbB_{M/p_j}[b'']=\bbB_{M/p_k}[b'']=0$ and $\bbB_{M/p_i p_j}[b''],\bbB_{M/p_i p_k}[b'']\leq 1$ for all $b''\in B$. As in the proof of Lemma \ref{edb2-nodiag} (i), this means that (\ref{toomanyb2}) cannot hold. %For the same reason, (\ref{toomanya}) cannot hold if $p_k\geq 5$. 
It remains to consider the case when (\ref{toomanya}) holds. In this case, we have
$$
\bbA_{M/p_ip_k}[x_j]\bbB_{M/p_ip_k}[0]=p_k-1=\phi(p_ip_k),
$$
so that $A_{x_j,0}\subset\ell_k(x_k)$. 
By Lemma \ref{edb2-nodiag} (ii) with $j$ and $k$ interchanged, we have $A_{x_k,0}\subset \ell_j(x_k)$, so that by (\ref{edb2-p_ksat}) with $j$ and $k$ interchanged,
\begin{equation*}
\bbA_{M/p_j^2}[x_k]=p_j.
\end{equation*}
Hence
\begin{align*}
|A\cap\Pi(0,p_k^2)|& \geq \bbA_{M/p_j}[x_j] + \bbA_{M/p_j^2}[x_k]
+\sum_{\mu=1}^{p_j-1}\bbA_{M/p_i^2}[\mu M/p_j ]
\\
&\geq (p_j-1) + p_j + (p_j-1) 
\\
& = p_ip_j + (p_j-2)>p_ip_j,
\end{align*}
contradicting Lemma \ref{planebound}. 

This proves that $M/p_ip_k\not\in\Div(B)$. The proof that $M/p_ip_j\not\in\Div(B)$ is identical, with the $j$ and $k$ indices interchanged.
This ends the proof of the proposition.
\end{proof}

%%%%%%%%%%%%%%%%%

We are now ready to complete the proof of Theorem \ref{db-theorem} under the assumption (DB2).
We first prove (\ref{Bdecomp}). If we had $A_{x_j,b}\subset \ell_k(x_k)$ or $A_{x_k,b}\subset \ell_j(x_j)$ for some $b\in B$, it would follow by Lemma (\ref{edb2-ell_jx_j}) (iii) that $\{M/p_ip_j,M/p_ip_k\}\cap \Div(B)\neq\emptyset$; however, that is impossible by Proposition \ref{nodiagdiv}.
By Lemmas \ref{edb2-lines} and \ref{edb2-ell_jx_j}, for every $b\in B$ we must have either 
\begin{equation}\label{db2-sati}
A_{x_j,b}\subset \ell_i(x_j)
\end{equation}
or
\begin{equation}\label{db2-satk}
A_{x_j,b}\subset \ell_k(x_j).
\end{equation}
Furthermore, if there were elements $b_i,b_k\in B$ such that (\ref{db2-sati}) holds for $b_i$ and (\ref{db2-satk}) holds for $b_k$, it would follow by Lemma \ref{edb2-unif1} (iii) that $\bbB_{M/p_ip_j}[b_k]=1$, which, again, is impossible by Proposition \ref{nodiagdiv}. The same holds with the indices $i$ and $j$ interchanged.

Hence either (\ref{db2-sati}) holds for all $b\in B$, or (\ref{db2-satk}) holds for all $b\in B$. In the first case, we have $B=B_0$ by (\ref{equiv-iline}). In the second case, also by (\ref{equiv-iline}), we have $B=B_1$.

Assume first that $B=B_1$. Then $A_{x_j}\subset\ell_k(x_j)$. By (\ref{edb2-p_ksat}), it follows that $(A,B)$ has a (1,2)-cofibered structure in the $p_k$ direction, with a cofiber in $A$ at distance $M/p_k^2$ from $x_j$. 	
We now use Lemma \ref{fibershift} to shift that cofiber to $x_j$. We then do the same with $j$ and $k$ indices interchanged, using that $A_{x_k}\subset\ell_j(x_k)$.
Let $A'$ be the set thus obtained, so that $A'\cap\Lambda=\{x_j,x_k\}*(F_j\cup F_k)$.
Then $A'$ is T2-equivalent to $A$, and satisfies (DB3). By Proposition \ref{db3-cor}, $A'$ is T2-equivalent to $\Lambda$, therefore so is $A$.

It remains to consider the case when $B=B_0$. By (\ref{edb2-p_isat}), for each $b\in B$ we have either
\begin{equation}\label{pi-littlestep}
\bbA_{M/p_i^2}[x_j]=1, \bbB_{M/p_i^2}[b]=p_i, \hbox{ and } M/p_i\in \Div(B),
\end{equation}
or 
\begin{equation}\label{pi-bigstep}
\bbA_{M/p_i^2}[x_j]=p_i, \bbB_{M/p_i^2}[b]=1, \hbox{ and } M/p_i\in \Div(A).
\end{equation}
We claim that 
\begin{equation}\label{db2-isatfinal}
\hbox{(\ref{pi-bigstep}) holds for all }b\in B.
\end{equation}
Assume, by contradiction, that (\ref{pi-littlestep}) holds for some $b_0\in B$. Then $\bbA_{M/p_i^2}[x_j]=1$, hence
we cannot have (\ref{pi-bigstep}) for any $b\in B$. Therefore, all $b\in B$ must satisfy (\ref{pi-littlestep}).
But this implies that
$$
\bbB_{M}[b]+\bbB_{M/p_i}[b]+ \bbB_{M/p_i^2}[b]=p_i^2 \ \ \forall b\in B,
$$
and in particular $p_i^2$ must divide $|B|$, which is not allowed. This proves the claim.

We see from (\ref{db2-isatfinal}) that $(A,B)$ has a (1,2)-cofibered structure in the $p_i$ direction, with a cofiber in $A$ at distance $M/p_i^2$ from $x_j$.
We now apply Lemma \ref{fibershift} to shift that cofiber to $x_j$. Let $A'$ be the set thus obtained, so that $x_j*F_i\in A'$ and $A\cap(\Lambda\setminus(x_j*F_i))=A'\cap(\Lambda\setminus(x_j*F_i))$. Then $A'$ is T2-equivalent to $A$, and contains a $p_i$ corner structure. By Theorem \ref{cornerthm}, $A'$ (therefore $A$) is T2-equivalent to $\Lambda$. This completes the proof of the theorem.

%%%%%%%%%%%%%%%%%%%%%%%%%%%%%%%%%%%%%%

\section{Fibered grids} \label{fibered-sec}

%%%%%%%%%%%%%%%%%%%%%%%%%%%%%%%%%%%%%%

\subsection{Main results for fibered grids}
Throughout most of this section we will work under the following assumption. 

\medskip\noindent
{\bf Assumption (F):} We have $A\oplus B=\ZZ_M$, where $M=p_i^{2}p_j^{2}p_k^{2}$. Furthermore,
$|A|=|B|=p_ip_jp_k$, $\Phi_M|A$, and $A$ is fibered on $D(M)$-grids.

\medskip
We are not making any assumptions on the parity of $M$ at this point, but some of the arguments below will differ between the odd and even cases. We will indicate this as appropriate.

Let $\cali$ be the set of elements of $A$ that belong to an $M$-fiber in the $p_i$ direction, that is, 
$$
\cali=\{a\in A:\ \bbA_{M/p_i}[a]=\phi(p_i)\}.
$$
The sets $\calj$ and $\calk$ are defined similarly.  
The assumption (F) implies that every element of $A$ belongs to an $M$-fiber in some direction, hence $A=\cali\cup\calj\cup\calk$. We emphasize that this does {\em not} have to be a disjoint union and that it is possible for an element of $A$ to belong to two or three of these sets.

Our main result on fibered grids is the following theorem.

\begin{theorem}\label{fibered-mainthm}
Assume that (F) holds.

\medskip\noindent
{\rm (I)}  
If $\cali\cap\calj\cap\calk\neq\emptyset$, then the tiling $A\oplus B=\ZZ_M$ is T2-equivalent to $\Lambda\oplus B=\ZZ_M$, where $\Lambda:=\Lambda(0,D(M))$. By Corollary \ref{get-standard}, both $A$ and $B$ satisfy (T2).

\medskip\noindent
{\rm (II)} Assume that $\cali\cap\calj\cap\calk=\emptyset$. Then, after a permutation of the $i,j,k$ indices if necessary, the following holds.

\begin{itemize}
\item [(II\,a)] At least one of the sets $\cali,\calj,\calk$ is empty. Without loss of the generality, we may assume that $\cali=\emptyset$, so that $A\subset \calj\cup\calk$.

\item [(II\,b)] If $A\subset \calj$ or $A\subset\calk$, then $A$ is $M$-fibered in the $p_j$ or $p_k$ direction, respectively. Consequently, 
the conditions of Theorem \ref{subtile} are satisfied in that direction. By Corollary \ref{slab-reduction}, both $A$ and $B$ satisfy (T2).

\item [(II\,c)] Suppose that $\cali=\emptyset$, and that $\calj\setminus\calk$ and $\calk\setminus\calj$ are both nonempty. 

\begin{itemize}
\item [$\bullet$]  If $\Phi_{p_i}|A$, then, after interchanging $A$ and $B$, the conditions of Theorem \ref{subtile} are satisfied in the $p_i$ direction. By Corollary \ref{slab-reduction}, both $A$ and $B$ satisfy (T2).

\item [$\bullet$]  If $\Phi_{p_i^2}|A$, then $A\subset \Pi(a,p_i)$ for any $a\in A$.
By Theorem \ref{subgroup-reduction}, both $A$ and $B$ satisfy (T2).

\end{itemize}
\end{itemize}

\end{theorem}

Theorem \ref{fibered-mainthm} extends \cite[Theorem 9.1]{LaLo2}, covering both odd and even cases. 
Part (I) of the theorem was already proved in \cite[Corollary 9.6]{LaLo2}, with no assumptions on the parity of $M$. We will therefore assume from now on that
\begin{equation}\label{emptyint}
\cali\cap\calj\cap\calk=\emptyset.
\end{equation}
Moreover, part (II\,b) is simply an application of Theorem \ref{subtile} and Corollary \ref{slab-reduction} to a fibered set. Again, no parity assumptions are needed.

It remains to prove (II\,a) and (II\,c). 
We first prove in Section \ref{sec-highdivisors} that Theorem \ref{fibered-mainthm} (II) holds when
\begin{equation}\label{highdivisors}
\{D(M)|m|M\}\cap \Div(B)=\{M\}.
\end{equation}
The proof is simple and based on the structural results on fibered grids in \cite{LaLo3}. This leaves us with the case when
\begin{equation}\label{notallhigh}
\{D(M)|m|M\}\cap \Div(B)\neq \{M\},
\end{equation}
which we assume from now on. %We prove Theorem \ref{fibered-mainthm}(II\,a) in Sections \ref{fibered-A}.  
In \cite{LaLo2}, the proof of (II\,a) in the odd case is organized according to whether the sets $\cali,\calj,\calk$ are pairwise disjoint or not. We depart from that here, considering instead the following sets of assumptions.

\medskip\noindent
\textbf{Assumption (F1):} We have $A\oplus B=\ZZ_M$, where $M=p_i^{2}p_j^{2}p_k^{2}$. Furthermore, $|A|=|B|=p_ip_jp_k$, $\Phi_M|A$, $A$ is fibered on $D(M)$-grids, (\ref{emptyint}) and (\ref{notallhigh}) hold, and $\Phi_{M_\nu}|A$ for some $\nu\in \{i,j,k\}$.

\medskip\noindent
\textbf{Assumption (F2):} We have $A\oplus B=\ZZ_M$, where $M=p_i^{2}p_j^{2}p_k^{2}$. Furthermore, $|A|=|B|=p_ip_jp_k$, $\Phi_M|A$, $A$ is fibered on $D(M)$-grids, (\ref{emptyint}) and (\ref{notallhigh}) hold, and $\Phi_{M_\nu}\nmid A$ for all $\nu\in \{i,j,k\}$.

\medskip

We start with preliminaries in Section
\ref{lowerfibering}. Our results on unfibered grids in that section both strengthen those of \cite{LaLo2} and extend them to the even case. 
We then prove Theorem \ref{fibered-mainthm} (II\,a) under the assumptions (F1) and (F2) in Sections \ref{caseF1} and \ref{caseF2}, respectively.

Finally, we prove Theorem \ref{fibered-mainthm} (II\,c) in Section \ref{fibered-F3}. With (II\,a), (II\,b), and (\ref{notallhigh}) all in place, we are left with the following assumption.

\medskip\noindent
\textbf{Assumption (F3):} We have $A\oplus B=\ZZ_M$, where $M=p_i^{2}p_j^{2}p_k^{2}$. Furthermore,  $|A|=|B|=p_ip_jp_k$, $\Phi_M|A$, $A$ is fibered on $D(M)$-grids, (\ref{notallhigh}) holds, $\cali=\emptyset$, $\calj\setminus \calk\neq \emptyset$, and $\calk\setminus \calj\neq \emptyset$.

\medskip

Assuming (F3), we will prove that either the slab reduction or the subgroup reduction applies as indicated in the theorem. The proof follows the same outline as in \cite{LaLo2}, but with the even case included. We also simplify the argument in several places by using the splitting formulations of the slab reduction in Theorem \ref{subtile}. (Compare e.g. our proof of Corollary \ref{cor-subtile1} to the proof of Lemmas 9.37 and 9.38 in \cite{LaLo2}.)

\subsection{Proof of Theorem \ref{fibered-mainthm}, assuming (\ref{highdivisors})}
\label{sec-highdivisors}

%%%%%%%%%%%%%%%%%%%%%%%%%%%

If (\ref{highdivisors}) holds, Theorem \ref{fibered-mainthm} (II) has a short proof based on the results of \cite[Section 7]{LaLo3}. Both Lemma \ref{maxout1} and Corollary \ref{maxout2} are valid under more general assumptions on the cardinalities of $A$ and $B$ than (F). In particular, Corollary \ref{maxout2} assumes that $|A|=p_ip_jp_k$, but does not require the same of $B$. Part (II\,a) of the theorem in this case follows from Corollary \ref{maxout2}, and (II\,c) follows from Theorem \ref{divisor-shortcut}.

\begin{lemma}\label{maxout1}
Assume that $A\oplus B=\ZZ_M$, where $M=p_i^{n_i}p_j^{n_j}p_k^{n_k}$, and that $A$ is fibered on $D(M)$-grids. Assume further that (\ref{highdivisors}) holds.
%\begin{equation}\label{highdivisors}
%\{D(M)|m|M\}\cap \Div(B)=\{M\}.
%\end{equation}
Then:
\begin{itemize}
\item[(i)] For any $D(M)$-grid $\Lambda$, we have $|\Sigma_A(\Lambda)|=|\Lambda|$. 
\item[(ii)] If $|A|=|\Lambda|=p_ip_jp_k$, then $A$ is contained in the union of just two of the sets $\cali,\calj,\calk$.
\end{itemize}
\end{lemma}

\begin{proof}
We first prove (i). For each $z\in\Lambda$, let $a_z\in A$ and $b_z\in B$ satisfy $z=a_z+b_z$. Suppose that $a_z=a_{z'}$ for some $z,z'\in \Lambda$ with $z\neq z'$. Then $b_z-b_{z'}=(z-a_z)-(z'-a_{z'})=z-z'$, so that $b_z\neq b_{z'}$ and $D(M)|(b_z-b_{z'})$. But that contradicts (\ref{highdivisors}).

If $|A|=p_ip_jp_k$, then (i) implies that $\Sigma_A(\Lambda)=A$.
By  \cite[Lemma 8.4]{LaLo3}, $\Sigma_A(\Lambda)$ may be written as a union of fibers in just two directions. This implies (ii).
\end{proof}

\begin{corollary}\label{maxout2}
Assume that $A\oplus B=\ZZ_M$, where $M=p_i^{n_i}p_j^{n_j}p_k^{n_k}$,
and that $|A|=p_ip_jp_k$, $\Phi_M|A$, and $A$ is fibered on $D(M)$-grids. Assume further that (\ref{emptyint}) and (\ref{highdivisors}) hold. Then at least one of the sets $\cali,\calj,\calk$ is empty.
\end{corollary}

\begin{proof}
Let $\Lambda:=\Lambda(0,D(M))$. By Lemma \ref{maxout1}, we have $\Sigma_A(\Lambda)=A$. 
By Lemma \ref{maxout1}, and after a permutation of the $i,j,k$ indices if necessary, we may assume that $A\subset\cali\cup\calj$. 

If $\cali=\emptyset$ or $\calj=\emptyset$, we are done. Assume therefore that there exist $z_i,z_j\in\Lambda$ such that $z_i=a_i+b_i,z_j=a_j+b_j$ with $a_i\in\cali,a_j\in\calj$, and $b_i,b_j\in B$. Assume also, for contradiction, that $\calk\neq\emptyset$. Since $A\subset \cali\cup\calj$, we may assume without loss of generality that there exists an element $a_k\in\cali\cap\calk$. By Lemma \ref{fiberedstructure}, we have $a_k*F_i*F_k\subset A$. In particular, there exists $z\in \Lambda(z_j,M/p_ip_j)$ such that $z=a+b$ for some $a\in a_k*F_i*F_k$. Then $a\in\cali\cap\calk$; but also, by the choice of $z_j$, we have $a\in\calj$. This contradicts (\ref{emptyint}).
\end{proof}

\begin{theorem}\label{divisor-shortcut}
Assume that $A\oplus B=\ZZ_M$, where $M=p_i^{2}p_j^{2}p_k^{2}$,
and that $|A|=p_ip_jp_k$, $\Phi_M|A$, and $A$ is fibered on $D(M)$-grids. Assume further that $\calk=\emptyset$ and that (\ref{highdivisors}) holds. Then we have at least one of the following:

\begin{itemize}
\item $A$ is fibered in at least one of the $p_i$ and $p_j$ directions, hence
the conditions of Theorem \ref{subtile} are satisfied in that direction. By Corollary \ref{slab-reduction}, both $A$ and $B$ satisfy (T2).

\item  $A\subset \Pi(a,p_k)$ for any $a\in A$.
By Theorem \ref{subgroup-reduction}, both $A$ and $B$ satisfy (T2).

\end{itemize}

\end{theorem}

\begin{proof}
Assume that (\ref{highdivisors}) holds, and that $A\subset\cali\cup\calj$ but $A$ is not fibered in either the $p_i$ or the $p_j$ direction. Let $\Lambda=\Lambda(0,D(M))$. By Lemma \ref{maxout1} (i), we have
$|\Sigma_A(\Lambda)|=|\Lambda|=|A|$, hence
$$
\Sigma_A(\Lambda)=A.
$$
In particular, the non-fibering assumption implies that 
$$
\Sigma_A(\Lambda)\cap\cali=A\cap\cali \neq\emptyset, \ \ \Sigma_A(\Lambda)\cap\calj=A\cap\calj \neq\emptyset.
$$
By \cite[Proposition 7.5 (ii)]{LaLo3}, we have 
$$
A=\Sigma_A(\Lambda)\subset \Pi(a,p_k^{n_k-1}) \hbox{ for any }a\in A,
$$
and the theorem is proved.
\end{proof}

%%%%%%%%%%%%%%%%%%%%%%%%%

\subsection{Fibering on lower scales}\label{lowerfibering}

%%%%%%%%%%%%%%%%%%%%%%%%%

We now begin the proof of the theorem in the case when both (\ref{emptyint}) and (\ref{notallhigh}) hold. 
We will continue to use the notation of \cite{LaLo2}.
Given $N|M$, we will use $\bbI^N$, $\bbJ^N$, $\bbK^N$ to denote the $N$-boxes associated with $\cali$, $\calj$, $\calk$ respectively. Recall also that $M_\nu=M/p_\nu^2$ for $\nu\in \{i,j,k\}$. We will need a few preliminary results from \cite{LaLo2}. Lemmas \ref{fiberedstructure} and \ref{fibercyc3} %{Knotunfiberedunstructured} %{planesubsetcalk} 
do not require any assumptions on the parity of $M$, so that the same results hold for any permutation of the indices $i,j,k$.

\begin{lemma}\label{fiberedstructure}\cite[Lemma 9.5]{LaLo2}. 
Assume (F) and (\ref{emptyint}), and let $a\in \calj\cap\calk$. Then $A\cap\Pi(a,p_i^{2})=a*F_j*F_k.$
\end{lemma}

\begin{lemma}\label{fibercyc3}\cite[Lemma 9.8]{LaLo2}
Assume (F). For each $\alpha\in\{1,2\}$, we have
$$\Phi_{M/p_k^{\alpha_k} }|A \ \Leftrightarrow \ \Phi_{M/p_k^{\alpha_k} }|\calk.
$$
In particular, if $\Phi_{M/p_k^{\alpha_k} }\nmid A$ for some $\alpha_k\in\{1,2\}$, then $\calk\neq\emptyset$.
\end{lemma}

%
%
%\begin{lemma}\label{Knotunfiberedunstructured}\cite[Lemma 9.9]{LaLo2}
%Let $A\oplus B=\ZZ_M, M=p_i^2p_j^2p_k^2, |A|=|B|=p_ip_jp_k$. Then
%$$
%%\begin{equation}\label{Kdivrest}
%\{m/p_k^\alpha :\ \alpha\in\{0,1,2\}, m\in \{M,M/p_i,M/p_j,M/p_ip_j\} \}\cap \Div(B)\neq	\emptyset .
%%\end{equation}
%$$
%\end{lemma}

Proposition \ref{no unfibered grids} below is the main result of this subsection.

\begin{proposition}\label{no unfibered grids}
Assume that (F), (\ref{emptyint}), and (\ref{notallhigh}) hold. Assume further that $\Phi_{N_k}|A$. Then $\calk$ mod $N_k$ cannot be unfibered on any $D(N_k)$-grid. 
\end{proposition}

\begin{proof}
Assume first that $M$ is odd. Suppose that $\calk$ mod $N_k$ is unfibered on $\Lambda:=\Lambda(a,D(N_k))$ for some $a\in A$. If
\begin{equation}\label{baddivisors}
\{N_k/p_i,N_k/p_j,N_k/p_ip_j\}\subset\Div_{N_k}(\calk),
\end{equation}
then (\ref{highdivisors}) holds, contradicting the assumptions of the proposition. Assume therefore that (\ref{baddivisors}) fails. By the classification results in \cite[Section 6.2]{LaLo2}, $\calk$ mod $N_k$ must then have one of the following structures on $\Lambda$.

\begin{itemize}
\item $p_\nu$-full plane as in \cite[Lemma 6.3]{LaLo2} for some $\nu\in\{i,j\}$, with $c_0=p_k$. Assume without loss of generality that $\nu=i$. Then, for $x\in\Lambda$ indicated in \cite[Lemma 6.2]{LaLo2}, we have 
$$
|\calk \cap\Pi(x,p_i^2) |\geq \phi(p_jp_k^2)>p_jp_k,
$$
contradicting Lemma \ref{planebound}.

\item $p_k$-corner as in \cite[Lemma 6.4]{LaLo2} (i), with $c_0= p_k$. Then there exist $a,a'\in\calk$ such that $(a-a',M)=N_k/p_k=M/p_k^2$ and $(a*F_i*F_k)\cup(a'*F_j*F_k)\subset\calk$. But then 
$$
|\calk \cap\Pi(a,p_i^2) |\geq (p_j+1)p_k>p_jp_k,
$$
contradicting Lemma \ref{planebound}.

\item $p_k$-almost corner as in \cite[Lemma 6.4]{LaLo2} (i), with $c_0= p_k$. Then there exists $a\in\calk$ such that 
$$
|\calk \cap\Pi(a,p_i^2) |\geq 2p_k \phi(p_j)>p_jp_k,
$$
again contradicting Lemma \ref{planebound}.

\end{itemize}
This proves the proposition in the odd case.

Assume now that $M$ is even, with $p_i=2$. Again, it suffices to consider the case when (\ref{baddivisors}) fails, so that 
it suffices to consider the unfibered structures in Lemma \ref{even_struct}. For any corner structures as in Lemma \ref{even_struct} (i), the proof is the same as for the odd case. It remains to consider the diagonal boxes structures from Lemma \ref{even_struct} (ii). We distinguish between the cases $\Phi_{N_i}|A$ (Lemma \ref{subset}, below) and $\Phi_{N_k}|A$ for $k\neq i$ (Lemma \ref{almostmaxedout}). These two lemmas complete the proof of the proposition.
\end{proof}

\begin{lemma}\label{subset}
Assume that (F), (\ref{emptyint}), and (\ref{notallhigh}) hold. Assume further that $p_i=2$ and $\Phi_{N_k}|A$. 
Let $\Lambda_0=\Lambda(a,D(N_k))$ for some $a\in\calk$. Then
 $\calk\cap\Lambda_0$ mod $\ZZ_{N_k}$ cannot be a set of diagonal boxes as in Lemma \ref{even_struct} (ii).

\end{lemma}

\begin{proof}
We identify $\Lambda_0$ with $\ZZ_{p_i}\times\ZZ_{p_j}\times\ZZ_{p_k}$, where for every $x\in \Lambda_0$ we have $\bbK^{N_k}_{N_k}[x]\in \{0,p_k\}$. Note that under this representation, elements that differ only in their $p_k$ coordinate are at distance $N_k/p_k$ from one another.
Assume, by contradiction, that 
$\calk\cap\Lambda_0$ is a set of diagonal boxes $(I\times J\times K)\cup (I^c\times J^c\times K^c)$ with multiplicity $p_k$, and let
\begin{equation}\label{aIJK}
a\in I\times J\times K.
\end{equation}
Since $p_i=2$, we may assume without loss of generality that $I=\{0\}, I^c=\{1\}$. 

\noindent
{\bf Claim 1.} Suppose that $\calk\cap\Lambda_0$ is a set of diagonal boxes as above. Then 
\begin{enumerate}
\item [(i)] $\max\{|J||K|,|J^c||K^c|\}\leq p_j$,
\item [(ii)] $\max\{|K|,|K^c|\}\leq 2$.
\end{enumerate}

\begin{proof}
For (i), assuming the contrary, we get
$$
|A\cap\Pi(a,p_i^2)|\geq p_k|I\times J\times K|=p_k|J||K|>p_jp_k,
$$
which contradicts Lemma \ref{planebound}.

For (ii), if $|K|>2$, then
$$
|A\cap\Pi(a,p_j^2)|\geq \bbK^{N_k}_{N_k}[a]+\bbK^{N_k}_{N_k/p_k}[a]>p_k|K|>p_ip_k,
$$
again contradicting Lemma \ref{planebound}. 

Applying the same argument with $a$ replaced by an element of $I^c\times J^c\times K^c$, we get  $|J^c||K^c|\leq p_j$ and $|K^c|\leq 2$.
\end{proof}

Claim 1(i) implies that $p_k=|K|+|K^c|\leq 4$, hence $p_k=3$ and $p_j\geq 5$. We assume this from here on, with $|K|=2, |K^c|=1$. Then 
\begin{equation}\label{Aell_k(a)count}
|A\cap\ell_k(a)|=p_k|K|=p_ip_k,
\end{equation}
with $A\cap\ell_k(a)= \calk\cap\ell_k(a)$.
By Lemma \ref{planebound}, $|A\cap\Pi(a,p_j^2)|=p_ip_k$ and 
\begin{equation}\label{Ainp_jplane}
(A\cap\Pi(a, p_j^2))\subset \ell_k(a).
\end{equation}
By Claim 1 (i), we have $|J|\leq p_j/2$ and therefore
\begin{equation}\label{J^cbound}
|J^c|\geq p_j/2>2.
\end{equation}

\medskip\noindent
{\bf Claim 2. } Suppose that $\calk\cap\Lambda_0$ is a set of diagonal boxes as above. Then $\Phi_{p_i^2}|A$.
\begin{proof}

We have
\begin{align*}
|A\cap\Pi(a, p_i)|&\geq p_k(|K||J|+|K^c||J^c|)\\
&=p_k(2|J|+|J^c|)>p_jp_k.
\end{align*}
By Corollary \ref{planegrid}, this implies the claim.
\end{proof}

\medskip\noindent
{\bf Claim 3. } Suppose that $\calk\cap\Lambda_0$ is a set of diagonal boxes as above. Then $\Phi_{p_i^2}|B$.
\begin{proof}

Let $x\in I^c\times J\times K^c$ with $(x-a,M)=M/p_ip_k^2$, where $a$ is as in (\ref{aIJK}). By (\ref{Ainp_jplane}), we have $x\in \ZZ_M\setminus A$.
 Consider the saturating set $A_x$. By (\ref{J^cbound}), we have $A_x\subset \Pi(x,p_j^2)=\Pi(a,p_j^2)$. By (\ref{Ainp_jplane}), we further have $A_x\subset \ell_k(a)$. As in (\ref{Aell_k(a)count}), we see that $A\cap\ell_k(a)= \calk\cap\ell_k(a)$, so that
$\bbA_{M/p_ip_k}[x]=0$ and
$$
\bbA_{M/p_ip_k^2}[x|\ell_k(a)]\bbB_{M/p_ip_k^2}[b]=\phi(p_ip_k^2) \hbox{ for any } b\in B.
$$
With $p_k=3$ and $p_i=2$, we have $\phi(p_ip_k^2)=6=p_ip_k=\bbA_{M/p_ip_k^2}[x|\ell_k(a)]$, so that 
$$\bbB_{M/p_ip_k^2}[b]=1.$$
Since $|K|=2$, we also have $M/p_k,M/p_k^2\in \Div(A)$, so that $\bbB^{M_k}_{M_k}[b]=1$ for any $b\in B$. These, in turn, imply that the set $B$ is organized in pairs so that for every $b\in B$ we have 
$$
%\begin{equation}\label{Bklineinters}
B\cap\ell_k(b)=\{b\}
%\end{equation}
$$
and there exists a unique $b'\in B$ with $(b-b',M)=M/p_ip_k^2$, so that $B\cap\Lambda(b,M/p_ip_k^2)=\{b,b'\}$. Since any plane $\Pi(y,p_i)$ with $y\in\ZZ_M$ may be written as a disjoint union of nonintersecting $M/p_ip_k^2$-grids, and the intersection of $B$ with each such grid is either empty or a two-point set as above, we see that 
$$
%\begin{equation}\label{cyclo-uniform}
|B\cap \Pi(y,p_i^2)|=\frac{1}{p_i} |B\cap \Pi(y,p_i)| \ \ \forall y\in\ZZ_M,
%\end{equation} 
$$
so that $\Phi_{p_i^2}|B$.
\end{proof}

Since Claim 3 contradicts Claim 2, the lemma is proved..
\end{proof}

%fart

\begin{lemma}\label{almostmaxedout}
Assume (F) and (\ref{emptyint}) with $p_i=2$. Suppose that $\Phi_{N_i}|A$, and that there exists a $D(N_i)$-grid $\Lambda$ on which $\cali$ mod $N_i$ is a set of diagonal boxes as in Lemma \ref{even_struct} (ii).
Then (\ref{highdivisors}) must hold.
\end{lemma}

\begin{proof}
Assume $\Phi_{N_i}|A$, and let $\Lambda$ be as in the statement of the lemma. Then
$\cali\cap\Lambda$ mod $N_i$ contains diagonal boxes $(I\times J\times K)\cup (I^c\times J^c\times K^c)$ %, $\Phi_{p_j^2}\Phi_{p_k^2}|A$, 
and
for any $x\in (I\times J^c\times K^c)\cup (I^c\times J\times K)$ we have
%\begin{equation}\label{p_ipartialfiber}
$$
\bbI^{N_i}_{N_i/p_i}[x]=\phi(p_i^2)=p_i.
$$
%\end{equation}

Without loss of generality, we may assume that $2=p_i<p_k<p_j$.
If $\min(|J|,|K|)\geq 2$, we have
$$
\{D(M)|m|M\}\subset \Div(I\times J\times K)\subset \Div(A),
$$
and similarly (\ref{highdivisors}) holds when $\min(|J^c|,|K^c|)\geq 2$. 
After relabelling $I,J,K$ as $I^c,J^c,K^c$ and vice versa if necessary, 
it remains to consider the case when 
$$
|J|=|K^c|=1.
$$
We then have
\begin{equation}\label{diagboxdiv2}
\{M/p_i,M/p_j,M/p_k,M/p_ip_j,M/p_ip_k\}\subset \Div(A\cap\Lambda),
\end{equation}
with $\{M/p_i,M/p_k,M/p_ip_k\}\subset \Div(I\times J\times K)$ and 
$M/p_ip_j\in \Div(I^c\times J^c\times K^c)$.

Assume, by contradiction, that $m\in \{D(M)|m|M\}\cap \Div(B)$ and $m\neq M$. By (\ref{diagboxdiv2}), we have $m\in\{M/p_jp_k,M/p_ip_jp_k\}$. Suppose that $b,b'\in B$ satisfy $(b-b',M)=m$. By translational invariance, we may assume that $b'=0$ and that $0\in I\times J\times K$. 

Let $\Lambda=\Lambda(0,D(M))$. Let $G\subset\Lambda$ be the set that projects to $I\times J\times K$ modulo $N_i$, so that $G$ is the union of $p_k-1$ disjoint $M$-fibers in the $p_i$ direction. Consider the sets $G$ and $b*G$, with $b$ as above. 
Since $p_k\geq 3$, we must have $M/p_j\in\Div(G,b*G)$. In other words, 
there exist $z,z'\in\Lambda$ such that $(z-z',M)=M/p_j$ and $z=a+b$, $z'=a'+b'=a'+0$, with $a,a'\in G$. 
It follows that $z*F_j$ splits with parity $(A,B)$, so that for every $z''\in z*F_j$ we have $z''=a''+b''$ for some $a''\in A\cap \Pi(0,p_j^2)$ and $b''\in B$.

On the other hand, if $z_1,z_2\in z*F_j$ satisfy $z_1\neq z_2$ and $z_\nu=a_\nu+b_\nu$ for some $a_\nu\in  A\cap \Pi(0,p_j^2)$ and $b_\nu\in B$, we cannot have $M/p_i|(a_1-a_2)$. Indeed, otherwise $b_1-b_2=(z_1-z_2)-(a_1-a_2)$ would be divisible by $M/p_i$, which is impossible since $M/p_i\in\Div(\cali)\subset \Div(A)$. Therefore at most $p_k-1$ points of $z*F_j$ can be tiled by points of $G$. Since $p_j-(p_k-1)\geq 3$, there are at least 3 more points of $z*F_j$ that need to be tiled some other way, each one requiring a different point of $A\cap \Pi(0,p_j^2)$. However, $|G|=p_i(p_k-1)$, so that by Lemma \ref{planebound}, $(A\cap \Pi(0,p_j^2))\setminus G$ may contain at most 2 distinct points. This contradiction proves the lemma.
\end{proof}

We collect a few results on fibered $N_i$-grids for future reference.

\begin{lemma}\label{fiberfibering-new}
Assume that (F) and (\ref{notallhigh}) hold.

\smallskip

(i) Let $\Lambda:= \Lambda(a_0,D(N_i))$ for some $a_0\in\cali$. If $\cali\cap\calj\cap\Lambda=\emptyset$ and $A$ is $N_i$-fibered on $\Lambda$, it cannot be fibered in the $p_j$ direction.

\smallskip
(ii) Assume that $p_k>\min_\nu p_\nu$. If $\Phi_{N_k}|A$, then $\calk$ is $N_k$-fibered on each $D(N_k)$-grid in one of the $p_i$ and $p_j$ directions. In particular, $\calk\subset\cali\cup\calj$.

\smallskip

(iii) If $\Phi_{N_i}|B$, then $B$ must be $N_i$-fibered in the $p_i$ direction.

\end{lemma}

\begin{proof}
Parts (i) and (iii) are from \cite[Lemma 9.11]{LaLo2}. While the lemma was only stated there for odd $M$, the parity assumption is not needed in the proof.

For (ii), assume that $\Phi_{N_k}|A$ and $p_k>p_i$. By Proposition \ref{no unfibered grids}, $\calk$ is $N_k$-fibered on each $D(N_k)$-grid. If there existed an element $a\in\calk$ such that $\calk\cap\Lambda(a,D(N_k))$ is $N_k$-fibered in the $p_k$ direction, we would have $|A\cap\Pi(a,p_j^2)|\geq|\calk\cap\ell_k(a)|\geq p_k^2>p_ip_k$, contradicting Lemma \ref{planebound}.
\end{proof}

Recall that $M_\nu=M/p_\nu^2$ has only two distinct prime factors for each $\nu\in\{i,j,k\}$. In particular, all $M_\nu$-cuboids are 2-dimensional, so that Lemma \ref{2d-cyclo} applies on that scale. Thus, if $\Phi_{M_i}|A$, then $A$ mod $M_i$ is a linear combination of $M_i$-fibers in the $p_j$ and $p_k$ directions, with non-negative integer coefficients. In particular, if $\Phi_{M_i}|A$
and
$$
\bbA^{M_i}_{M_i}[x]\in\{0,c_0\}\ \ \forall x\in\ZZ_M,
$$
then $A$ is $M_i$-fibered in one of the $p_j$ and $p_k$ directions on every $D(M_i)$-grid. Similar statements hold with $A$ replaced by $B$, as well as for other permutations of the indices $i,j,k$. In particular, we have the following fibering result. 

\begin{lemma}\label{M_kfibering}
Assume that (F) holds. Suppose $\Phi_{M_k}|A$, then for every $a_k\in \calk$ we have $\bbK^{M_k}_{M_k}[a_k]=p_k$. Moreover, we must have one of the following:
\begin{equation}\label{M_kfiberp_i}
a_k \text{ belongs to an } M_k \text{-fiber in the } p_i \text{ direction, i.e. } \bbK^{M_k}_{M_k/p_i}[a_k]=p_k\phi(p_i),
\end{equation}
or
\begin{equation}\label{M_kfiberp_j}
a_k \text{ belongs to an } M_k \text{-fiber in the } p_j \text{ direction, i.e. } \bbK^{M_k}_{M_k/p_j}[a_k]=p_k\phi(p_j).
\end{equation} 
In addition, if (\ref{M_kfiberp_i}) holds then 
\begin{equation}\label{M_kfiber}
|A\cap\Pi(a_k,p_j^2)|=p_ip_k \text{ and } (A\cap\Pi(a_k,p_j^2))\subset \calk,
\end{equation}
and if (\ref{M_kfiberp_j}) holds then $|A\cap\Pi(a_k,p_i^2)|=p_jp_k$ and $(A\cap\Pi(a_k,p_i^2))\subset \calk$.
\end{lemma}
\begin{proof}
We write $\calk$ mod $M_k$ as a linear combination of $M_k$-fibers in the $p_i$ and $p_j$ directions with non-negative integer coefficients, as permitted by 
Lemma \ref{2d-cyclo} with $N=M_k$. Then for every $a_k\in\calk$ we have
$\bbK^{M_k}_{M_k}[a_k]=(c_i+c_j)p_k,$ where $c_\nu$ is the number of $M_k$-fibers in the $p_\nu$ direction passing through $a_k$ in that representation.
%Then
%$$
%\bbK^{M_k}_{M_k/p_\nu}[a_k]=c_\nu p_k\phi(p_\nu),\ \ \nu\in\{i,j\}.
%$$
Without loss of generality assume that $c_i>0$. It follows that
\begin{align*}
|A\cap\Pi(a_k,p_j^2)|&\geq \bbK^{M_k}_{M_k}[a_k]+\bbK^{M_k}_{M_k/p_i}[a_k]\\
&\geq (c_i+c_j) p_k+c_ip_k\phi(p_i)\\
&\geq p_ip_k
\end{align*}	 
Recall from Lemma \ref{planebound} that $p_ip_k\geq |A\cap\Pi(a_k,p_j^2)|$, so that the chain of inequalities above must holds with equalities. In particular, $\bbK^{M_k}_{M_k}[a_k]=c_i=1$ and $c_j=0$. 
\end{proof}

%%%%%%%%%%%%%%%%%%%%%%%%%%%%%%

\subsection{Proof of Theorem \ref{fibered-mainthm} (II\,a), 
Case (F1)}\label{caseF1}

For the reader's convenience, we recall the assumption (F1) and the result we need to prove.

\medskip\noindent
\textbf{Assumption (F1):} We have $A\oplus B=\ZZ_M$, where $M=p_i^{2}p_j^{2}p_k^{2}$. Furthermore, $|A|=|B|=p_ip_jp_k$, $\Phi_M|A$, $A$ is fibered on $D(M)$-grids, (\ref{emptyint}) and (\ref{notallhigh}) hold, and $\Phi_{M_\nu}|A$ for some $\nu\in \{i,j,k\}$.

\begin{proposition}\label{emptyset-F1}
	Assume that (F1) holds. Then one of the sets $\cali,\calj,\calk$ is empty.
\end{proposition}

\medskip
Without loss of generality we shall assume that
\begin{equation}\label{Phi_M_kdiv}
	\Phi_{M_k}|A.
\end{equation}
Assume for contradiction that the proposition fails, so that
\begin{equation}\label{nonempty}
	\cali,\calj,\calk\neq\emptyset.
\end{equation}
We will prove that this is impossible.

\begin{lemma}\label{M_kfibercyc}
Assume that (F1), (\ref{Phi_M_kdiv}), and (\ref{nonempty}) hold. If there exists $a_k\in\calk$ satisfying (\ref{M_kfiberp_i}), then $\Phi_{p_j^2}\nmid A$. 
The same holds with $i$ and $j$ interchanged.
\end{lemma}

\begin{proof}
Assume for contradiction that (\ref{M_kfiberp_i}) holds for some $a_k\in\calk$, but 
$\Phi_{p_j^2}|A$.
The second assumption, together with (\ref{M_kfiber}), implies that $A\subset\Pi(a_k,p_j)$. 
	
Suppose that $a_j\in\calj$. Then $a_j*F_j\subset \Pi(a_k,p_j)$, and there exists an element $a'_j\in (a_j*F_j)\cap \Pi(a_k,p_j^2)$. By (\ref{M_kfiber}), we have $a'_j\in \calj\cap \calk$, so that 
\begin{equation}\label{p_i^2planefiber}
a_j*F_j*F_k= a'_j*F_j*F_k= A\cap\Pi(a_j,p_i^2)
\end{equation}
by Lemma \ref{fiberedstructure}. In particular, $a_j\in\calk$. Since $a_j\in\calj$ was arbitrary, we have $\calj\subset \calk$. 
	
Additionally, (\ref{M_kfiberp_i}) and (\ref{p_i^2planefiber}) imply that 
$|A\cap\Pi(a_j,p_i^2)|=p_jp_k$ and 
$|A\cap\Pi(a_j,p_i)|>p_jp_k$. By Corollary \ref{planegrid}, we have $\Phi_{p_i^2}|A$ and  $A\subset \Pi(a_j,p_i)$.

Now let $a_i\in\cali$. By (\ref{p_i^2planefiber}) and Lemma \ref{planebound}, there must be an element $a'_i\in a_i*F_i$ such that $a'_i\in a_j*F_j*F_k$. But then $a'_i\in\cali\cap\calj\cap \calk$, contradicting (\ref{emptyint}).
\end{proof}

\begin{lemma}\label{M_kfiber_unif}
Assume that (F1), (\ref{Phi_M_kdiv}) and (\ref{nonempty}) hold. If there exists $a_k\in\calk$ satisfying (\ref{M_kfiberp_i}), then the same holds for all $a\in\calk$. Moreover, $\Phi_{p_j}|A$. The same is true with the $i$ and $j$ indices interchanged. 
\end{lemma}

\begin{proof}
Assume for contradiction that there exist $a,a'\in \calk$ such that $a$ satisfies (\ref{M_kfiberp_i}) and $a'$ satisfies (\ref{M_kfiberp_j}). We also assume, without loss of generality, that $p_i>p_j$. 
	
From (\ref{M_kfiberp_i}) and (\ref{M_kfiber}) we see that 
$$
|A\cap\Pi(a,p_i)|\geq p_ip_k>p_jp_k.
$$
By Corollary \ref{planegrid}, we get $\Phi_{p_i^2}|A$. This, however, implies by Lemma \ref{M_kfibercyc} that (\ref{M_kfiberp_j}) cannot hold for $a'$, a contradiction.
The second conclusion follows from Lemma \ref{M_kfibercyc}.
\end{proof}

In light of Lemmas \ref{M_kfibercyc} and \ref{M_kfiber_unif}, we may assume from now on that
\begin{equation}\label{lastM_kcase}
(\ref{M_kfiberp_i}) \text{ holds for all } a\in\calk\text{, and } \Phi_{p_j}|A.
\end{equation}

We record the following simple consequence.
\begin{corollary}\label{F1assumcor}
Assume that (F1), (\ref{Phi_M_kdiv}), (\ref{nonempty}) and (\ref{lastM_kcase}) hold. Then:
\begin{itemize}
\item[(i)] $|A\cap\Pi(a_k,p_i)|\geq p_ip_k$ for all $a_k\in\calk$,
\item[(ii)] $|A\cap\Pi(x,p_j)|=p_ip_k$ for all $x\in \ZZ_M$,
\item[(iii)] $\calk\cap\Pi(a_j,p_j)=\emptyset$ for all $a_j\in \calj$ (in particular, $\calj\cap\calk=\emptyset$),
\item[(iv)] $(\cali\setminus\calj)\cap\Pi(a_j,p_j)\neq \emptyset$ for all $a_j\in \calj$,
\item[(v)] $p_j<p_k$.
\end{itemize} 
\end{corollary}
\begin{proof}
Parts (i) and (ii) follow directly from, respectively, the first and second part of (\ref{lastM_kcase}). 
	
We now prove (iii). Let $a_j\in \calj$, and assume for contradiction that there exists an element $a_k\in\calk\cap\Pi(a_j,p_j)$. By (\ref{M_kfiber}), we have $|A\cap\Pi(a_k,p_j^2)|=p_ip_k$.
Since $a_j*F_j$ cannot all be contained in $\Pi(a_k,p_j^2)$, we get $|A\cap \Pi(a_j,p_j)|>p_ip_k$, contradicting (ii).
	
By (iii), we have $A\cap\Pi(a_j,p_j)\subset\cali\cup\calj$. It follows that there exist integers $c_j>0, c_i\geq0$ such that 
$$
p_ip_k=|A\cap\Pi(a_j,p_j)|=c_ip_i+c_jp_j,
$$
with $c_ip_i$ accounting for the elements of $(\cali\setminus\calj)\cap\Pi(a_j,p_j)$. Hence $p_k=c_i+c'_jp_j$ for some $c'_j>0$, and $p_j<p_k$ as claimed in (v). Furthermore, we must have $c_i>0$, since $p_j$ does not divide $p_k$. This proves (iv).
\end{proof}

\begin{lemma}\label{Phi_M_inotA}
Assume that (F1), (\ref{Phi_M_kdiv}), (\ref{nonempty}), and (\ref{lastM_kcase}) hold. Then $\Phi_{M_i}\nmid A$.
\end{lemma}
\begin{proof}
Assume, by contradiction, that $\Phi_{M_i}| A$. By Lemma \ref{fibercyc3}, we have $\Phi_{M_i}|\cali$. 
	
Let $a_j\in\calj$. By Corollary \ref{F1assumcor} (iii),  $A\cap\Pi(a_j,p_j)\subset\cali\cup\calj$. 
By Lemma \ref{M_kfibering} with $M_k$ replaced by $M_i$, 
each $a_i\in\cali\cap\Pi(a_j,p_j)$ must satisfy either
\begin{equation}\label{Mi-choosep_j}
\bbI^{M_i}_{M_i/p_j}[a_i]=p_i\phi(p_j).
\end{equation}
or 
\begin{equation}\label{Mi-choosep_k}
\bbI^{M_i}_{M_i/p_k}[a_i]=p_i\phi(p_k)
\end{equation}
Suppose that (\ref{Mi-choosep_k}) holds for some $a_i\in\cali\cap\Pi(a_j,p_j)$. Such an element must exist due to Corollary \ref{F1assumcor} (iv). Then
$|\cali\cap \Pi(a_i,p_j^2)|=p_ip_k$, so that
$$
|A\cap\Pi(a_j,p_j)|\geq  |\cali\cap\Pi(a_i,p_j^2)|+ |(a_j*F_j)\setminus\{a_j\}|
>p_ip_k,
$$
contradicting Corollary \ref{F1assumcor} (ii).
	
Hence (\ref{Mi-choosep_j}) holds for all $a_i\in\cali\cap\Pi(a_j,p_j)$.
We conclude that there must be two nonnegative integers $c_i,c_j$ such that 
$$
p_ip_k=|A\cap\Pi(a_j,p_j)|=c_ip_ip_j+c_jp_j,
$$
with the first term accounting for all $a_i$ as above and the second term accounting for all $a_j\in(\calj\setminus\cali)\cap\Pi(a_j,p_j)$.
This, however, is not allowed since $p_j$ does not divide $p_ip_k$. 
\end{proof}

The main step in the proof of Proposition \ref{emptyset-F1} is the following proposition. 
\begin{proposition}\label{initialstatement}
Assume that (F1), (\ref{Phi_M_kdiv}), (\ref{nonempty}) and (\ref{lastM_kcase}) hold. Then 
\begin{equation}\label{Phi_N_inotA}
\Phi_{N_i}\nmid A.
\end{equation}
\end{proposition} 
We first finish the proof of Proposition \ref{emptyset-F1}, assuming Proposition \ref{initialstatement}.

\begin{lemma}\label{N_iM_ifiber}
Assume that (F1), (\ref{Phi_M_kdiv}), (\ref{nonempty}), and (\ref{lastM_kcase}) hold. Then $B$ is $N_i$-fibered in the $p_j$ direction, with $\bbB_{M/p_ip_j}[b]=\phi(p_j)$ for all $b\in B$. Consequently, $p_j=\min_\nu p_\nu,M/p_ip_j\in \Div(B)$, and $\cali\cap\calj=\emptyset$.
\end{lemma}
\begin{proof}
Consider $B$ modulo $N_i$. Since $M/p_i\notin\Div(B)$, we have 
$$
%\begin{equation}\label{01-BmodNi}
\bbB^{N_i}_{N_i}[y]=\bbB_M[y]+\bbB_{M/p_i}[y]\in\{0,1\}\hbox{ for all }y\in\ZZ_M.
%\end{equation}
$$
By Lemma \ref{Phi_M_inotA} and (\ref{Phi_N_inotA}), we have $\Phi_{N_i}\Phi_{M_i}|B$.
Furthermore, since $M/p_i\in\Div(A)$, (\ref{binary-cond-2}) holds with $c_0=1$.
It follows from Corollary \ref{doublediv} that $B$ is a union of pairwise disjoint $N_i$-fibers in the $p_j$ and $p_k$ directions. We note that if there exists an element $b\in B$ which belongs an $N_i$-fiber in the $p_j$ direction, then $p_j<p_i$. In that case, together with Corollary \ref{F1assumcor} (v), we have $p_j=\min_\nu p_\nu$.
	
It remains to prove that $B$ mod $N_i$ cannot contain an $N_i$-fiber in the $p_k$ direction. Indeed, assume for contradiction that $\{b_1,\dots,b_{p_k}\}$ is such a fiber. Since $M/p_k\in\Div(A)$, we must have $(b_\mu-b_\nu,M)=M/p_ip_k$ for $\mu\neq\nu$. This is only possible if $p_k<p_i$.  However, consider any $M_k$-fiber in the $p_i$ direction in $\calk$, as provided by (\ref{lastM_kcase}). If $p_k<p_i$, then any such fiber must include $M/p_ip_k$ as a difference. Thus $M/p_ip_k\in \Div(A)\cap \Div(B)$, a contradiction.
\end{proof}

By Corollary \ref{F1assumcor} and Lemma \ref{N_iM_ifiber}, we have $|A\cap(a_k,p_i)|\geq p_ip_k>p_jp_k$. Corollary \ref{planegrid} implies that
\begin{equation}\label{Phi_p_i^2|A}
\Phi_{p_i^2}|A.
\end{equation}

\begin{lemma}\label{Phi_N_inotAlemma}
Assume (F1), (\ref{Phi_M_kdiv}), (\ref{nonempty}), and (\ref{lastM_kcase}). Then $\Phi_{M_j}\nmid A$.
	
\end{lemma}

\begin{proof}
Assume, by contradiction, that $\Phi_{M_j}|A$. By Lemma \ref{fibercyc3}, this implies that $\Phi_{M_j}|\calj$. By Lemma \ref{M_kfibering} with $k$ and $j$ interchanged, $\calj$ mod $M_j$ is a union of disjoint $M_j$-fibers in the $p_i$ and $p_k$ directions. Suppose first that there exists an $M_j$-fiber in the $p_i$ direction in $\calj$, with $\bbJ^{M_j}_{M_j/p_i}[a_j]=p_j\phi(p_i)$ for some $a_j\in\calj$. Since $p_j<p_i$ by Lemma \ref{N_iM_ifiber}, it follows that $M/p_ip_j\in \Div(A)$; but this also contradicts Lemma \ref{N_iM_ifiber}. 
	
Hence $\bbJ^{M_j}_{M_j/p_k}[a_j]=p_j\phi(p_k)$ for all $a_j\in\calj$. In particular, $|\calj\cap\Pi(a_j,p_i^2)|\geq p_jp_k$, and by Lemma \ref{planebound}, the latter holds with equality. This together with (\ref{Phi_p_i^2|A}) implies that $A\subset \Pi(a_j,p_i)$. 
	
Now, let $a_k\in \calk$. Since $a_k\in \Pi(a_j,p_i)$ and satisfies (\ref{M_kfiberp_i}), there must be an element $a_k'\in\calk\cap \Pi(a_j,p_i^2)$. Since  $a_k'\notin\calj$ by Corollary \ref{F1assumcor} (iii),
we get $|A\cap \Pi(a_j,p_i^2)|>p_jp_k$, contradicting Lemma \ref{planebound}.
\end{proof}

\begin{lemma}\label{Phi_M_jnotA}
Assume (F1), (\ref{Phi_M_kdiv}), (\ref{nonempty}), and (\ref{lastM_kcase}). Then $\Phi_{N_j}\nmid A$. 
\end{lemma}
\begin{proof}
Let $a_j\in\calj$. By Corollary \ref{F1assumcor} (iii) and 
Lemma \ref{N_iM_ifiber}, we have $a_j\notin\cali\cup\calk$.
Hence $A\cap\Lambda(a_j,D(M))\subset \calj$, and there exist 
$x_i,x_k\in\ZZ_M\setminus A$ such that $(a_j-x_\nu,M)=M/p_\nu$ and $\bbJ^{N_j}_{N_j}[a_\nu]=0$ for $\nu\in \{i,k\}$. 
	
We have $\Phi_{N_j}|A$ if and only if $\Phi_{N_j}| \calj$. Consider the evaluation of $\calj$ on an $N_j$-cuboid with one face containing vertices at $a_j, x_i$, and $x_k$, and the other face in $\Pi(a_k,p_j)$ for some $a_k\in\calk$. In order to balance that cuboid, $\calj\cap\Pi(a_k,p_j)$ must be nonempty. But then $\calj\cap\Pi(a_k,p_j^2)$ is nonempty, contradicting (\ref{M_kfiber}).
\end{proof}

\begin{proof}[Proof of Proposition \ref{emptyset-F1}]
By Lemmas \ref{Phi_N_inotAlemma} and \ref{Phi_M_jnotA}, we have $\Phi_{N_j}\Phi_{M_j}|B$. It follows from Corollary \ref{doublediv}, with $c_0=1$ since $M/p_j\in\Div(A)$, that $B$ is a union of pairwise disjoint $N_j$-fibers in the $p_i$ and $p_k$ directions. 
	
Let $\nu\in\{i,k\}$, and suppose that 
$\{b_1,\dots,b_{p_\nu}\}$ is a $N_j$-fiber in the $p_\nu$ direction. Since $M/p_\nu\in\Div(A)$, we must have $(b_\mu-b_{\mu'},M)=M/p_jp_\nu$ for all $\mu\neq\mu'$. 
However, this is not possible, since $p_j=\min_\nu p_\nu$ by Lemma \ref{N_iM_ifiber}. This contradiction concludes the proof of the proposition. 
\end{proof}

\begin{proof}[Proof of Proposition \ref{initialstatement}]
	
The proof is divided into several steps. In each of the following claims, the assumptions of the proposition are assumed to hold. We will also assume, by contradiction, that (\ref{Phi_N_inotA}) does not hold, so that $\Phi_{N_i}|A$. By Lemma \ref{fibercyc3}, this implies that
$$
%\begin{equation}\label{fibercyc-Ni}
\Phi_{N_i}|\cali.
%\end{equation}
$$
	
\medskip
\noindent
{\bf Claim 1.} {\it Let $a_j\in \calj$. Then:
\begin{itemize}
%\item[(i)] Every $a_i\in\cali\cap\Pi(a_j,p_j)$ must belong to an $N_i$-fiber in the $p_i$ or $p_j$ direction. 
\item[(i)] There must exist $a_i\in \cali\cap\Pi(a_j,p_j)$ such that
\begin{equation}\label{longaline}
\ell_i(a_i)\subset A, \hbox{ and } a_i\not\in\calj.
\end{equation}
\item[(ii)] Consequently, $p_i=\min_\nu p_\nu$. 
\end{itemize}
} 
\begin{proof}
Let $a_i\in (\cali\setminus \calj) \cap\Pi(a_j,p_j)$, as provided by 
Corollary \ref{F1assumcor} (iv).
By Proposition \ref{no unfibered grids}, the grid $\Lambda(a_i,D(N_i))$ must be $N_i$-fibered in some direction. However, fibering in the $p_\nu$ direction for some $\nu\in\{j,k\}$ would imply that $a_i*F_i*F_\nu\subset A$; for $\nu=j$, this is impossible by the assumption that $a_i\in \cali\setminus \calj$, and for $\nu=k$, this is prohibited by Corollary \ref{F1assumcor} (iii).
		
It follows that any $a_i\in(\cali\setminus \calj)\cap\Pi(a_j,p_j)$ must belong to an $N_i$-fiber in the $p_i$ direction, so that (\ref{longaline}) holds. This, moreover, implies that $|A\cap\ell_i(a_i)|=p_i^2$.
If $p_i>p_j$ or $p_i>p_k$, this contradicts Lemma \ref{planebound}.
\end{proof}

Let $a_j\in\calj$, and let $a_i\in \cali\cap\Pi(a_j,p_j)$ satisfy (\ref{longaline}).  Replacing $a_j$ by another element of $a_j*F_j$, and $a_i$ by another element of $\ell_i(a_i)$, if necessary, we may assume without loss of generality that $M_k|a_i-a_j$. 
	
\medskip
\noindent
{\bf Claim 2.} {\it With $a_i$ and $a_k$ as above, we have 
\begin{equation}\label{p_j^2planedif}
(a_i-a_j,M)=M/p_k^2.
\end{equation}
}
\begin{proof}
Assume for contradiction that  (\ref{p_j^2planedif}) fails, so that $(a_i-a_j,M)=M/p_k$. Then $a_i\in\Lambda:=\Lambda(a_j,D(M))$. The set $A\cap\Lambda$ contains no elements of $\calk$ by Corollary \ref{F1assumcor} (iii), hence it cannot be $M$-fibered in the $p_k$ direction. Furthermore, $A\cap \Lambda$ cannot be $M$-fibered in the $p_j$ direction, since $a_i\not\in\calj$. Thus it must be $M$-fibered in the $p_i$ direction, so that
\begin{equation}\label{planebound-met}
a_j*F_i*F_j\subset A,
\end{equation}
and $|A\cap\Pi(a_j,p_k)|>p_ip_j$ since $a_i\not\in a_j*F_i*F_j$.
It follows by Corollary \ref{planegrid} that $\Phi_{p_k^2}|A$ and $A\subset\Pi(a_j,p_k)$. However, this implies that $\calk\subset \Pi(a_j,p_k)$, hence $\calk\cap \Pi(a_j,p_k^2)\neq\emptyset$. Thus $\Pi(a_j,p_k^2)$ must contain the $p_ip_j$ points in (\ref{planebound-met}), and at least one additional point of $\calk$ which, by Corollary \ref{F1assumcor} (iii), does not belong to $a_j*F_i*F_j$. This violates Lemma \ref{planebound}.
\end{proof}
	
Claim 2, together with (\ref{nonempty}), yields the following (partial) list of divisors in $A$:
\begin{equation}\label{list}
M/p_i,M/p_j,M/p_k, M/p_i^2,M/p_k^2,M/p_ip_k^2,M/p_i^2p_k^2\in \Div(A)
\end{equation}

\medskip
\noindent
{\bf Claim 3.} {\it $B$ is $N_k$-fibered in the $p_j$ direction, with $\bbB_{M/p_jp_k}[b]=\phi(p_j)$ for all $b\in B$.}
\begin{proof}
We claim that 
\begin{equation}\label{M/pipknotA}
M/p_ip_k\notin \Div(A)
\end{equation}
Indeed, if (\ref{M/pipknotA}) were not true, then this together with (\ref{list}) would imply
$$
|B\cap \Pi(b,p_j^2)|\leq p_i \text{ for all } b\in B.
$$
By Corollary \ref{F1assumcor} (v),
$$
|B|\leq p_ip_j^2<p_ip_jp_k,
$$
a contradiction.  
		
Next, let $a_k\in \calk$. By (\ref{lastM_kcase}), $a_k$ satisfies (\ref{M_kfiberp_i}), and due to (\ref{M/pipknotA}) we must have
\begin{equation}\label{Kplanestruc}
\bbK_{M/p_ip_k^2}[a_k]=p_k\phi(p_i).
\end{equation}
		
Fix $b\in B$ and $x\in \ZZ_M$ with $(x-a_k,M)=M/p_j$, and consider the saturating set $A_{x,b}$. Recall from Corollary \ref{F1assumcor} (ii) that $A\cap \Pi(a_k,p_j)\subset \Pi(a_k,p_j^2)$, hence by (\ref{bispan}) $A_{x,b}\subset \Pi(a_k,p_j^2)$. Together with (\ref{M/pipknotA}) and (\ref{Kplanestruc}), the latter implies 
\begin{align*}
1&=\langle\bbA[x],\bbB[b]\rangle\\
&=\frac{1}{\phi(p_jp_k)}\bbA_{M/p_jp_k}[x|\Pi(a_k,p_j^2)]\bbB_{M/p_jp_k}[b]+\frac{1}{\phi(p_ip_jp_k^2)}\bbA_{M/p_ip_jp_k^2}[x|\Pi(a_k,p_j^2)]\bbB_{M/p_ip_jp_k^2}[b]\\
& =\frac{1}{\phi(p_j)}\bbB_{M/p_jp_k}[b]+\frac{1}{\phi(p_jp_k)}\bbB_{M/p_ip_jp_k^2}[b]\\
&=\frac{1}{\phi(p_j)}\sum_{(y-b,M)=M/p_j}\Big(\bbB_{M/p_k}[y]+\frac{1}{\phi(p_k)}\bbB_{M/p_ip_k^2}[y]\Big)
\end{align*}

Since $M/p_ip_k^2\in\Div(A)$ by (\ref{list}), Lemma \ref{triangles} implies that
\begin{equation}\label{notwodiv}
\bbB_{M/p_k}[y]\cdot\bbB_{M/p_ip_k^2}[y]=0, \text{ for all } y \in\ZZ_M.
\end{equation}
On the other hand, again by (\ref{list}), we have 
\begin{equation}\label{M/p_ip_k^2Bbound}
\bbB_{M/p_ip_k^2}[y]\leq \phi(p_i) < \phi(p_k),
\end{equation}
where at the last step we used Claim 1 (ii). 
Given (\ref{notwodiv}) and (\ref{M/p_ip_k^2Bbound}), the only way to saturate
$\langle\bbA[x],\bbB[b]\rangle$ is to have $\bbB_{M/p_jp_k}[b]=\phi(p_j)$ as claimed. 
\end{proof}

At this point, it may be useful to pause and consider the geometric meaning of what we have proved so far. Suppose that $0\in A\cap B$, with $0\in\calj$. By \cite[Lemma 8.4]{LaLo3}, the grid $\Lambda:= \Lambda(0,D(M))$ must be tiled by fibers in at most 2 directions. However, the subgrid $F_j*F_k$ cannot be tiled solely by fibers of $\calj$, since by Claim 3 each such fiber would tile $p_j$ fibers of $F_j*F_k$ and $p_k$ is not divisible by $p_j$. Therefore $\Sigma_A(\Lambda)\subset\cali\cup\calj$, with $(\cali\setminus\calj)\cap\Sigma_A(\Lambda)\neq\emptyset$. We will see that this forces $B$ to have very strong fibering properties, which eventually become incompatible with each other.

We now return to the proof of the proposition.

\medskip
\noindent
{\bf Claim 4.} {\it We have 
$$
\Phi_{N_j}\nmid A .
$$
}
\begin{proof}
Assume by contradiction that $\Phi_{N_j}|A$. 
By Lemma \ref{fibercyc3}, $\Phi_{N_j}|\calj$. We have $p_j>p_i$ by Claim 1, so that by Lemma \ref{fiberfibering-new} (ii), $\calj$ must be $N_j$-fibered on each $D(N_j)$-grid in one of the $p_i$ and $p_k$ directions. However, Claim 3 implies that $M/p_jp_k\in\Div(B)$. Hence $\calj$ is $N_j$-fibered in the $p_i$ direction, and
$\calj\subset \cali$. We prove that this is not allowed. 

Let $a_j\in \calj$, so that $a_j*F_i*F_j\subset A$. By Lemma \ref{planebound}, $A\cap\Pi(a_j,p_k^2)=a_j*F_i*F_j$. If we had $\Phi_{p_k^2}| A$, then $A$ would be contained in $\Pi(a_j,p_k)$; this, however, contradicts the conclusion of Claim 2. It follows that $\Phi_{p_k}|A$, $\Phi_{p_k^2}|B$, and 
\begin{equation}\label{Ap_kplane}
A\cap\Pi(a_j,p_k)=a_j*F_i*F_j.
\end{equation}

Let $x\in \ZZ_M$ with $(x-a_j,M)=M/p_k$ and consider the saturating set $A_{x,b_0}$, where $b_0\in B$ is arbitrary. By (\ref{Ap_kplane}), for every $y\in \ZZ_M$ with $(y-b_0,M)=M/p_k$ we have 
$$
1=\bbB_M[y]+\bbB_{M/p_i}[y]+\bbB_{M/p_j}[y]+\bbB_{M/p_ip_j}[y].
$$
This implies that $B\cap\Lambda(b_0,D(M))=\{b_0,b_1,b_2,\ldots, b_{p_k-1}\}$, where 
$p_k\parallel b_\mu-b_{\mu'}$ for all $\mu,\mu'\in \{0,1,\dots,p_k-1\}, \mu\neq\mu'$. On the other hand, taking Claim 3 into account, this $p_k$-tuple of elements $b_0,\ldots, b_{p_k-1}$ can be grouped into pairwise disjoint $N_k$-fibers in the $p_j$ direction, each of cardinality $p_j$. This implies $p_k$ is divisible by $p_j$, which is not allowed. 
\end{proof}

\medskip
\noindent
{\bf Claim 5.} {\it Given $a_i\in \cali$ satisfying (\ref{longaline}), we have $\bbI^{M_i}_{M_i/p_\mu}[a_i]=0$ for $\mu\in \{j,k\}$.
}
\begin{proof}
Let $a_i$ satisfy (\ref{longaline}). Let $b,b'\in B$ with $(b-b',M)=M/p_jp_k$, as follows from Claim 3. Let $y,y'\in \ZZ_M\setminus B$ with $(b-y,M)=(b'-y',M)=M/p_k$, $(b-y',M)=(b'-y,M)=M/p_j$, and consider the saturating set $B_{y,a_i}$. Then by (\ref{bispan})
$$
B_{y,a_i}\subset\ell_i(b)\cup\ell_i(y)\cup\ell_i(b')\cup\ell_i(y').
$$
If $B_{y,a_i}\cap(\ell_i(b)\cup\ell_i(b'))$ were nonempty, then $\{M/p_i,M/p_i^2\}\cap \Div(B)$ would be nonempty, contradicting (\ref{list}). Thus $B_{y,a_i}\subset(\ell_i(y)\cup \ell_i(y'))$. It follows that $\{M/p^\delta_ip_j,M/p^\delta_ip_k\}\subset \Div(B)$ for some $\delta\in\{1,2\}$. By (\ref{longaline}), in both cases we get $\bbI^{M_i}_{M_i/p_\mu}[a_i]=0$ for $\mu\in \{j,k\}$.	 
\end{proof}

The next three claims are identical to Claims 7, 8, and 9 in the proof of Proposition 9.14 of \cite{LaLo2}. The proofs are exactly the same as in \cite{LaLo2}, and are therefore omitted. The only difference is that, in the proof of Claim 6, we have to start by choosing $a_i$ satisfying (\ref{longaline}). 
	
\medskip
\noindent
{\bf Claim 6.} {\it We have $\Phi_{M_i/p_j}\Phi_{M_i/p_k}|B$.}

\medskip
\noindent
{\bf Claim 7.} {\it For each $\mu\in\{j,k\}$, $B$ is $M_i$-fibered in the $p_\mu$ direction, with
$\bbB^{M_i}_{M_i/p_\mu}[b]=\phi(p_\mu)$ for each $b\in B$.
}
	
\medskip
\noindent
{\bf Claim 8.} {\it $B$ is $N_j$-fibered in the $p_i$ direction, with
$\bbB^{M}_{M/p_ip_j}[b]=\phi(p_i)$ for each $b\in B$.
}
	
We are now in a position to finish the proof of Proposition \ref{initialstatement}. Fix $b\in B$. By Claims 1 and 7, $\bbB^{M_i}_{M_i}[b]=1$ and $\bbB^{M_i}_{M_i/p_j}[b]=\phi(p_j)$. Thus
$$
|B\cap \Lambda(b,M_i/p_j)|=p_j \hbox{ for all }b\in B.
$$
On the other hand, we can write $\Lambda(b,M_i/p_j)$ as a union of pairwise disjoint grids $\Lambda(y_\nu, M/p_ip_j)$ for an appropriate choice of $y_\nu$. By Claim 8, each such grid contains either 0 or $p_i$ elements of $B$. But this implies that $p_i$ divides $p_j$, a contradiction.
%
%
%Let $b_1,\ldots, b_{p_j-1}\in B$ be all the elements with  
%$$
%(b-b_\mu,M_i)=(b_\mu-b_{\mu'},M_i)=M_i/p_j \text{ for all } \mu,\mu'\in\{1,\ldots, p_j-1\}, \mu\neq\mu'.
%$$
%By Claim 8 
%$$
%\bbB_{M/p_ip_j}[b]=\bbB_{M/p_ip_j}[b_\mu]=\phi(p_i) \text{ for all } \mu\in\{1,\ldots, p_j-1\}.
%$$ 
%However, the two properties above imply $p_i$ must divide $p_j$, a contradiction. 
\end{proof}

%%%%%%%%%%%%%%%%%%%%%%%%%%%%%%%%%%%%%%%%%%%%%%%%%%%%

%%%%%%%%%%%%%%%%%%%%%%%%%%%%%%%%%%%%%%%%%%%%%%%%

\subsection{Proof of Theorem \ref{fibered-mainthm} (II\,a), 
Case (F2)}\label{caseF2}

%%%%%%%%%%%%%%%%%%%%%%%%%%%%%%%%%%%%%%%%%%%%%%%%%

We recall the assumptions and our main result in this case.

\medskip\noindent
\textbf{Assumption (F2):} We have $A\oplus B=\ZZ_M$, where $M=p_i^{2}p_j^{2}p_k^{2}$. Furthermore, $|A|=|B|=p_ip_jp_k$, $\Phi_M|A$, $A$ is fibered on $D(M)$-grids, (\ref{emptyint}) holds, and $\Phi_{M_\nu}\nmid A$ for all $\nu\in \{i,j,k\}$.

\begin{proposition}\label{emptyset-F2}
	Assume that (F2) holds. Then one of the sets $\cali,\calj,\calk$ is empty.
\end{proposition}

Assume, for contradiction, that (\ref{nonempty}) holds. Without loss of generality, we may also assume that
\begin{equation}\label{order}
p_i<p_j<p_k.
\end{equation}

\begin{lemma}\label{N_JN_KnotA}
Assume that (F2), (\ref{nonempty}), and (\ref{order}) hold. Then neither $\Phi_{N_j}$ nor $\Phi_{N_k}$ divides $A$.
\end{lemma}
\begin{proof}
By (\ref{order}) and Lemma \ref{fiberfibering-new} (ii), if $\Phi_{N_j}|A$, then $\calj\subset \cali\cup\calk$. However, that would imply $\Phi_{M_j}|A$, contradicting (F2). The same argument holds with $j$ and $k$ interchanged.
\end{proof}

\begin{corollary}\label{BN_jfibering}
Assume that (F2), (\ref{nonempty}), and (\ref{order}) hold.
Then $\Phi_{N_j}\Phi_{M_j}|B$. Consequently, $B$ is $N_j$-fibered in the $p_i$ direction, with 
\begin{equation}\label{BN_jM_jfibering}
\bbB_{M/p_ip_j}[b]=\phi(p_i) \text{ for all } b\in B.
\end{equation}
Moreover, we have $\Phi_{p_i}|A$ and $\Phi_{p_i^2}|B$.
\end{corollary}

\begin{proof}
The first part follows from (F2) and Lemma \ref{N_JN_KnotA}. This, in turn, implies by Corollary \ref{doublediv} (with $c_0=1$, since $N_j\in\Div(A)$) that $B$ is a union of pairwise disjoint $N_j$-fibers in the $p_i$ and $p_k$ directions. We also have $p_j<p_k$ by  (\ref{order}), so that having an $N_j$-fiber in the $p_k$ direction in $B$ would imply that $N_j\in\Div(B)$, a contradiction. This proves the fibering claim. Finally, (\ref{BN_jM_jfibering}) and Lemma \ref{cyclo-grid} imply that $\Phi_{p_i^2}|B$ as claimed.
\end{proof}

\begin{proposition}\label{F2Phi_N_inotA}
Assume that (F2), (\ref{nonempty}), and (\ref{order}) hold.
Then $\Phi_{N_i}\nmid A$. 
\end{proposition}

As in Section \ref{caseF2}, we first finish the proof of Proposition \ref{emptyset-F2}, assuming Proposition \ref{F2Phi_N_inotA}.

\begin{proof}[Proof of Proposition \ref{emptyset-F2}]
By (F2) and Proposition \ref{F2Phi_N_inotA}, we have $\Phi_{N_i}\Phi_{M_i}|B$. 
Applying Corollary \ref{doublediv} with $c_0=1$ once more, we see that $B$ is a union of pairwise disjoint $N_i$-fibers in the $p_j$ and $p_k$ directions, so that for every $b\in B$ we must have
$$
\bbB_{M/p_ip_\nu}[b]=\phi(p_\nu) \hbox{ for some }\nu=\nu(b)\in\{j,k\}.
$$ 
By (\ref{order}), this implies $M/p_\nu\in \Div(A)\cap \Div(B)$ for at least one $\nu\in\{j,k\}$, which is a contradiction. 
\end{proof}

\begin{proof}[Proof of Proposition \ref{F2Phi_N_inotA}]
Again, we split the proof into several steps. In each of the following claims, the assumptions of the proposition are assumed to hold. We will also assume,  by contradiction, that $\Phi_{N_i}|A$. By Lemma \ref{fibercyc3}, this implies that $\Phi_{N_i}|\cali$. 

\medskip
\noindent
{\bf Claim 1.} {\it $\cali$ is a union of pairwise disjoint $N_i$-fibers in the $p_i$ and $p_k$ directions. Moreover, $\bbI^{N_i}_{N_i/p_i}[a_i]=\phi(p_i^2)$ must hold for at least one element $a_i\in\cali$, and $M/p_i^2\in \Div(A)$.}
\begin{proof}
By Proposition \ref{no unfibered grids}, $\cali$ mod $N_i$ must be $N_i$-fibered on each $D(N_i)$-grid. 
By (\ref{BN_jM_jfibering}), we have $M/p_ip_j\notin \Div(A)$, therefore $N_i/p_j\notin \Div_{N_i}(\cali)$ and $\cali$ cannot be $N_i$-fibered in the $p_j$ direction on any $D(N_i)$-grid. This implies the first part of the lemma.

If $\cali$ were $N_i$-fibered in the $p_k$ direction, this would imply $\cali\subset\calk$. Then, however, we would have $\Phi_{M_i}|A$, contradicting (F2). Hence at least one $a_i\in\cali$ must belong to an $N_i$-fiber in the $p_i$ direction, as claimed.
\end{proof}

\medskip
\noindent
{\bf Claim 2.} {\it $B$ is $M_i$-fibered in the $p_k$ direction, with 
$\bbB^{M_i}_{M_i/p_k}[b]=\phi(p_k)$ for each $b\in B$. Consequently, 
\begin{equation}\label{f2-Bdivisors}
M/p_ip_k \in \Div(B).
\end{equation}
Moreover:
\begin{itemize}
\item[(i)] There is no element $b\in B$ satisfying $\bbB^{M_i}_{M_i/p_j}[b]=\phi(p_j)$. 
\item[(ii)] $\cali$ is $N_i$-fibered in the $p_i$ direction.
\end{itemize}
}
\begin{proof}
We have $\Phi_{M_i}|B$ by (F2). This means that $B$ satisfies the assumptions of Lemma \ref{M_kfibering}, with $\calk$ replaced by $B$, the $p_k$ and $p_i$ directions interchanged, and with multiplicity 1 (instead of $p_k$) by Claim 1. It follows that $B$ is a union of pairwise disjoint $M_i$-fibers in the $p_j$ and $p_k$ directions.

We now argue as in the proof of Proposition \ref{initialstatement}.
Assume for contradiction that (i) fails, so that some $b\in B$ belongs to an $M_i$-fiber in the $p_j$ direction, with $\bbB^{M_i}_{M_i/p_j}[b]=\phi(p_j)$. 
This means that
$$
|B\cap \Lambda(b,M_i/p_j)|=p_j \hbox{ for all }b\in B.
$$
On the other hand, we can write $\Lambda(b,M_i/p_j)$ as a union of pairwise disjoint $M/p_ip_j$-grids. By (\ref{BN_jM_jfibering}), each such grid contains either 0 or $p_i$ elements of $B$. But then $p_i|p_j$, which is obviously false. This proves (i), and the fibering claim for $B$ follows.
Part (ii) now follows from (\ref{f2-Bdivisors}) and Claim 1. 
\end{proof}

\medskip
\noindent
{\bf Claim 3.} {\it There exists $b_0\in B$ such that $\bbB_{M/p_jp_k}[b_0]=\phi(p_j)$.} 
\begin{proof}
By (F2) and Lemma \ref{N_JN_KnotA}, we have $\Phi_{N_k}\Phi_{M_k}|B$.
It follows from Corollary \ref{doublediv} with $c_0=1$ that $B$ is a union of pairwise disjoint $N_k$-fibers in the $p_i$ and $p_j$ directions, so that for every $b\in B$ we have $\bbB_{M/p_\mu p_k}[b]=\phi(p_\mu)$ for some $\mu\in \{i,j\}$. However, if the latter was true with $\mu=i$ for all $b\in B$, then Claim 2 and the divisibility argument in its proof would show that $p_i|p_k$. This contradiction proves the claim. 
\end{proof}	

We note that (\ref{BN_jM_jfibering}), (\ref{f2-Bdivisors}), and Claim 3 show that $\{M/p_ip_j,M/p_ip_k,M/p_jp_k\}\subset \Div(B)$. Hence $\cali, \calj$ and $\calk$ are pairwise disjoint.

The next two claims are proved in \cite{LaLo2}, Claim 7 and 8 of Proposition 9.14. The proof is identical. 

\medskip
\noindent
{\bf Claim 4.} {\it $\Phi_{M_i/p_j}\Phi_{M_i/p_k}|B$.}

\medskip
\noindent
{\bf Claim 5.} {\it $B$ is $M_i$-fibered in both of the $p_j$ and $p_k$ directions, so that for all $b\in B$ we have
$$
\frac{1}{\phi(p_j)}\bbB^{M_i}_{M_i/p_j}[b]=\frac{1}{\phi(p_k)}\bbB^{M_i}_{M_i/p_k}[b]=1.
$$
}

Since Claim 5 is in direct contradiction with Claim 2 (i), the proposition follows. 
\end{proof}

%%%%%%%%%%%%%%%%%%%%%%%%%%%%%%%%%%%%%%%%%%%%%%%

\subsection{Proof of Theorem \ref{fibered-mainthm}, part (II\,c)}\label{fibered-F3}

%%%%%%%%%%%%%%%%%%%%%%%%%%%%%%%%%%%%%%%%%

In this section, we will work under the following assumption.

\medskip\noindent
\textbf{Assumption (F3):} We have $A\oplus B=\ZZ_M$, where $M=p_i^{2}p_j^{2}p_k^{2}$. Furthermore,  $|A|=|B|=p_ip_jp_k$, $\Phi_M|A$, $A$ is fibered on $D(M)$-grids, $\cali=\emptyset$, (\ref{notallhigh}) holds, and
\begin{equation}\label{symdif}
	\text{the sets } \calj\setminus\calk \text{ and } \calk\setminus\calj \text{ are nonempty}.
\end{equation}

\begin{proposition}\label{F3structure}
Assume (F3). Then the conclusion (II\,c) of Theorem \ref{fibered-mainthm} holds.
\end{proposition}

The proof below works regardless of whether $\calj$ and $\calk$ are disjoint or not. If $\calj\cap\calk\neq\emptyset$, then
(since $\cali=\emptyset$) any element $a\in\calj\cap\calk$ must satisfy the conditions of Lemma \ref{fiberedstructure} (i), so that 
$$
%\begin{equation}\label{jkintersection}
A\cap\Pi(a,p_i^{2})=a*F_j*F_k.
%\end{equation}
$$
It follows that the set $\calj\setminus\calk$ is $M$-fibered in the $p_j$ direction, and 
$\calk\setminus\calj$ is $M$-fibered in the $p_k$ direction.

We begin with the case when at least one of $\Phi_{M_j}$ and $\Phi_{M_k}$ divides $A$.

\begin{lemma}[\cite{LaLo2}, Lemma 9.33]\label{Kp_i^2plane} 
Assume (F3), and that $\Phi_{M_k}|A$. Then 
$$
%\begin{equation}\label{Phi_p_j^2|A}
\Phi_{p_j^2}|A.
%\end{equation}
$$
Furthermore, $\calk$ is $M_k$-fibered in the $p_j$ direction, so that for every $a_k\in\calk$ we have
\begin{equation}\label{KM/p_k^2jfiber}
\mathbb{K}^{M_k}_{M_k/p_j}[a_k]=p_k\cdot\phi(p_j).
\end{equation}
and 
$$
%\begin{equation}\label{Kp_i^2p_jplane}
A\cap\Pi(a_k,p_i^2)\subset\Lambda(a_k,p_i^2p_j).
%\end{equation}
$$
The same holds with $p_k$ and $p_j$ interchanged.
\end{lemma}

\begin{lemma}[\cite{LaLo2}, Lemma 9.34]\label{splitting}
Assume (F3).
The following holds true:
\begin{enumerate}
\item [(i)] If $\Phi_{p_i^2}\Phi_{M_k}|A$, then $A$ is contained in a subgroup.
\item [(ii)] If $\Phi_{p_i}|A$, then $|A\cap\Pi(a,p_i)|=p_jp_k$ for all $a\in A$.
Moreover, for every $a\in A$ we have either $A\cap \Pi(a,p_i)\subset \calj$ or $A\cap \Pi(a,p_i)\subset \calk$.  \end{enumerate}
\end{lemma}

\begin{corollary}\label{slab1}
Assume that (F3) holds, and that
\begin{equation}\label{Phi_{p_i}|Aass}
\Phi_{p_i}\Phi_{M_k}|A.
\end{equation}
Then
\begin{itemize}
\item[(i)] $|A\cap \Pi(a,p_i^2)|=p_jp_k$ for every $a\in \calk$.
\item[(ii)] Assume, in addition, that $\Phi_{M_j}|A$. Then the conditions of Theorem \ref{subtile} (the slab reduction) are satisfied in the $p_i$ direction,
after interchanging $A$ and $B$. 
\end{itemize}
\end{corollary}

\begin{proof}
Let $a\in\calk$, then by (\ref{KM/p_k^2jfiber}) we have
$|A\cap \Pi(a,p_i^2)|\geq p_jp_k$, and by Lemma \ref{planebound} this must hold with equality. This implies (i). 

For (ii), we have $\Phi_{p_i^2}|B$ by (\ref{Phi_{p_i}|Aass}). Moreover, in this case (i) also holds for $a\in\calj$, hence for all $a\in A$. This, however, means that (\ref{maxplanebound}) holds with $A$ and $B$ interchanged. By Corollary \ref{slabcor} (ii), the slab reduction conditions are satisfied as claimed.
\end{proof}

\begin{lemma}\label{Phi_N_j|B}
Assume (F3) and (\ref{Phi_{p_i}|Aass}).
Then $\Phi_{N_j}\nmid A$. 
\end{lemma}
\begin{proof}
If $\Phi_{N_j}$ divides $A$, it also divides $\calj$. By Proposition \ref{no unfibered grids}, $\calj$ must be $N_j$-fibered on all $D(N_j)$-grids.
%\begin{equation}\label{fiberatti}
%\hbox{if }\Phi_{N_j}|A \hbox{, then }\calj \hbox{ must be }N_j\hbox{-fibered on }D(N_j) \hbox{ grids.}
%\end{equation}
The rest of the argument appears in \cite{LaLo2}, Lemma 9.35. 
\end{proof}

Assume next that $\Phi_{M_j}\nmid A$. Then, by Lemma \ref{Phi_N_j|B},
\begin{equation}\label{N_jM_jB}
\Phi_{N_j}\Phi_{M_j}|B.
\end{equation}

\begin{lemma}\label{Bjfiver}
Assume (F3), (\ref{Phi_{p_i}|Aass}), and (\ref{N_jM_jB}). Then $B$ is $N_j$-fibered in the $p_i$ direction, so that for every $b\in B$,
\begin{equation}\label{Bp_ifiber}
	\bbB_M[y]+ \bbB_{M/p_j}[y]=1 \text{ for every } y\in \ZZ_M \text{ with } (b-y,M)=M/p_i,
\end{equation}
\end{lemma}

\begin{proof}
By (\ref{N_jM_jB}) and Corollary \ref{doublediv} with $c_0=1$,  $B$ is a union of pairwise disjoint $N_i$-fibers in the $p_j$ and $p_k$ directions. 
It follows that for every $b\in B$, either (\ref{Bp_ifiber}) holds, or else
\begin{equation}\label{Bp_kfiber}
	\bbB_{M}[y]+\bbB_{M/p_j}[y]=1 \text{ for every } y\in \ZZ_M \text{ with } (b-y,M)=M/p_k.
\end{equation}
Assume, by contradiction, that (\ref{Bp_kfiber}) holds for some $b\in B$. Since $M/p_k\in \Div(A)$, we must have $p_k<p_j$ and $M/p_jp_k\in \Div(B)$. On the other hand, $\calk$ satisfies (\ref{KM/p_k^2jfiber}), so that $p_k<p_j$ implies $M/p_jp_k\in \Div(A)$, contradicting divisor exclusion. 
\end{proof}

\begin{corollary}\label{cor-subtile1}
	Assume (F3), (\ref{Phi_{p_i}|Aass}), and (\ref{N_jM_jB}). Then the conditions of Theorem \ref{subtile} are satisfied in
 the $p_i$ direction,
	after interchanging $A$ and $B$. 
\end{corollary}

\begin{proof}
By Theorem \ref{subtile} (iii), it suffices to prove that for all $a\in A$, $b\in B$, and for all  $y\in \ZZ_M$ with $(y-b,M)=M/p_i$, we have
\begin{equation}\label{sat-splitB}
B_{y,a}\subset \Pi(y,p_i^2).
\end{equation} 
For $a\in\calk$, we have $A\cap\Pi(a,p_i)=A\cap \Pi(a,p_i^2)$ by 
Lemma \ref{splitting} (ii) and Lemma \ref{slab1} (i).
 This clearly implies (\ref{sat-splitB}). 

Assume now that $a\in \calj$. If $y\in B$, (\ref{sat-splitB}) holds trivially. Otherwise, we have $\bbB_{M/p_j}[y]=1$ by (\ref{Bp_ifiber}), so that 
$$
\langle\bbA[a],\bbB[y]\rangle=\frac{1}{\phi(p_j)}\bbA_{M/p_j}[a]\bbB_{M/p_j}[y]=1,
$$
which proves (\ref{sat-splitB}).
\end{proof}

It remains to prove Proposition \ref{F3structure} under the assumption that
\begin{equation}\label{cycassumption}
\Phi_{M_\nu}\nmid A \text{ for } \nu\in \{j,k\}. 
\end{equation}
Without loss of generality, we may also assume that 
\begin{equation}\label{porzadekalfabetyczny}
p_k>p_j.
\end{equation}
Since by (\ref{symdif}) $\calk\not\subset \calj$, it follows from Lemma \ref{fiberfibering-new} (ii) that $\Phi_{N_k}\Phi_{M_k}|B$. As in the proof of Lemma \ref{Bjfiver}, we see that $B$ mod $N_k$ is a union of pairwise disjoint $N_k$-fibers in the $p_i$ and $p_j$ directions, so that every $b\in B$ satisfies at least one of the following:
\begin{equation}\label{BTp_ifiber}
\bbB_{M}[y]+\bbB_{M/p_k}[y]=1 \text{ for every } y\in \ZZ_M \text{ with } (b-y,M)=M/p_i.
\end{equation}
\begin{equation}\label{BTp_jfiber}
\bbB_M[y]+ \bbB_{M/p_k}[y]=1 \text{ for every } y\in \ZZ_M \text{ with } (b-y,M)=M/p_j,
\end{equation}
In particular, 
%\begin{equation}\label{Btopdivisors}
$$
\{M/p_i,M/p_j, M/p_ip_k,M/p_jp_k\} \cap \Div(B)\neq \emptyset.
$$
%\end{equation}

\begin{lemma}\label{BPhi_{N_j}}
Assume (F3), (\ref{cycassumption}), and (\ref{porzadekalfabetyczny}). 
Then:
\begin{itemize}
\item[(i)] $\Phi_{N_j}\nmid A$,
\item[(ii)] $B$ is $N_j$-fibered in the $p_i$ direction, % and $M_k$-fibered in the $p_j$ direction, 
so that (\ref{Bp_ifiber})
holds for all $b\in B$,
\item[(iii)] $\Phi_{p_i}|A$.
\end{itemize}
\end{lemma}
\begin{proof}
We start with (i). Assume for contradiction that
$\Phi_{N_j}|A$. 
By Proposition \ref{no unfibered grids}, $\calj$ must be $N_j$-fibered on each $D(N_j)$-grid, so that
$\calj\setminus\calk$ must be $N_j$-fibered in the $p_j$ direction.
%
%Indeed, if $M$ is odd, this follows from Lemma \ref{Kunfibered} applied to $p_j$, (\ref{Btopdivisors}), and Lemma \ref{fiberfibering} (i). If $M$ is even but $p_j>2$, this follows from Corollary \ref{subset} %Lemma \ref{evensubset} 
%and, again, Lemma \ref{fiberfibering} (i),
%
%
%Assume now that $p_j=2$. Then the conclusion of Lemma \ref{evenunfiberedN_igrid} (i) cannot hold for $p_j$, 
%since (\ref{Btopdivisors}) contradicts (\ref{diagboxdiv1}). Therefore 
%the conclusion of Lemma \ref{evenunfiberedN_igrid} (ii) holds, so that $\calj\setminus \calk$ modulo $N_j$ contains diagonal boxes and $\{M/p_i,M/p_j, M/p_jp_k\}\subset \Div(A)$. In addition, it follows from (\ref{Btopdivisors}) that (\ref{BTp_jfiber}) cannot hold for any $b\in B$. Hence $B$ is $N_k$-fibered in the $p_i$ direction, with 
%$$\bbB^{N_k}_{N_k/p_i}[b]=\bbB_{M/p_ip_k}[b]=\phi(p_i)\hbox{ for all }b\in B.
%$$
%This means that $B$ satisfies the 
%assumptions of Lemma \ref{cyclo-grid}, with $m=M/p_ip_k$ and $s=p_i^2$.
%Hence $\Phi_{p_i^2}|B$. On the other hand, Lemma \ref{evenunfiberedN_igrid} implies $\Phi_{p_i^2}|A$, which is a contradiction. By Lemma \ref{fiberfibering} (i), the set $\calj\setminus\calk$ must be $N_j$-fibered in the $p_j$ direction. 
%

Let $a_j\in \calj\setminus\calk$.
We now consider two cases.
 
\begin{itemize} 
\item Suppose that $\Phi_{p_i}|A$.
By Lemma \ref{splitting} (ii), we have $A\cap\Pi(a_j,p_i)\subset\calj$ and $|A\cap\Pi(a_j,p_i)|=p_jp_k$. If there was an element $a\in (\calj\cap\calk)\cap \Pi(a_j,p_i)$, it would follow that $A\cap\Pi(a_j,p_i)=a*F_j*F_k$, and in particular $a_j\in a*F_j*F_k$, contradicting the choice of $a_j$. Therefore  $A\cap\Pi(a_j,p_i)\subset\calj\setminus\calk$.
But then the fibering of $\calj\setminus\calk$ implies that $p_jp_k$ is divisible by $p_j^2$, a contradiction.

\item Assume now that $\Phi_{p_i^{2}}|A$. Let $A'$ be a translate of $A$ such that $a_j\in A'_{p_i}$. By the cyclotomic divisibility assumption, we have $|A'_{p_i}|=p_jp_k$. On the other hand, by the fibering properties of $A$, 
$$
p_jp_k=|A'_{p_i}|=c_jp_j^2+c_kp_k,\ \ c_j>0.
$$
Thus $c_j=p_kc'_j$ and $p_j=c'_jp_j^2+c_k$ with $c'_j>0$, a contradiction.
\end{itemize}
Therefore $\Phi_{N_j}\nmid A$, proving (i). 

Next, we prove (ii). By (i) together with (\ref{cycassumption}), we have $\Phi_{N_j}\Phi_{M_j}|B$. As in the proof of Lemma \ref{Bjfiver}, $B$ mod $N_j$ is a union of pairwise disjoint $N_j$-fibers in the $p_i$ and $p_k$ directions, so that
every element of $B$ satisfies at least one of (\ref{Bp_ifiber}) and (\ref{Bp_kfiber}). However, by (\ref{porzadekalfabetyczny}), we must in fact we have (\ref{Bp_ifiber}) for all $b\in B$, otherwise $M/p_k\in \Div(A)\cap \Div(B)$ which is a contradiction. This proves (ii).

Finally, the $N_j$-fibering in (ii) implies that
the assumptions of Lemma \ref{cyclo-grid} hold for $B$, with $m=M/p_ip_j$ and $s=p_i^2$. It follows that $\Phi_{p_i^{2}}|B$, and therefore $\Phi_{p_i}|A$.
\end{proof}

\begin{lemma}\label{noM_nusubtile}
Assume (F3), (\ref{cycassumption}), and (\ref{porzadekalfabetyczny}). 
Then the conditions of Theorem \ref{subtile} (the slab reduction) are satisfied in the $p_i$ direction,
after interchanging $A$ and $B$. 
\end{lemma}
\begin{proof}
We verify that the condition (ii) of Theorem \ref{subtile} holds. That is, given $r\in R$, we show that every $M$-fiber in the $p_i$ direction splits with parity $(rA,B)$. Since the only property of the set $A$ that our proof uses is the fact that every element belongs to an $M$-fiber in either the $p_j$ or the $p_k$ direction, and this property is preserved under the mapping $A\mapsto rA$, it suffices to consider the case $r=1$. 

Let $x\in \ZZ_M$, and let $a\in A, b\in B$ satisfy $x=a+b$. If $a\in \calj$, then $x*F_i*F_j$ is tiled by $a*F_j\subset A$ and the $N_j$-fiber in the $p_i$ direction in $B$ containing $b$, provided by (\ref{Bp_ifiber}).

Suppose now that $a\in \calk$. If $b$ belongs to an $N_k$-fiber in the $p_i$ direction in $B$ (as in (\ref{BTp_ifiber})), then $x*F_i*F_k$ is tiled by $a*F_k\subset A$ and that fiber
(this is the same argument as for $a\in\calj$, with $j$ and $k$ interchanged). If on the other hand $b$ satisfies (\ref{BTp_jfiber}), with each element $b'$ of its $N_k$-fiber in the $p_j$ direction belonging to a $N_j$-fiber in the $p_i$ direction, then $\Lambda(x,D(M))$ is tiled by $a*F_k\subset A$ and all $p_ip_j$ elements $b''\in B$ satisfying $(b-b'',M)\in \{M, M/p_i,M/p_ip_j,M/p_jp_k\}$. In both cases, $x*F_i$ splits with parity $(A,B)$, as required.
\end{proof}
This concludes the proof of Proposition \ref{F3structure}.

%%%%%%%%%%%%%%%%%%%%%%%%%%%%%%%%%%%%%%%%%%%%%%

\section{Acknowledgements}

The first author is supported by NSERC Discovery Grant 22R80520. The second author is supported by ISF Grant 607/21.

%%%%%%%%%%%%%%%%%%%%%%%%%%%%%%%%%%%%%%%%%%%%%%

%%%%%%%%%%%%%%%%%%%%%%%%%%%%%%%%%%%%%%%%%%%%%

\bibliographystyle{amsplain}

%\bigskip
%\fontsize{6}{8}\selectfont\blindtext

\noindent{{\sc {\L}aba:} Department of Mathematics, University of British Columbia, Vancouver,
B.C. V6T 1Z2, Canada}, 
{\it ilaba@math.ubc.ca}

\noindent{ORCID ID: 0000-0002-3810-6121}

\medskip

\noindent{{\sc Londner:}  Department of Mathematics, Faculty of Mathematics and Computer Science, Weizmann Institute of Science, Rehovot 7610001, Israel},
{\it itay.londner@weizmann.ac.il}

\noindent{ORCID ID: 0000-0003-3337-9427}

\end{document}